\documentclass[11pt]{article}
\usepackage{glaudo_en_pkg}
\usepackage[OT2,OT1]{fontenc}
\usepackage{forloop}

\DeclareFontFamily{U}{wncy}{}
\DeclareFontShape{U}{wncy}{m}{n}{<->wncyr10}{}
\DeclareSymbolFont{mcy}{U}{wncy}{m}{n}
\DeclareMathSymbol{\Sha}{\mathord}{mcy}{"58} 
\DeclareMathSymbol{\B}{\mathord}{mcy}{"42} 

\DeclareMathOperator{\Br}{Br}
\DeclareMathOperator{\Fun}{Fun}

\DeclareMathOperator{\Stab}{\mathfrak{Stab}}
\DeclareMathOperator{\Ker}{Ker}

\DeclareMathOperator{\Tot}{Tot}
\DeclareMathOperator{\TORS}{\text{TORS}}
\DeclareMathOperator{\Coker}{Coker}
\DeclareMathOperator{\inv}{inv}
\DeclareMathOperator{\ab}{ab}

\DeclareMathOperator{\X}{\chi}
\DeclareMathOperator{\Ab}{\text{Ab}}
\DeclareMathOperator{\UPic}{UPic}
\DeclareMathOperator{\Mor}{Mor}

\renewcommand{\H}{\mathbb{H}}
\newcommand{\G}{\mathbb{G}}
\newcommand{\I}{\mathbb{I}}
\newcommand{\act}[2]{\leftidx{^{#1}}{{#2}}{}}

\usepackage{color}

\definecolor{orange}{rgb}{1,0.5,0}
\definecolor{random}{rgb}{0.7,0,0}

\title{The étale Brauer-Manin obstruction to strong approximation on homogeneous spaces}

\author{Julian L. Demeio}

\begin{document}
	
	\maketitle

	\begin{abstract}
		It is known that, under a necessary non-compactness assumption, the Brauer-Manin obstruction is the only one to strong approximation on homogeneous spaces $X$ under a linear group $G$ (or under a connected algebraic group, under assumption of finiteness of a suitable Tate-Shafarevich group), provided that the geometric stabilizers of $X$ are connected. In this work we prove, under similar assumptions, that the étale-Brauer-Manin obstruction to strong approximation is the only one for homogeneous spaces with arbitrary stabilisers. We also deal with some related questions, concerning strong approximation outside a finite set of valuations. Finally, we prove a compatibility result, suggested to be true by work of Cyril Demarche, between the Brauer-Manin obstruction pairing on quotients $G/H$, where $G$ and $H$ are connected algebraic groups and $H$ is linear, and certain abelianization morphisms associated with these spaces.
	\end{abstract}
	
\section{Introduction}
Let $X$ be an affine variety, described by a system of polynomial equations with integer coefficients. It is a classical question to determine whether $X$ satisfies the \textit{strong approximation} property for integral points, that is, whether the natural reduction map $X(\mathbb{Z}) \to \prod_{i=1}^m X(\mathbb{Z}/p_i^{k_i}\mathbb{Z})$ is surjective for all finite sets of prime powers $\{ p_i^{k_i} \}_{i=1}^m$. Equivalently, strong approximation holds if all local congruence solutions for the equations defining $X$ can be lifted to integral solutions of those same equations. For example, when $X$ is the affine line $\mathbb{A}^1_{\mathbb{Z}}$, strong approximation is satisfied thanks to the Chinese remainder theorem.

In modern terms, the question of strong approximation for a variety $X$, defined over a number field $K$, concerns the density of the $K$-rational points of $X$ in the $S$-adelic points $X(\mathbb{A}_K^S)$ (where the $S$-components are removed, $S$ being some finite set of places); also of interest are \textit{obstructions} that explain the failure of strong approximation for certain classes of varieties.

In this article we consider the case of $X$ being a homogeneous space under an algebraic group. The first result concerning this class of varieties is a theorem of Kneser (see e.g. \cite[Theorem 7.12]{PlatonovRapunchkin}), who proved that strong approximation holds for a simply connected simple linear algebraic group $G$ (such as $K$-forms of $\operatorname{SL}_n$), provided that $S$ contains at least one valuation $v$ for which $X(K_v)$ is non-compact.
The theorem of Kneser has then been generalized several times, for example by Harari \cite{Harari}, Colliot-Thélène and Xu \cite{CTXu}, and, most recently, by Borovoi and Demarche \cite{borovoi}, whose results encompass all the previous ones. In \cite{borovoi}, the authors study the case of a homogeneous space under a connected (not necessarily linear) group $G$ with connected geometric stabilizers, and prove that, under some technical assumptions, the Brauer-Manin obstruction to strong approximation is the only one. The precise technical conditions one needs are recalled in Section \ref{Sec:MainTheorem}, but it is worth noting at the outset that they include the finiteness of the Tate-Shafarevich group of the maximal abelian variety quotient of $G$ (a condition which is by now standard in the field).

On the other hand, as already proved in \cite{Demarche_example}, the Brauer-Manin obstruction to strong approximation is in general not the only one when the geometric stabilizers of $X$ are not connected. In particular, Demarche showed that there is a non-trivial étale-Brauer-Manin obstruction to strong approximation (see Section \ref{Sec:pre} for a definition) for some homogeneous spaces with finite stabilizers.
In this paper we build on \cite[Thm 1.4]{borovoi} to prove that, under some technical assumptions, the étale-Brauer-Manin obstruction to strong approximation is the only one for a general homogeneous space $X$ under a connected group $G$ (with not necessarily connected geometric stabilizers). A precise statement is as follows.

\begin{theorem}\label{Thm:Homspaces}
	Let $G$ be a connected algebraic group over a number field $K$. We assume that the Tate-Shafarevich group $\Sha(K,G^{ab})$ is finite. Let $X$ be a left homogeneous space under $G$.  Let $S\supset M_K^{\infty}$ be a finite set of places of $K$. We assume that $G^{sc}(K)$ is dense in $G^{sc}(\A^S_K)$.  Set $S_f:=S\cap M_K^{fin}=S \setminus M_K^{\infty}$.  Then the set $X(\A_K)_{\bullet}^{\acute{e}t,\Br}$ is equal to the closure of the set $G^{scu}(K_{S_f}) \cdot X(K) \subset X(\A_K)_{\bullet}$ for the adelic topology.
\end{theorem}

The main idea of the proof of Theorem \ref{Thm:Homspaces} is to show that, if there is no étale-Brauer-Manin obstruction to the existence of $K$-rational points on $X$, then there is a homogeneous space $Z$ under $G$, with connected geometric stabilizers, and a finite $G$-equivariant morphism $Z \to X$ that makes $Z$ a torsor over $X$ (under some finite group scheme). This allows one to apply the aforementioned result of Borovoi and Demarche to $Z$. To obtain the torsor $Z \to X$ we crucially rely on \cite[Lemma 7.1]{demarche}, a result that first appeared (in the proper case) in \cite{Stoll}.

In Section \ref{Sec:Removing_places} we prove Theorem \ref{Thm:S-et-Obstr}, an analogue of Theorem \ref{Thm:Homspaces} in which one removes an additional finite set of places. In order to do so, we will need a detailed description of the values of elements of $\operatorname{Br}(X_{K_v})$ on $X(K_v)$, which we will obtain by using the abelianization construction of Demarche \cite{Demarche_abelian} (that in turn builds on earlier work of Borovoi \cite{Borovoi_abelianization}). The explicit description that we will need is given in Theorem \ref{Thm:compatibility}, which appears to be new, and possibly also of independent interest. Standard dévissage methods seem not to yield a proof of Theorem \ref{Thm:compatibility}, so our approach relies on an explicit (and not very exciting) Galois cocycle computation instead.

In the course of proving Theorem \ref{Thm:Homspaces} we will also obtain Theorem \ref{Thm:SObstr}, which is an analogue of the strong approximation result by Borovoi and Demarche \cite{borovoi} with a finite set of places removed. This appears to be new, and, as it will be remarked, does not seem to follow directly by projection from \cite[Thm 1.4]{borovoi}, as one may think at first sight.

Part of the results of this paper were also obtained independently and almost simultaneously by Francesca Balestrieri \cite{Francesca}. Namely, she proves Theorem \ref{Thm:Homspaces} in the case that $X$ has a rational point and is a homogeneous space under a linear group. The results of this paper that are not covered by hers are Theorem \ref{Thm:Homspaces} in the case that $X(K) = \emptyset$ (in particular, handling this case is what forces us to use the Weil restriction arguments of Section \ref{Sec:MainTheorem}), Theorems \ref{Thm:SObstr} and \ref{Thm:S-et-Obstr} (which deal with the question of strong approximation after removing some non-archimedean places), and Theorem \ref{Thm:compatibility}, which aims to connect \cite[Cor. 6.3]{Demarche_abelian} with \cite[Thm 1.4]{borovoi}.

\subsection{Structure of the paper}

In Section \ref{Sec:notation} and Section \ref{Sec:pre} we set up the notation and recall some of the preliminaries for the results that are proved in the paper. 
In Section \ref{Sec:MainTheorem} we prove Theorem \ref{Thm:Homspaces}. 
In Section \ref{Sec:compatibility}, we prove Theorem \ref{Thm:compatibility}, which we use to prove the equivalence of \cite[Cor. 6.3]{Demarche_abelian} with \cite[Thm. 1.4]{borovoi},  so that we can then use this equivalence in Section \ref{Sec:Removing_places} to prove Theorem \ref{Thm:SObstr}, and then use this last to prove Theorem \ref{Thm:S-et-Obstr}. We also remark that the proof of Theorem \ref{Thm:S-et-Obstr} is logically independent from the proof of Theorem \ref{Thm:Homspaces}, and reproves it completely. However, since, to prove Theorem \ref{Thm:S-et-Obstr}, we use Theorem \ref{Thm:compatibility} with its long calculation, we preferred keeping the two theorems separate in the exposition. Finally, the appendices contain no new results, but just facts that are recalled for convenience, or because they did not appear explicitly in the literature. We include a diagram of the logical implications of this paper and a couple of results in the literature that play an essential role:

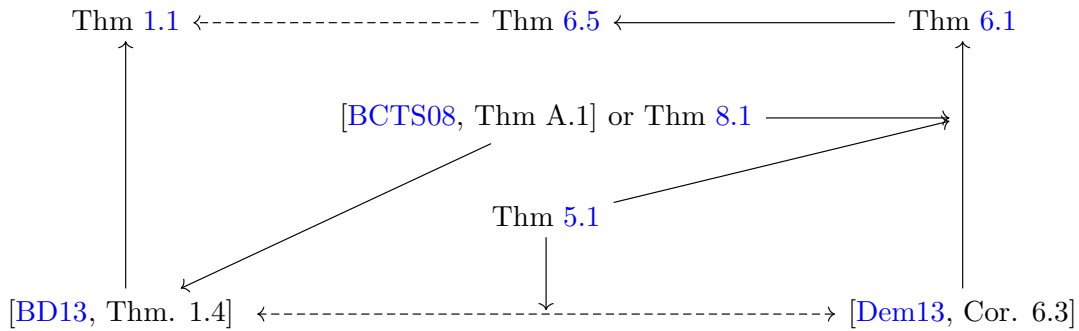
\begin{equation*}
	\begin{tikzcd}
	\text{Thm \ref{Thm:Homspaces}}                            & \text{Thm \ref{Thm:S-et-Obstr}} \arrow[l, dashed]    & \text{Thm \ref{Thm:SObstr}} \arrow[l] \\
	& \text{\cite[Thm A.1]{BoCTSko} or Thm \ref{Thm:HPforX}} \arrow[r] \arrow[ldd] & {}                   \\
		& \text{Thm \ref{Thm:compatibility}} \arrow[d] \arrow[ru]           &                     \\
	\text{\cite[Thm. 1.4]{borovoi} }\arrow[rr, leftrightarrow, dashed] \arrow[uuu] & {}                              & \text{\cite[Cor. 6.3]{Demarche_abelian}} \arrow[uuu]     
	\end{tikzcd}
\end{equation*}
In the diagram above, a non-dotted arrow between $X$ and $Y$ indicates that $X$ is used (either in the current mathematical literature or in this paper) to prove $Y$.
An arrow from $X$ pointing to another arrow indicates that $X$ gets used (again, either in the current mathematical literature or in this paper) to prove the implication it points to.
 A dotted arrow indicates that it is possible to prove $Y$ using $X$ (without using any of the other results appearing in the diagram, except at most the ones pointing to the dotted arrow from $X$ to $Y$).



\begin{acknowledgements}
	I am deeply thankful to my advisor, David Harari, who suggested the topic to me, read the paper carefully and gave many crucial suggestions. I thank Cyril Demarche, for the conversations we had concerning some aspects of his papers and for his precious comments. I also thank my colleagues and friends Guglielmo Nocera and Michele Pernice, who helped me understand the basics of stacks and $2$-groups. I thank Francesca Balestrieri for sharing her preprint with me. Finally, I thank Davide Lombardo for helpful linguistic suggestions.
\end{acknowledgements}

\section{Notation}\label{Sec:notation}
Unless specified otherwise, $k$ will always denote a field of characteristic $0$ and $K$ a number field. For a number field $K$, $M_K$ will denote the set of places of $K$, $M_K^{fin}$ (resp. $M_K^{\infty}$) the finite (resp. archimedean) places. For a place $v \in M_K$, $K_v$ will denote the $v$-adic completion of $K$, and, for $v \in M_K^{fin}$, $O_v \subset K_v$ will denote the $v$-adic integers.   The topological ring of adeles of $K$, i.e. the ring $\prod'_{v \in M_K} K_v$ (the restricted product being on $O_v \subset K_v$) is denoted by $\A_K$. For a finite subset $S \subset M_K$, $\A_K^S$ denotes the topological ring of $S$-adeles, i.e. the ring $\prod'_{v \in M_K \setminus S} K_v$, and $K_S$ denotes the product $\prod_{v \in S} K_v$. We will also use the following notations:
\begin{itemize}
	\item $\A_{O_K}\defeq \prod_{v \in M_K} O_v$;
	\item $\A_{O_{K,S}}  \defeq \prod_{v \in M_K \setminus S} O_v \times \prod_{v \in S} K_v$.
\end{itemize}

When $\phi:X \rightarrow Y$ is a morphism defined over $K$, we will denote by $\phi_v:X_{K_v}\rightarrow Y_{K_v}$ the induced morphism among the base-changed varieties $X_{K_v}=X \times_K  K_v$ and $Y_{K_v}=Y \times_K  K_v$.

All schemes appearing are separated, therefore, we tacitly always assume this hypothesis throughout the paper.

When $X$ is a variety defined over $K$, the notation $X(\A_K)_{\bullet}$ will denote the adelic points where each archimedean component is collapsed to the (discrete) topological space of its connected components. 

An {\itshape adelic-like object} is a product $\prod'_{v \in M_K} P_v$, where $P_v$ (parametrized by the places $v \in M_K$) are sets, such that almost all of them have an integral version $P_{O_v}$ (endowed with a natural morphism $P_{O_v} \rightarrow P_v$), and the restricted product is taken with respect to these integral versions. 

For a subset $Y \subset \prod_{v \in M_K}' P_v$ of an adelic-like object and a set $S \subset M_K$, we denote by $Y_S$ the set $(\pi^S)^{-1}(\pi^S(Y))$, where $\pi^S: \prod_{v \in M_K}' P_v\rightarrow \prod_{v \in M_K \setminus S}' P_v$ is the standard projection.

Group actions (and, correspondingly, homogeneous spaces) will be assumed to be left actions unless specified otherwise. In particular, most of the torsors appearing will, instead, be right. This will be specified each time.

Let $G$ be a connected algebraic group (not necessarily linear) over $k$. Then, according to Chevalley's Theorem (see \cite{Conrad} for a proof), $G$ fits into a (canonical) short exact sequence:
\[
1 \rightarrow G^{lin} \rightarrow G \rightarrow G^{ab} \rightarrow 1,
\]
where $G^{lin}$ is a connected linear $k$-group, and $G^{ab}$ is a $k$-abelian variety.

For a connected algebraic group $G$ over $k$, we will use the following notation (borrowing it from \cite{borovoi}):

$G^u$ is the unipotent radical of $G^{lin}$;

$G^{red}=G^{lin}/G^u$ is the reductive $k$-group associated with $G$;

$G^{ant}$ the maximal anti-affine subvariety of $G$ (we refer to \cite[Ch. 10]{Milne} for the definition and main properties);

$G^{ss}=[G^{red},G^{red}]\subset G^{red}$ is the commutator subgroup of $G^{red}$ (it is a semisimple $k$-group);

$G^{sc}$ is the universal cover of $G^{ss}$, a simply connected semisimple $k$-group;

$G^{scu}$ is the fibered product $G^{sc}\times_{G^{red}}G^{lin}$ (with its canonical group structure), which fits into an exact sequence:
\[
1 \rightarrow G^u \rightarrow G^{scu} \rightarrow G^{sc} \rightarrow 1.
\]
We have a canonical homomorphism $G^{scu}=G^{sc}\times_{G^{red}}G^{lin}\rightarrow G^{lin}\hookrightarrow G$.

For a field $k$, $\Sch_k$ will denote the category of quasi-projective $\Spec k$-schemes.

When $X$ is defined over $L$, a finite extension of $k$, we denote by $R_{L/k}X$ the Weil restriction to $k$ of $X$ (see \cite[Ch. 4]{Weilrestr}  for the definition and basic properties of the functor $R_{L/k}$). When $L$ is Galois over $k$, with Galois group $\Gamma$ (which we consider naturally endowed with its \textit{left} action on $L/k$), there is an action of $\Gamma$ on $\Sch_L$, which may be described as follows. If $(X,p_X) \in \Sch_L$, where $X$ is a scheme and $p_X:X \rightarrow \Spec L$ is the structural morphism, then $(X,p_X)^{\gamma}\defeq (X,\gamma \circ P_X)$, where $\gamma:\Spec L \rightarrow \Spec L$ denotes (the map induced by) conjugation by $\gamma$ (we refer to \cite[4.11.1]{Weilrestr} for more details).

A cover $f:X \rightarrow Y$ is a finite surjective morphism.

Whenever we have two functors $F:\mathcal{C}\rightarrow \mathcal{D}$ and $G:\mathcal{C}\rightarrow \mathcal{D}$, and a collection of morphisms $\mathcal{F}^Y:F(Y) \rightarrow G(Y)$, indexed by the objects $Y \in \mathcal{C}$, we say that such a collection is a natural transformation if, for every morphism $f:Y_1 \rightarrow Y_2$ in $\mathcal{C}$, we have that $G(f)\circ\mathcal{F}^{Y_1}=\mathcal{F}^{Y_2}\circ F(f)$. We also say, with a slight abuse of notation, that the morphisms are a natural transformation when their collection is. 

For a product $\bigtimes_{i \in I}X_i$ and a subset $J \subset I$ we denote (when there is no risk of confusion) by $\pi_J$ the projection $\bigtimes_{i \in I}X_i \rightarrow \bigtimes_{i \in J}X_i$.

When we have a contraviarant functor $\mathcal{F}:\mathcal{C} \rightarrow \Set$, two objects $X, Y \in \mathcal{C}$, a morphism $f:X \rightarrow Y$, and an element $a \in \mathcal{F}(Y)$, we will use the notation $\grestricts{a}{X}$ to denote the element $[\mathcal{F}(f)](a) \in \mathcal{F}(X)$ (taking inspiration from the classical notation $\restricts{a}{X}$ for restrictions, when $f$ is an embedding).

If $A$ is an abelian group, $A^{\wedge}$ will denote its profinite completion.

For a torsion abelian group $A$, $A^D$ will denote the dual $\Hom(A,\Q/\Z)$ endowed with the compact-open topology. If $A$ is a profinite abelian group, $A^D$ will denote the torsion group $\Hom_{cont}(A,\Q/\Z)$, where $\Q/\Z$ is endowed with its discrete topology.

If $G$ is a group, and $M$ is an abelian group, we will say that a set-morphism $G \rightarrow \Aut(M)$ is a group {\it pseudo-action} of $G$ on $M$. We will keep using the term ``group action'' when $G \rightarrow \Aut(M)$ is a group-homomorphism.

When $[H \rightarrow G]$ is a crossed module, and there is no risk of confusion, we will use the notation $\leftidx{^g}{h}$ to denote the left action of $g \in G$ on $h \in H$. We remind that this action is compatible with the left action of $H$ by conjugation on itself. 

If $G/k$ is an algebraic group, and $k \subset F$ is a field extension, we will use the notation $H^i(F,G)$ (with $i \in \N$ and $i=0,1$ if $G$ is not abelian) to denote the cohomology group/set $H^i(\Gamma_{\overline{F}/F},G(\overline{F}))$. If $G$ is not abelian we assume that the cohomological set $H^1(\Gamma_{\overline{F}/F},G(\overline{F}))$ is the one of right cocycles (i.e. those that correspond to left $F$-torsors under $G$ through \cite[p.18, 2.10]{Skorobogatov}). We will use the notation $H_{sx}^1(\Gamma_{\overline{F}/F},G(\overline{F}))$ to denote left cocycles instead.

If $\eta \in H^1(K,G)$, we use the notation $G^{\eta}$ to denote the inner twist of $G$ by $\eta$, and $G_{\eta}$ (resp. $G_{\eta'}$) to denote the left (resp. right) principal homogeneous space of $G$ obtained by twisting $G$ by $\eta$ (resp. $\eta'=\eta^{-1}$, which is a right cocycle). This twist is naturally endowed with a right (resp. left) action of $G^{\eta}$. See \cite[p. 12-13]{Skorobogatov} for more details on these constructions.

If $Z$ is a $k$-variety endowed with a right $G$-action, and $\eta \in H^1(K,G)$, we use the notation $Z_{\eta}$ to denote the twisted $k$-variety  $(Z \times^G G_{\eta})$ \cite[p. 20]{Skorobogatov}. This is naturally endowed with a right $G^{\eta}$-action. Analogously, one can twist with respect to a left action and a left cocycle $\eta' \in H_{sx}^1(K,G)$.


\section{Reminders}\label{Sec:pre}

We recall that, when $X$ is a variety defined over a number field $K$, the Brauer group $\Br(X)$ is defined as $H^2(X_{\acute{e}t},\G_m)$, and there exists a canonical pairing (called the \textit{Brauer Manin} pairing):
\[
\begin{matrix}
\Br(X)\times X(\A_K) &\longrightarrow &\Q/\Z, \\
(b,x)&\longmapsto & <b,x>,
\end{matrix}
\]
which is defined as follows: if $x=(x_v)_{v \in M_K}$, then $<b,x>=\sum_{v \in M_K} \inv_v(b(x_v))$, where $\inv_v:\Br(K_v)\rightarrow \Q/\Z$ is the usual invariant map (see e.g. \cite[Thm 8.9]{Harari_book} for a definition). The pairing is continuous in $x$ and additive in $b$. If $x \in X(K) \subset X(\A_K)$ or $b$ comes from $\Br K$, then $<b,x>=0$ (see \cite[Sec. 5]{Skorobogatov} for a proof, this is essentially a consequence of the classical Albert–Brauer–Hasse–Noether Theorem). 

We recall, moreover, that the pairing (being continuous on $X(\A_K)$, and taking values in a discrete group) is constant on the archimedean connected components of $X$, hence it induces a pairing:
\[
\Br(X)\times X(\A_K)_{\bullet} \longrightarrow \Q/\Z.
\]

We refer the reader to \cite[Section 5]{Skorobogatov} for more details on the Brauer-Manin pairing.

We denote by $X(\A_K)^{\Br X}$ the following (closed) subset of $X(\A_K)$:
\[
\{x \in X(\A_K) \mid \ <b,x>=0 \ \forall b \in \Br(X)\}.
\]
We then have that $\overline{X(K)} \subset X(\A_K)^{\Br X}\subset X(\A_K)$, i.e. $\Br (X)$ provides an obstruction to the existence and (adelic) density of $K$-rational points.

We also recall that, for an algebraic group $G/K$, and for all (right) torsors $f:Y\xrightarrow{G}X$ under $G$, one has that:
\begin{equation}\label{Eq:RatPointsTorsor}
	X(K)\subset \bigcup_{[\sigma]\in H_{sx}^1(K,G)} f^{\sigma}(Y^{\sigma}(K)),
\end{equation}

where, for any cocycle $\sigma \in Z^1(K,G)$, $f^{\sigma}:Y^{\sigma}\xrightarrow{G^{\sigma}}X$ denotes the torsor $f$ twisted by $\sigma$ (see \cite[Sec 2.2]{Skorobogatov}), and $[\sigma]$ denotes the class of $\sigma$ in $H_{sx}^1(K,G)$. For a right torsor $f:Y\xrightarrow{G} X$, we define:

\[
X(\A_K)^f \defeq \bigcup_{[\sigma]\in H_{sx}^1(K,G)} f^{\sigma}(Y^{\sigma}(\A_K)).
\]

Because of (\ref{Eq:RatPointsTorsor}), one has the following two inclusions:
\[
X(K)\subset X(\A_K)^{\acute{e}t,\Br} \defeq \bigcap_{\substack{f:Y \xrightarrow{G} X \\  G \text{ finite} \\ \text{group scheme}}}\bigcup_{[\sigma]\in H_{sx}^1(K,G)} f^{\sigma}(Y^{\sigma}(\A_K)^{\Br Y^{\sigma}}), 
\]
and
\[
X(K)\subset X(\A_K)^{desc} \defeq \bigcap_{\substack{f:Y \xrightarrow{G} X \\  G \ \text{linear}}} X(\A_K)^f.
\]

Hence, both $X(\A_K)^{\acute{e}t,\Br}$ and $X(\A_K)^{desc}$ provide obstructions to the existence of $K$-rational points. Less obviously, Cao, Demarche and Xu \cite[Prop. 6.4]{demarche} prove, through a Chevalley-Weil-like argument, that both $X(\A_K)^{\acute{e}t,\Br}$ and $X(\A_K)^{desc}$ are closed in $X(\A_K)$, hence they provide an obstruction to (adelic) density of $K$-rational points as well. Moreover, they prove that the two obstructions are in fact equal, i.e. $X(\A_K)^{\acute{e}t,\Br}=X(\A_K)^{desc}$.

The question of strong approximation asks whether, for a $K$-variety $X$ (where $X$ is not necessarily proper), and a finite subset $S \subset M_K$, $X(K)$ is dense in $X(\A_K^S)$ (where it is embedded through its diagonal image). Unfortunately, when $S=\emptyset$ one does not expect in general to have strong approximation for affine varieties. This is mainly because of a compactness issue. 
Namely, in the affine case, $X(K)$ is closed in $X(\A_K)$ (since $\A^n(K)$ is closed in $\A^n(\A_K)$), so there is no chance for $X(K)$ to be dense in it, unless $X(K)=X(\A_K)$, which can never happen if $\dim X \geq 1$ and $X(\A_K) \neq \emptyset$. Using the same argument, one sees that, in order for an affine variety $X$ with $X(\A_K^S) \neq \emptyset$ to satisfy strong approximation outside $S$, there must exist at least one $v \in S$ for which $X_{K_v}$ is not compact.

\subsection{Reminders on Galois cocycles}\label{Ssec:cocycles}

Let $\Gamma$ be a profinite group, and $C=[\cdots \rightarrow 0 \rightarrow M_n\xrightarrow{f_n} \dots \xrightarrow{f_{-n+1}} M_{-n} \rightarrow 0 \rightarrow \cdots]$ be a complex of $\Gamma$-modules, where $M_i$ is in degree $-i$. 

For a $\Gamma$-module $M$:

\begin{enumerate}
	\item for $j \geq 0$, $C^j \defeq C^j(\Gamma,M)$ denotes the set $\Fun (\Gamma^n, M)$ (and we will denote functions in $\Fun (\Gamma^n, M)$ with the notation $\alpha_{\sigma_1,\dots,\sigma_j}$), for $j<0$ we set $C^j\defeq 0$;
	\item $\partial ^j:C^j\rightarrow C^{j+1}$ denotes the morphism 
	$$\left(\partial ^j \alpha\right)_{\sigma_1,\ldots,\sigma_{j+1}} \defeq -\leftidx{^{\sigma_1}}{\alpha_{\sigma_2,\ldots,\sigma_{j+1}}}+\sum_{i=1}^{j}(-1)^{i+1} \alpha_{\sigma_1,\ldots,\sigma_{i-1},\sigma_{i} \sigma_{i+1}\ldots,\sigma_{j+1}}+(-1)^{j} \alpha_{\sigma_1,\ldots,\sigma_{j}};$$
	\item $Z^j \defeq Z^j(\Gamma,M) \defeq \Ker (\partial ^j)$, and $B^{j}\defeq \im (\partial ^{j-1})$ (and $B^0\defeq 0$), and $H^j(\Gamma,M) \defeq \faktor{Z^j}{B^j}$;
	\item For a pairing $\star \cdot \star: M \times M' \rightarrow M''$, $\alpha \in C^j(\Gamma,M)$ and $\beta \in C^k(\Gamma,M')$, $\alpha \cup \beta \in C^{j+k}(\Gamma,M'')$ denotes the cocycle $(\alpha \cup \beta)_{\sigma_1,\ldots,\sigma_{j+k}}=\alpha_{\sigma_1,\dots,\sigma_j}\cdot \leftidx{^{\sigma_1 \cdots \sigma_j}}{\beta_{\sigma_{j+1},\ldots,\sigma_{j+k}}}$.
\end{enumerate}

\begin{remark}\label{Rmk:Pseudo_Action}
	When $\Gamma$ pseudo-acts on $M$ (instead of acting, see Section \ref{Sec:notation}), we still use the above notation, which keeps making sense (even though in this case it is not directly related to any cohomology construction, at least to the author's knowledge).
\end{remark}

For the complex $C$:

\begin{enumerate}
	\item For $j \in \Z$, $C^j \defeq C^j(\Gamma, C) \defeq  C^{n+j}(\Gamma,M_n) \bigoplus \dots \bigoplus C^{j-n}(\Gamma,M_{-n})$;
	\item $\partial ^j:C^j\rightarrow C^{j+1}$ denotes the morphism 
	\[\partial ^j((\alpha_n, \dots, \alpha_{-n})) = (\partial ^{j+n} \alpha_n,(-1)^{j+1} f_n(\alpha_n) + \partial ^{j+n-1}\alpha_{n-1},\dots, {(-1)^{j+1}}f_1(\alpha_{-n+1}) + \ \partial ^{j-n} \alpha_{-n});\]
	\item the $j$-cocyles are $Z^j \defeq Z^j(\Gamma,C) \defeq \Ker (\partial ^j)$, the $j$-coborders are $B^{j+1}\defeq \im (\partial ^{j+1})$, and the $j$-hypercohomology is $H^j(\Gamma,C) \defeq \faktor{Z^j}{B^j}$.
\end{enumerate}

\begin{notation}\label{Not:Delta}
	When $\alpha, \beta \in C^j(\Gamma,C)$ we use the notation $\alpha \triangleq \beta$ to mean that $\alpha - \beta \in B^j(\Gamma,C)$.
\end{notation}


\begin{notation}
	To avoid having too many subscripts, in the course of the proof of Theorem \ref{Thm:compatibility} we will use $\sigma, \eta$ instead of $\sigma_1, \sigma_2$.
\end{notation}

\begin{remark}\label{Rmk:computing_spec}
	Let $M[2] \xrightarrow{i} C \defeq [M_2 \xrightarrow{f_2} M_1 \xrightarrow{f_1} M_0]$ be a morphism of complexes of $\Gamma$-modules that is a quasi-isomorphism, and $\alpha=(a_{\gamma_1,\gamma_2},b_{\gamma},c) \in Z^0(\Gamma,C)$. Let $b_{\gamma}'\in \Fun(\Gamma,M_2)$ be such that $f_2(b_{\gamma}')= b_{\gamma} - (\partial c')$,  where $c' \in M_1$ is such that $f_1(c')=c$. Let $\beta \in Z^2(\Gamma,M)=Z^0(\Gamma,M[2])$ be such that $i_*(\beta) = a_{\gamma_1,\gamma_2}- (\partial b_{\gamma}') \in Z^0(\Gamma,C)$ satisfies that $i_*[\beta]=[\alpha] \in \H^0(\Gamma,C)$.
\end{remark}

The following lemma is well-known (since it basically just unravels the definition of derived cup product in a special case), however the author was not able to find a reference with the explicit signs, which will be needed for the computation in the proof of Theorem \ref{Thm:compatibility}. We recall that, for a bicomplex $\mathcal{C}= C^{\bullet,\bullet}$ (with differentials $\partial_1: C^{\bullet,\bullet} \rightarrow C^{\bullet+1,\bullet}, \partial_2:C^{\bullet,\bullet} \rightarrow C^{\bullet,\bullet+1}$), the totalizing complex (see e.g. \cite[p. 8]{Weibel}) $\Tot^{\bullet}(\mathcal{C})$ is defined through $\Tot^{n}(\mathcal{C}) \defeq \oplus_{i+j=n} C^{i,j}$, with differential $\partial_1+(-1)^{i}\partial_2$.


\begin{lemma}\label{Lem:Cup_product}
	Let $\mathcal{C} := [\dots \rightarrow C_n \xrightarrow{d} C_{n+1} \rightarrow \dots]$ and $\mathcal{C'} := [\dots \xrightarrow{d} C_n' \rightarrow C_{n+1}' \rightarrow \dots]$ be two bounded complexes of $\Gamma_k$-modules, and let 
	\begin{equation}\label{Eq:pairings}
		C_i \otimes C_{-i}' \xrightarrow{(.)} \bar{k}^*, \ i \in \Z,
	\end{equation}
	be compatible pairings of $\Gamma_k$-modules, i.e. such that the following pairing diagrams are commutative for all $i \in \Z$:
	
	
	\[
		\begin{tikzcd}[column sep=tiny]
	{	C_i} \arrow[d, "d"'] &\times &{C'_{-i}} \arrow[rrr]                       &  &  & {\bar{k}^*}       \\
	{C_{i+1}}     & \times &{C'_{-i-1}} \arrow[rrr] \arrow[u, "d"']  &  &  & {\bar{k}^*} \arrow[u, Rightarrow, no head]                                                                
	\end{tikzcd}.
	\]

	Then, we have a (canonical) pairing:
	\[
	\H^i(k,\mathcal{C}) \otimes \H^j(k,\mathcal{C'}) \xrightarrow{\cup} \H^{i+j}(k,\bar{k}^*),
	\] 
	that is induced from the following map on the level of cochains:
	
	
	\begin{equation}\label{Eq:Cup_product}
	\begin{matrix}
			&C^i(k,\mathcal{C}) &\otimes &C^j(k,\mathcal{C'}) &\xrightarrow{\cup} &C^{i+j}(k,\bar{k}^*),\\
	
		&(\alpha_h)_{h \in \Z} &\otimes &(\alpha'_h)_{h \in \Z} &\mapsto & \sum_{h \in \Z}(-1)^{jh+\binom{h}{2}}\alpha_h \cup \alpha_{-h} .
	\end{matrix}
	\end{equation}

	Moreover, the pairing induced by (\ref{Eq:Cup_product}) coincides with the derived cup-product associated with the derived pairing $\mathcal{C}\otimes^L \mathcal{C'}\rightarrow \Tot^{\bullet}(\mathcal{C}^{\bullet}\otimes\mathcal{C'}^{\bullet})  \rightarrow \overline{k}^*[0]$ induced by the pairings (\ref{Eq:pairings}).
\end{lemma}

\begin{proof}
	This is basically just unraveling the definitions of derived cup-product. We recall that, if we have a morphism $\Tot^{\bullet}(\mathcal{C}^{\bullet}\otimes\mathcal{C'}^{\bullet})\rightarrow K^{\bullet}$ , this induces a morphism $\mathcal{C}^{\bullet}\otimes^L\mathcal{C'}^{\bullet} \rightarrow\Tot^{\bullet}(\mathcal{C}^{\bullet}\otimes\mathcal{C'}^{\bullet})\rightarrow K^{\bullet}$. 
	We denote by $R^{\bullet}\Gamma:D(\Ab_{\Gamma_k})\rightarrow D(\Ab)$ the functor that computes group cohomology of $\Gamma_K$-modules through the standard resolution (so, as recalled in Subsection \ref{Ssec:cocycles}). 
	The derived cup-product is then defined (in this case) as the following composition:
	\[
	\Tot^{\bullet}(\H^{\bullet}(R^{\bullet}\Gamma\mathcal{C})\otimes \H^{\bullet}(R^{\bullet}\Gamma\mathcal{C'})) \rightarrow \H^{\bullet}(\Tot^{\bullet}(R^{\bullet}\Gamma\mathcal{C} \otimes R^{\bullet}\Gamma\mathcal{C'})) \rightarrow \H^{\bullet}(R^{\bullet}\Gamma(\Tot^{\bullet}(\mathcal{C}\otimes \mathcal{C'})))\rightarrow \H^{\bullet}(R^{\bullet}\Gamma(K^{\bullet})).
	\]
	
	Unraveling the morphisms above with the correct signs, reveals that the pairing is exactly the one in (\ref{Eq:Cup_product}) (the correct signs can also be found by checking that (\ref{Eq:Cup_product}) is the only choice of signs compatible with the Leibniz rule for cup products, and with trivial signs at the degree $0$ level).
%
\end{proof}

\subsection{Reminders on abelian(ized) cohomology}\label{SSec:abelianized}

We brielfy recall here a construction of Demarche \cite{Demarche_abelian}, as well as his main result in \textit{loc. cit.}. We refer the reader to his paper for a more detailed exposition and for proofs of the results that follow.

In what follows $X$ will always denote a quotient $G/H$, where $G$ is a connected $k$-group, and $H$ is a connected linear closed subgroup.

\paragraph{Construction of abelianized cohomology} 	We assume first that $G^{lin}$ is reductive. Let $T_H \subset T_G$ be two maximal tori in, respectively, $H$ and $G$. Let $SA_G$ be a maximal semi-abelian subvariety of $G$ containing $T_G$. Let $H^{sc}\xrightarrow{\rho_H}{H}$ and $G^{sc}\xrightarrow{\rho_G}{G}$ be as in Section \ref{Sec:notation}, and let $T_H^{sc}$ and $T_G^{sc}$ be, respectively, the tori $T_H \times_H H^{sc}$ and $SA_G \times_G G^{sc}$.

Let $C_X$ be the cone of the following morphism of complexes:
\[
[T_H^{sc}\xrightarrow{\rho_H}{T_H}] \xrightarrow{[\iota^{sc},\iota]} [T_G^{sc} \xrightarrow{\rho_G} SA_G],
\]

where both complexes start in degree $-1$, and $\iota:H \rightarrow G$ denotes the closed embedding. We call $C_X$ the {\itshape abelianized complex} of $X$. Moreover, we set the notation $C_H \defeq [T_H^{sc}\xrightarrow{\rho_H}{T_H}]$ and $C_G \defeq [T_G^{sc} \xrightarrow{\rho_G} SA_G]$.

For a field $F \supset k$, we denote the (Galois) hypercohomology of $C_X$ by $\H_{ab}^i(F,X) \defeq \H^i(F,C_X)$, and we refer to it as the \textit{abelianized cohomology} of $X$.

When $G^{lin}$ is not reductive, let $G'\defeq G/G^u$, where $G^u \subset G$ denotes the unipotent radical of $G$, let $H_1 \defeq \iota(H)/(\iota(H)\cap  G^u)$, and let $X' \defeq G'/H_1$. There exists a natural surjection $X \rightarrow X'$. One can repeat the above ``abelianization'' construction for $X'$, and defines $\H_{ab}^i(F,X) \defeq \H^i(F,C_{X'})$. We refer to \cite{Demarche_abelian} for more details on this.

\paragraph{Abelianization morphism} For each field $F \supset k$, there exists an \textit{abelianization} map \cite{Demarche_abelian}: 
\begin{equation}\label{Eq:ab0}
	\ab_F^0:X(F) \rightarrow \H^0(F,C_X).
\end{equation}
 The map (\ref{Eq:ab0}) factors (by construction) through $X'(F)$ as follows: $X(F) \rightarrow X'(F) \rightarrow \H^0(F,C_{X'})$. When $F$ is a local field there is a natural topology on $\H^0(F,C_X)$, see \cite[p. 20]{Demarche_abelian}. Moreover, we have the following lemma (implicit in \cite{Demarche_abelian}):
\begin{lemma}\label{Lem:ab_continuity_openness}
	With the above notation, let $k=K$ be a number field. We choose smooth group-scheme models, over $\Spec O_{K,S}$, where $S$ is some finite set of primes of $K$, $\tilde{\iota}: \mathcal{H} \rightarrow \mathcal{G}$, for $H$, $G$ and the closed embedding $\iota$. We define $\mathcal{X} \defeq \mathcal{G}/\mathcal{H}$, and define $C_{\mathcal{X}}$ to be the abelianized complex of $\mathcal{X}$.
	The following hold:
	\begin{enumerate}
		\item For a non-archimedean  $v \in M_K$, and $F=K_v$, the local abelianization map (\ref{Eq:ab0}) $\ab_{K_v}^0$ is surjective;
		\item For all $v \in M_K$, $F=K_v$, $\ab_{K_v}^0$ is continuous and open, and for $v \notin S \cup M_K^{\infty}$, $(\ab_{K_v}^0)(\mathcal{X}(O_v))=\H^0(O_v,C_\mathcal{X})$;
	\end{enumerate}
	Moreover, if $\P^0(K,C_X) \defeq \prod'_{v \in M_K} \H^0(K_v,C_X)$ denotes the restricted product with respect to $\H^0(O_v,C_\mathcal{X}) \subset \H^0(K_v,C_X)$, we have that the restricted product morphism $\ab^0:X(\A_K)_{\bullet} \rightarrow \P^0(K,C_X)$ is continuous and open.
\end{lemma}
\begin{proof}
	The last statement is an immediate consequence of the points above. For the first point and the first part of the second point above see \cite[Cor. 2.21]{Demarche_abelian} and \cite[p. 21]{Demarche_abelian}. For the second part of the second point, we notice that $\ab^0_{K_v}(\mathcal{X}(O_v))\subset \H^0(O_v,\mathcal{C_X})$ simply by functoriality. For the equality, see the following commutative diagram with exact (as pointed-sets) rows:
	\begin{equation*}
		\begin{tikzcd}
		\mathcal{H}(O_v) \arrow[r] \arrow[d] & \mathcal{G}(O_v) \arrow[r] \arrow[d] & \mathcal{X}(O_v) \arrow[r] \arrow[d] & {H^1(O_v,\mathcal{H})} \arrow[d] \\
		{\H^0(O_v,C_\mathcal{H})} \arrow[r]  & {\H^0(O_v,C_\mathcal{G})} \arrow[r]  & {\H^0(O_v,C_\mathcal{X})} \arrow[r]  & {\H^1(O_v,C_\mathcal{H})}        
		\end{tikzcd},
	\end{equation*}
	and the following short exact sequence of pointed sets:
	\[
	\mathcal{G}(O_v)\rightarrow \H^0(O_v,\mathcal{C_G}) \rightarrow H^1(O_v,\mathcal{G}^{sc})
	\]
	and notice that $H^1(O_v,\mathcal{H})=0$, $H^1(O_v,\mathcal{G}^{sc})=0$ by \cite[Thm 6.1]{PlatonovRapunchkin}. Moreover, by \cite[Thm 6.1]{PlatonovRapunchkin} again, $H^1(O_v,T_{\mathcal{H}})=0$ and $H^1(O_v,T_{\mathcal{H}^{sc}})=0$. We also have $H^2(O_v,T_{\mathcal{H}^{sc}})=0$ because $H^2(O_v,(T_{\mathcal{H}^{sc}})_{tor}) \twoheadrightarrow H^2(O_v,T_{\mathcal{H}^{sc}})$ and $H^2(O_v,(T_{\mathcal{H}^{sc}})_{tor})=0$ because the cohomological dimension of $\Spec \mathbb{F}_v$ is $1$. Hence $\H^1(O_v,\mathcal{C_H})=\H^1(O_v,[T_{\mathcal{H}^{sc}} \rightarrow T_{\mathcal{H}}])=0$, where, for a commutative algebraic group scheme $G \rightarrow S$ the subscript $tor$ defines the increasing limit $\lim_{n \to \infty} G[n]$ (the limit here is taken with respect to the partial order $n \prec m$ if $n|m$). Hence, the surjectivity of $X(O_v) \rightarrow \H^0(O_v,\mathcal{C_X})$ follows.
%
%
%
\end{proof}


\paragraph{Relation to Brauer-Manin obstruction} Let $C_X^d$ be the dual complex of $C_X$ (this complex is described explicitly in \cite[Sec. 5]{Demarche_abelian}, and its definition will be recalled in Subsection \ref{Ssec:Proof_theorem}, diagram (\ref{Table1})). We have, for every field $F \supset k$, a pairing (see Lemma \ref{Lem:Cup_product}):
\begin{equation}\label{Mor:pairing}
\H^0(F,C_X)\times \H^2(F,C_X^d)\rightarrow \H^2(F,\overline{F}^*)=\Br(F).
\end{equation}
When $F$ is a local field, (\ref{Mor:pairing}) induces, through the local invariant $\inv_v:\Br(F) \rightarrow \Q/\Z$,  a pairing:
\begin{equation}\label{Mor:local_pairing}
\H^0(F,C_X)^{\wedge}\times \H^2(F,C_X^d)\rightarrow \Q/\Z,
\end{equation}
which is perfect when $v \in M_K^{fin}$ (since $C_X$ is a cone of a complex of $1$-motives, this basically follows by devissage from \cite[Thm 0.1]{harariszamuely}, and is proved explicitly in Lemma \ref{Lem:duality}). For a number field $K$, the local pairing (\ref{Mor:local_pairing}) on its completions, induces a morphism (see \cite[p.21]{Demarche_abelian}):

\begin{equation}\label{Mor:theta}
\P^0(K,C_X)^{\wedge}\xrightarrow{\theta} \H^2(K,C_X^d)^D.
\end{equation}

Demarche \cite{Demarche_comparison} defined a morphism

\begin{equation}\label{Eq:alpha}
\alpha:\H^2(k,C_X^d) \rightarrow \Br_a(X_k,G) ,
\end{equation}

where $\Br_1(X,G) \defeq \Ker(\Br(X)\rightarrow \Br(\bar{G}))\subset \Br(X)$ and $\Br_a(X,G) \defeq \Br_1(X,G)/\Br(k)$ (we refer to {\itshape loc. cit.} or to Subsection \ref{Ssec:Proof_theorem} for more details).

Moreover, Demarche proved that $\alpha$ sits in the following exact sequence (which will not be needed in this paper, but is included for completeness):  
\[
\mathrm{NS}(\bar{G}^{\mathrm{ab}})^{\Gamma_{k}} \rightarrow \H^{2}(k, C^d_{X}) \xrightarrow{\alpha} \operatorname{Br}_{a}(X, G) \rightarrow H^{1}(k, \mathrm{NS}(\bar{G}^{\mathrm{ab}})).
\] 

We will prove in Section \ref{Sec:compatibility} Theorem \ref{Thm:compatibility}, which proves that $\alpha$ is compatible with Brauer-Manin and local pairings.

The following is the main theorem of \cite{Demarche_abelian}.

\begin{theorem}[Demarche]\label{Thm:Dem_ab}
	Let $K$ be a number field, $G$ a connected $K$-group, $S$ a finite set of places of $K$. Let $H$ be a connected linear $K$-subgroup of $G$, and let $X \defeq G/H$. We assume that the group $G^{sc}$ satisfies strong approximation outside $S$, and that $\Sha(K,G^{ab})$ is finite. Let $C_X^d$ and $\theta:X(\A_K)_{\bullet} \rightarrow (\H^0(K,C_X^d)/\Sha(K,C_X^d))^D$ be defined as above. Then the kernel of $\theta$ (i.e. $\theta^{-1}(\{0\})$) is the closure of $G^{scu}(K_{S^{fin}})\cdot X(K)$ in $X(\A_K)_{\bullet}$.
\end{theorem}

\begin{remark}
	As remarked by Demarche in \cite[Rmq 6.4]{Demarche_abelian}, Theorem \ref{Thm:Dem_ab} implies the main theorem of \cite{borovoi}, once a suitable compatibility of Brauer and local duality pairings is proven. The compatibility needed is exactly the one proven in Theorem \ref{Thm:compatibility}. In fact, in Section \ref{Sec:Removing_places}, this connection will be made explicit in Theorem \ref{Thm:SObstr} and Remark \ref{...tobeadded}.
\end{remark}

\subsection{Reminders on the morphism $\alpha$}\label{SSec:abelianized2}


In this subsection we recall part of the construction of $\alpha$, given in \cite{Demarche_comparison}, of which we borrow the notation. Throughout $H$ will be a linear connected subgroup of a connected algebraic group $G$ with $G^{lin}$ reductive, both defined over a field $k$.

We fix a Levi decomposition of $H=H^u\cdot H^{red}$, such that $Z_{H^{\text {red }}}$ is contained in the maximal torus $T_{G}$ of $G^{\text {lin }}$ (keeping the notation of Subsection \ref{SSec:abelianized}). We put $Z^{\prime} :=T_{G}/Z_{H^{\text {red }}}, Z_{1} :=G / Z_{H^{\text {red }}}, H^{\prime}:=H^{\text {red }} / Z_{H^{\text {red }}}, \widetilde{H} :=\operatorname{Ker}\left(H \rightarrow H^{\prime}\right)=H^u \cdot Z_{H^{\text {red }}}$, and $Z :=G / \widetilde{H}$. We have a natural morphism
$Z_{1} \rightarrow Z$ that gives $Z_{1}$ the structure of a $Z$-torsor under $H^{u}$. We notice that $X=Z/H'$, and that we have the following commutative diagram with exact rows and columns:

\[
\begin{tikzcd}
H^u \arrow[r, hook] \arrow[d, Rightarrow, no head] & \tilde{H} \arrow[r, two heads] \arrow[d, hook] & Z_{H^{red}} \arrow[d, hook]  \\
H^u \arrow[r, hook]                                & H \arrow[d, two heads] \arrow[r, two heads]    & H^{red} \arrow[d, two heads] \\
& H' \arrow[r, Rightarrow, no head]              & H'                          
\end{tikzcd}
\]

We use the following notation

\[Q_{geom} \defeq [ \bar{k}(Z)^{*}  \rightarrow \operatorname{Div}(\bar{Z})   \rightarrow                  \operatorname{Pic}^{\prime}(\bar{Z} / \bar{X})   ],\]
where the complex ends in degree $2$.
We refer to \cite[p. 4]{Demarche_comparison} for the definition of $\operatorname{Pic}^{\prime}(\bar{Z} / \bar{X})$, which we remind being isomorphic to $\Pic (\overline{H'})$.


The following proposition is the first step in Demarche's construction of the aforementioned morphism $\alpha$. We refer to \cite[p. 9]{Weibel} for a definition of the truncation operator.
The following is consequence of \cite[Thm. 2.1, Prop. 2.2]{Demarche_comparison}:
\begin{proposition}[Demarche]\label{Prop:Demarche_derived}
	We keep the above notation, and assume that $k$ is either a number field or a local field.
	Let $p:Z\rightarrow X$ denote the natural projection, and $p_X:X \rightarrow \Spec k$ the structural morphism. There exists a natural 
	quasi-isomorphism
	\begin{equation}\label{Triangle}
		\tau_{ \leq 2} \mathbf{R} p_{X_{*}} \mathbf{G}_{m X}\rightarrow Q_{geom},
	\end{equation}
	  which induces an isomorphism:
	\begin{equation}\label{Iso:CgeomtoBrauer}
	\H^2(k,Q_{geom})\xleftarrow{\sim} \Br_a(X,G).
	\end{equation}
\end{proposition}

The second step is a morphism $\H^2(k, C_X^d) \rightarrow  \H^2(k,Q_{geom}) \cong \Br_a(X,G)$, but this will be recalled in detail in the next section, so we skip its construction for now.






\section{Sufficiency of étale-Brauer-Manin obstruction}\label{Sec:MainTheorem}

In this section, we prove Theorem \ref{Thm:Homspaces}.

\begin{remark}\label{Rmk:BorovoiDemarche}
	In the proof of Theorem \ref{Thm:Homspaces} we are going to use the aforementioned result of Borovoi and Demarche. Their main theorem in \cite{borovoi} can be stated exactly as Theorem \ref{Thm:Homspaces}, but restricting to the case when $X$ has connected geometric stabilizers, and substituting  $X(\A_K)_{\bullet}^{\acute{e}t,\Br}$ with $X(\A_K)_{\bullet}^{\Br X}$.
\end{remark}


\begin{example}\label{Ex:Demarche}
	This example is borrowed from \cite[Thm 2.1]{Demarche_example}. Let $p$ be a prime, and $H$ a finite constant non-commutative group of order $p^n$, such that the $p^{n+1}$-roots of unity are contained in $K$. Let $X \defeq G/H$, where $G$ is any semisimple simply connected algebraic group, and $H \hookrightarrow G$ is any embedding. Then, for any $S$, one has that $X(\A_K^S)^{\Br X} \neq \overline{X(K)}^S$ (where $\overline{X(K)}^S$ denotes the closure of $X(K)$ in $X(\A_K^S)$).
	In particular, in general, one could not hope for the statement of Theorem \ref{Thm:Homspaces} to be true with $X(\A_K)_{\bullet}^{\acute{e}t,\Br}$ replaced by $X(\A_K)_{\bullet}^{\Br X}$, i.e. the Brauer-Manin obstruction is {\itshape not} the only one to strong approximation for homogeneous space. 
\end{example}

\begin{remark}\label{Rmk:Demarche}
	 In view of the example above, it would be interesting to know if there are any intermediate obstructions $X(\A_K)_{\bullet}^{\acute{e}t,\Br} \subset X(\A_K)_{\bullet}^{?} \subset X(\A_K)_{\bullet}^{\Br X}$ such that $G^{scu}(K_{S_f}) \cdot X(K)$ is dense in $X(\A_K)_{\bullet}^{?}$.
\end{remark}

\begin{remark}\label{Rmk:Sha}
	The group $\Sha(K,G^{ab})$ is defined as the kernel of the map \[H^1(K,G^{ab})\rightarrow \prod_{v \in M_K} H^1(K_v,G^{ab}).\] It is strongly conjectured to be always finite, and it is known to be in some specific cases, for instance, when $G^{ab}$ is an elliptic curve of analytic rank $0$ or $1$ defined over $\Q$ (see \cite{AnEllCurve}).
\end{remark}

\begin{remark}\label{Rmk:PlatonovThm}
	We remind the reader of the aforementioned Theorem of Platonov \cite[Theorem 7.12]{PlatonovRapunchkin}, which states that $G^{sc}(K)$ is dense in $G^{sc}(\A^S_K)$ if and only if $G^{sc}$ has no simple component $G_i \subset G^{sc}$ with $G_i(K_S)$ compact. This makes the hypothesis ``$G^{sc}(K)$ is dense in $G^{sc}(\A^S_K)$'' of Theorem \ref{Thm:Homspaces} easily verifiable.
\end{remark}

\begin{remark}\label{Rmk:StrApprox}
	We note that, if $x \in X(\A_K)_{\bullet}$ is in $X(\A_K)_{\bullet}^{\acute{e}t,\Br}$, then Theorem \ref{Thm:Homspaces} tells us that it lies in the closure of $G^{scu}(K_{S_f}) \cdot X(K)$. Hence, its projection $x^{S_f}$ to $X(\A_K^{S_f})_{\bullet}$ lies in the closure of $X(K)$ in $X(\A_K^{S_f})_{\bullet}$ (since the projection of $G^{scu}(K_{S_f})$ is trivial). Therefore, Theorem \ref{Thm:Homspaces} may be seen as a theorem saying that, under its assumptions, the étale-Brauer-Manin obstruction to strong approximation is the only one for homogeneous spaces.
\end{remark}

\begin{remark}\label{Rmk:not_trivial_strong}
	We notice that, although, by the previous remark, we have that, for a finite $S \subset M_K^{fin}$ large enough, $\overline{X(K)}^S\supset \pi^S(X(\A_K)_{\bullet}^{\acute{e}t,\Br})$ (where $\pi^S:X(\A_K)_{\bullet}\rightarrow X(\A_K^S)_{\bullet}$ denotes the standard projection and $\overline{\star}^S$ denotes the closure in the $S$-adeles), this is not necessarily an equality, as the example presented in Proposition \ref{GBsp}, which follows, clearly points out. However, one can still aim to describe the set $\overline{X(K)}^S \subset X(\A_K^S)_{\bullet}$, although it becomes a less trivial consequence of Theorem \ref{Thm:Homspaces}. This is done in Section \ref{Sec:Removing_places}.
\end{remark}

\begin{proposition}\label{GBsp}
	Keeping the notation of the above remark, let $K=\Q$, $X=\G_m/\Q$ and $S=\{2\}$. We have that $X(\A_K)_{\bullet}^{\acute{e}t,\Br}=X(\A_{\Q})_{\bullet}^{\Br X}=\overline{X(\Q)}=X(\Q)=\Q^*$, while $\overline{X(\Q)}^S\subset X(\A_\Q^S)_{\bullet}$ is not countable.
\end{proposition}

\begin{proof}
	We have the following inclusions:
	\[
	X(\A_{\Q})_{\bullet}^{\Br X} \supset X(\A_K)_{\bullet}^{\acute{e}t,\Br X}\supset X(\Q)=\Q^*.
	\]
	Moreover, by \cite[Thm. 4]{Harari} (applied to the case $X=\G_m$), we have that $X(\A_{\Q})_{\bullet}^{\Br X} ={\G_m}(\A_{\Q})_{\bullet}^{\Br \G_m} \subset \overline{\Q^*}=\Q^*$ (the closure being in the idelic topology of $(\I_{\Q})_{\bullet}=({\G_m})(\A_{\Q})_{\bullet}$). Hence we have that all the inclusions above are equalities.
	
	On the other hand, we have that $\overline{X(\Q)}^S$, being the closure of $\Q^*$ in $({\I^S_{\Q}})_{\bullet}$, is equal to $\Q^*\cdot 2^{\hat{\Z}}$, where the embedding $2^{\hat{\Z}} \hookrightarrow ({\I^S_{\Q}})_{\bullet}$ is defined as described in the next paragraph.
	
	The morphism
	\[
	2^{\star}:\Z \rightarrow ({\I^S_{\Q}})_{\bullet}, \ \ n \mapsto 2^n,
	\]
	is continuous if we endow $\Z$ with the profinite topology (i.e. the one induced by the embedding $\Z \hookrightarrow \widehat{\Z}$, where $\widehat{Z}$ is endowed with its profinite topology) and $({\I^S_{\Q}})_{\bullet}$ with its natural topology. Therefore, since $({\I^S_{\Q}})_{\bullet}$ is complete, there is a unique continuous extension of $2^{\star}$ to $\hat{\Z}$, which defines an embedding $2^{\hat{\Z}} \hookrightarrow ({\I^S_{\Q}})_{\bullet}$.
	
\end{proof}


\subsection{Preliminaries for the proof of Theorem \ref{Thm:Homspaces}}

\begin{proposition}\label{Prop:FiberedProduct}
	Let $G$ be a $k$-algebraic group, and let $\phi:X \rightarrow Z$, $\psi:Y \rightarrow Z$ be $G$-equivariant morphisms between $k$-varieties equipped with a (left) $G$-action. Then the fibered product $X \times_Z Y$ is equipped with a natural (left) $G$-action, such that the two projections on $X$ and $Y$ are $G$-equivariant.
\end{proposition}

\begin{proof}
	Let $m_X:G\times_k X \rightarrow X$, $m_Y:G\times_k Y \rightarrow Y$, and $m_Z:G\times_k Z \rightarrow Z$ be the morphisms representing the $G$-action on $X$, $Y$ and $Z$. We know by hypothesis that 
	
	\begin{equation}
	\begin{tikzcd}
	Y \arrow[d, "\psi"'] \arrow[r, "m_Y"] & G\times_k Y \arrow[d, "{(\psi,id_G)}"] \\
	Z \arrow[r, "m_Z"]                    & G\times_k Z                           
	\end{tikzcd} \quad \quad \text{ and } \quad  \quad  \begin{tikzcd}
	X \arrow[d, "\phi"'] \arrow[r, "m_X"] & G\times_k X \arrow[d, "{(\phi,id_G)}"] \\
	Z \arrow[r, "m_Z"]                    & G\times_k Z                           
	\end{tikzcd}
	\end{equation}
	commute. This implies that the following diagram (without the dotted arrow) commutes:
	
	\begin{equation}
	\begin{tikzcd}
	&                                & G \times_k X \times_Z Y   \arrow[lld] \arrow[rd] \arrow[dd, "m_{X \times_Z Y}", dotted] &                                            \\
	G\times_k X \arrow[dd, "m_X"] \arrow[rd] &                                &                                                                   & G\times_k Y \arrow[dd, "m_Y"] \arrow[lld] \\
	& G\times_k Z \arrow[dd, "m_Z"] & X \times_Z Y \arrow[lld] \arrow[rd]                               &                                            \\
	X \arrow[rd]                              &                                &                                                                   & Y \arrow[lld]                              \\
	& Z                              &                                                                   &                                           
	\end{tikzcd},
	\end{equation}
	
	and, hence, the dotted arrow $m_{X \times_Z Y}$ exists and makes the whole diagram commute (by the universal property of the fibered product applied on $X \times_Z Y$). It is now a straightforward verification to see that $m_{X \times_Z Y}$ defines a group action that makes the projections of $X \times_Z Y$ to $X$ and $Y$ be $G$-equivariant.
	
	%
	%
	
	
\end{proof}

We shall need the following more or less standard facts about Weil restriction (we recall that, according to our notation, $\Sch_k$ denotes quasi-projective $\Spec k$-schemes):

\begin{proposition}\label{Prop:WeilRestriction}
	Let $L/k$ be a finite Galois extension with Galois group $\Gamma$.
	\begin{enumerate}
		\item[(i)] The functor $R_{L/k}:\Sch_L \rightarrow \Sch_k$ is a right adjoint to the base change functor $\Sch_k \rightarrow \Sch_L$.
		
		\item[(ii)] For every $k$-variety $X \in \Sch_k$ there exists a closed embedding $\iota_X:X \hookrightarrow R_{L/k}X_L$. Moreover, these morphisms form a natural transformation between the identity functor on $\Sch_k$ and $R_{L/k} \circ (\star \times_{\Spec K} \Spec L)$.
		
		\item[(iii)] For every $L$-variety $Y \in \Sch_L$ there is an isomorphism:
		\[
		\psi^Y=(\psi^Y_{\gamma}): (R_{L/k}Y)_L \longrightarrow \prod_{\gamma \in \Gamma}Y^{\gamma}.
		\]
		Moreover, these morphisms form a natural transformation between the functors 
		$Y \mapsto (R_{L/k}Y)_L$ and $Y \mapsto \prod_{\gamma \in \Gamma}Y^{\gamma}$.
		\item[(iv)] For every $X \in \Sch_k$, one has that $\psi^{X_L}\circ (\iota_X)_L=\Delta_{X_L}$, where $\Delta_{X_L}:X_L \rightarrow \prod_{\gamma \in \Gamma} (X_L)^{\gamma}$ denotes the diagonal embedding \footnote{We are implicitly using the fact that, if $\mathcal{X}$ is defined over $k$, then, for every $\gamma \in \Gamma$, there is a natural identification between $\mathcal{X}^{\gamma}$ and $\mathcal{X}$.}.
	\end{enumerate}
\end{proposition}
\begin{proof}
	\item[(i)] This is the definition of Weil restriction, which exists by \cite[Cor. 4.8.1]{Weilrestr}. 
	\item[(ii)] See \cite{Weilrestr}[4.2.5].
	\item[(iii)] See \cite{Weilrestr}[4.11.3].
	\item[(iv)] The morphism $\pi:X_L \rightarrow X$ induces a base changed morphism $\pi_L:(X_L)_L \rightarrow X_L$. Since $L/k$ is Galois, one may identify $(X_L)_L$ with $\coprod_{\gamma \in \Gamma} (X_L)_{\gamma}$, which may again be naturally identified with $\coprod_{\gamma \in \Gamma} (X_L)$. Using this identification, the morphism $\pi_L$ corresponds to the codiagonal morphism.

	This identification induces the following commutative diagram:
	
	\begin{center}
		\begin{equation}\label{Comm_diagr}
		\begin{tikzcd}[]
		\Hom_k(X,R_{L/k}X_L) \arrow[r,"\pi^*"]&[-1em] \Hom_k(X_L,R_{L/k}X_L) \arrow[r,equal]&[-2em] \Hom_L(X_L,(R_{L/k}X_L)_L)\arrow[r,"(\psi^{X_L})_{*}"]& \Hom_L(X_L,\prod_{\gamma \in \Gamma} (X_L)^{\gamma} ) \\
		\Hom_L(X_L,X_L) \arrow[r,"\pi_L^*"] \arrow[u] &\Hom_L((X_L)_L,X_L) \arrow[u] \arrow[rr,"\sim"] &  & \Hom_L(\coprod_{\gamma \in \Gamma}(X_L)^{\gamma^{-1}},X_L), \arrow[u,equal] \\
		\end{tikzcd}
		\end{equation}
	\end{center}
	
	where the first two vertical morphisms are the ones induced from the definition of $R_{L/k}$. The commutativity of the first square follows from the definition of $R_{L/k}$, while the commutativity of the second square is the definition of $\psi^{X_L}$ (see \cite{Weilrestr}[4.11.3]).
	
	Point $(iv)$ now follows from considering the identity morphism in the bottom left corner of (\ref{Comm_diagr}), and looking at its image in the top right corner of the same diagram following the two distinct paths up-right-right and right-right-up.
\end{proof}

\begin{remark}\label{Rmk:Prodisom}
	By Proposition \ref{Prop:WeilRestriction}(i) the functor $R_{L/k}$ preserves (fibered) products. Hence, for every couple $(Y_1,Y_2)$ of $L$-varieties, the morphism $R_{L/k}(Y_1\times_L Y_2) \xrightarrow{R_{L/k}\pi_{1}\times R_{L/k}\pi_{2}} R_{L/k}Y_1 \times_k R_{L/k}Y_2$ is an isomorphism.
\end{remark}


\begin{remark}\label{Rmk:Weil_action}
	We observe that, if $m:G_L \times_L Y   \rightarrow Y$ is an action of $G_L$ on $Y$, there is a natural action of $R_{L/k}G_L$ on $R_{L/k}Y$ defined by the following composition:
	\[
	R_{L/k}G_L \times_k R_{L/k}Y  \xrightarrow{(R_{L/k}\pi_{1}\times R_{L/k}\pi_{2})^{-1}} R_{L/k}(G_L \times_L Y ) \xrightarrow{R_{L/k}m} R_{L/k}Y.
	\]
	Moreover, by functoriality of $R_{L/k}$, this induced action has the property that, for each $G_L$-equivariant morphism $f:Y_1 \rightarrow Y_2$ between $G_L$-varieties, the morphism $R_{L/k}f$ is a $R_{L/k}G_L$-equivariant morphism.
\end{remark}

Let $G$ be a $k$-algebraic group, $L/k$ be a finite Galois extension, with Galois group $\Gamma$, and let $Y/L$ be an $L$-variety endowed with a (left) $G_L$-action. We observe that Proposition \ref{Prop:WeilRestriction}(ii) (applied on $G$) gives a natural embedding $G \hookrightarrow R_{L/k}G_L$, which one can easily verify (using the fact that $\iota_{\star}$ is a natural transformation in $\star \in \Sch_k$)  to be a group homomorphism (and, hence, embedding).

\begin{proposition}\label{Prop:weilres_group}
	With the above notation, the following hold:
	\begin{enumerate}
		\item[(i)] If $Y=X_L$, with $X$ defined over $k$, and the action of $G_L$ is induced by base change from one of $G$ on $X$, then the natural embedding:
		\[
		\iota_X:X \hookrightarrow R_{L/k}X_L
		\]
		is $G$-equivariant (where the $G$-action on $R_{L/k}X_L$ is the one induced from the action of $R_{L/k}G_L$ on $R_{L/k}X_L$, defined as in Remark \ref{Rmk:Weil_action}, restricted to $G$ through the embedding $\iota_G:G \hookrightarrow R_{L/k}G_L$).
		\item[(ii)] There exists a natural $G_L$-equivariant isomorphism:
		\[
		\psi^Y\defeq \prod_{\gamma \in \Gamma}\psi^Y_{\gamma}: (R_{L/k}Y)_L \longrightarrow \prod_{\gamma \in \Gamma}Y^{\gamma},
		\]
		where the action of $G_L$ on $(R_{L/k}Y)_L$ is the one induced from the action of $R_{L/k}G_L$ on $R_{L/k}Y$, defined as in Remark \ref{Rmk:Weil_action}, restricted to $G$ through the embedding $\iota_G:G \hookrightarrow R_{L/k}G_L$, and the action on $\prod_{\gamma \in \Gamma}Y^{\gamma}$ is induced from the diagonal embedding $\Delta_{G_L}:G_L \rightarrow \prod_{\gamma \in \Gamma}G_L^{\gamma} \cong \prod_{\gamma \in \Gamma}G_L$.
	\end{enumerate}
\end{proposition}
\begin{proof}
	\item[(i)] Let $m_X:G\times_k X   \rightarrow X$ be the morphism defining the action of $G$ on $X$. Then, $(i)$ follows from the commutativity of the following diagram:
	
	\begin{center}
		\begin{tikzcd}[column sep=6.5em]
		G\times_k X  \arrow[d, "\iota_{X}\times \iota_G"] \arrow[r,equal] & G\times_k X \arrow[d,"\iota_{X\times G}"] \arrow[r, "m_X"] & X \arrow[d, "\iota_X"] \\
		R_{L/k}G_L \times_k R_{L/k}X_L   \arrow[r, "(R_{L/k}\pi_{1}\times R_{L/k}\pi_{2})^{-1}"] & R_{L/k}(G_L \times_L X_L  ) \arrow[r, "R_{L/k}m_X"]  & R_{L/k}X_L, \\
		\end{tikzcd}
	\end{center}
	
	which, in turn, is a consequence of the fact that, for $\mathcal{X} \in \Sch_k$, the morphisms $\iota_{\mathcal{X}}:\mathcal{X}\rightarrow R_{L/k}\mathcal{X}_L$ introduced in Proposition \ref{Prop:WeilRestriction}[(i)] are a natural transformation.
	\item[(ii)] Let $m_Y:G_L \times_L Y   \rightarrow Y$ be the morphism defining the action of $G_L$ on $Y$. Then, $(ii)$ follows from the commutativity of the following diagram:
	
	\begin{center}
		\begin{tikzcd}[column sep=4em]
		G_L \times_L (R_{L/k}Y)_L  \arrow[r, hook, "  (\iota_G)_L\times \id"] \arrow[d, "\iota \times \psi^Y"] & (R_{L/k}G_L)_L \times_L (R_{L/k}Y)_L  \arrow[d, "\psi^{G_L} \times \psi^Y  "] \arrow[r, "(R_{L/k}\pi_{1}\times R_{L/k}\pi_{2})^{-1}"] &[2em] (R_{L/k}(G_L \times_L Y ))_L \arrow[r, "(R_{L/k}m_Y)_L"] \arrow[d,"\psi^{(G_L \times_L Y )}"]& (R_{L/k}Y)_L \arrow[d,"\psi^Y"]\\
		\prod_{\gamma \in \Gamma}G_L \times_L Y^{\gamma}  \arrow[r, "(\Delta_{G_L}) \times  \id  "] &  \prod_{\gamma \in \Gamma}(G_L)^{\gamma} \times_L\prod_{\gamma \in \Gamma}Y^{\gamma}\arrow[r, "\sim"] & \prod_{\gamma \in \Gamma}(G_L \times_L Y)^{\gamma} \arrow[r,"\prod_{\gamma \in \Gamma} m^{\gamma}"] &\prod_{\gamma \in \Gamma}Y^{\gamma}.\\
		\end{tikzcd}
	\end{center}
	
	The commutativity of the first square follows from the equality $\Delta_{G_L} = \psi^{G_L}\circ(\iota_G)_L$, which was proven in Proposition \ref{Prop:WeilRestriction}[(iv)]. The commutativity of the central and last square is a consequence of the fact that, for $\mathcal{Y} \in \Sch_L$, the morphisms $\psi^{\mathcal{Y}}$ introduced in Proposition \ref{Prop:WeilRestriction}[(ii)] are a natural transformation.
\end{proof}

%
%
%

\begin{notation}
	For a group $G/k$ acting on a variety $Y$, and a point $y \in Y(\bar{k})$, we denote by $\Stab_{\bar{G}}y$ the stabilizer of $y$ in $\bar{G}\defeq G_{\bar{k}}$.
\end{notation}

Let $G$ be an algebraic $k$-group and let $Y$ be a $G_L$-variety (i.e. an $L$-variety endowed with a left $G_L$-action). We know by Remark \ref{Rmk:Weil_action} and the following discussion that $G$ acts on $(R_{L/k}Y)$ through the diagonal embedding.

\begin{corollary}\label{Cor:RestrStabilizer}
	Keeping the notation of Proposition \ref{Prop:weilres_group}(ii), we have that, for each $\bar{x} \in (R_{L/k}Y)(\bar{k})$:
	\[
	\Stab_{\bar{G}}(\bar{x})= \cap_{\gamma \in \Gamma} \Stab_{\bar{G}}(\psi^Y_{\gamma}(x)).
	\]
	%
\end{corollary}

\begin{proof}
	This immediately follows from Proposition \ref{Prop:weilres_group}(ii).
\end{proof}

\begin{proposition}\label{Prop:HomSpace}
	Let $\pi:Z \rightarrow X$ be a finite surjective $G$-equivariant morphism among $k$-varieties endowed with a (left) $G$-action. If $X$ is a homogeneous space and $Z$ is geometrically integral, then $Z$ is a homogeneous space as well.
\end{proposition}
\begin{proof}
	Let $\bar{z} \in \bar{Z}(\bar{k})$ be any geometric point, and let $Y=\bar{G} \cdot \bar{z}$ be its $\bar{G}$-orbit. We assume that $Y$ is endowed with the $\bar{k}$-variety structure that comes from the natural isomorphism $\bar{G}/\Stab_{\bar{G}}\bar{z} \cong Y$. In particular, we have by \cite[Lemma 9.30]{Milne} that $Y$ is locally closed in $\bar{Z}$.
	
	Since $X$ is a homogeneous space, we have that $\pi(Y(\bar{k}))=\pi(\bar{G} \cdot \bar{z}(\bar{k}))=\pi(\bar{z})\cdot \bar{G}(\bar{k})={X}(\bar{k})$. Hence, the morphism $\restricts{\pi}{Y}:Y \rightarrow X$ is surjective on $\bar{k}$-points, hence dominant (by Nullstellensatz). Therefore, if $Y^c\subset \bar{Z}$ denotes the Zariski-closure of $Y$ in $\bar{Z}$ (which coincides with the Zariski-closure of $Y(\bar{k})$ in $\bar{Z}$ by Nullstellensatz), we have that $\dim X= \dim \bar{Z} \geq \dim Y^c\geq \dim Y \geq \dim X$, and, hence, since $\bar{Z}$ is irreducible, $Y^c=\bar{Z}$. 
	
	We want to show that actually $Y=\bar{Z}$. Since $Y$ is locally closed in $\bar{Z}$ and $\bar{Z}$ is reduced, it is enough, by the Nullstellensatz, to show that $\bar{Z}(\bar{k})=Y(\bar{k})$. We assume, by contradiction, that there exists a $\bar{s} \in \bar{Z}(\bar{k})  \setminus Y$. We have, as before, that $\bar{G} \cdot \bar{s}$ is dense in $\bar{Z}$. Therefore, since both $\bar{G} \cdot \bar{s}$ (which we give again a $\bar{k}$-variety structure as before) and $Y$ are constructible and dense in $\bar{Z}$, they both contain some non-empty Zariski open subset of $\bar{Z}$, and their intersection is non-empty. This is a contradiction because we assumed $\bar{s} \notin Y=\bar{G} \cdot \bar{z}$, hence $\bar{G} \cdot \bar{s}\cap \bar{G} \cdot \bar{z} =\emptyset$.
\end{proof}

%

\subsection{Proof of Theorem \ref{Thm:Homspaces}}

The lemmas that follow play a major role (especially Lemma \ref{Lem:conn_comps}) in the proof of Theorem \ref{Thm:Homspaces}.

\begin{lemma}\label{Lem:induced_action}
	Let $G$ be a connected algebraic group, $X$ be a $k$-scheme of finite type endowed with a $G$-action, and $X^0 \subset X$ be a connected component of $X$. Then the $G$-action on $X$ induces one on $X^0$ (i.e. there exists a unique action on $X^0$ that makes the embedding $X^0 \hookrightarrow X$ $G$-equivariant).
\end{lemma}
\begin{proof}
	The proof is straightforward.
\end{proof}

We recall that, if $F$ is a group acting on the right on a $k$-variety $Z$, and $G$ acts on the left with an action that commutes with the one of $F$, and $\eta \in H_{sx}^1(K,F)$, there is a natural left action of $G$ on $Z_{\eta}$, commuting with the right $F^{\eta}$-action.

\begin{lemma}\label{Lem:conn_comps}
	Let $G$ be a connected algebraic group over $K$ and $X$ a left $G$-homogeneous space, and assume that there exists a finite group scheme $F/\Spec K$, and a right $F$-torsor $\phi:Z \rightarrow X$ such that $Z$ is endowed with a left $G$-action, commuting with the $F$-action, with connected geometric stabilizers and such that $\phi$ is $G$-equivariant. Suppose that $X(\A_K)^{\acute{e}t} \neq \emptyset$. Then, there exists a $\eta \in H_{sx}^1(K,F)$ and a connected component $Z'$ of $Z_{\eta}$ such that $Z'$, endowed with the $G$-action of Lemma \ref{Lem:induced_action}, is a $G$-homogeneous space with connected stabilizers. Moreover, there is a finite subgroup $F' \subset F^{\eta}$ such that $Z' \rightarrow X$ is a right $F'$-torsor.
\end{lemma}
\begin{proof}
	By \cite[Lemma 7.1]{demarche} there exists an element $\eta \in H_{sx}^1(K,F)$ and a connected component $Z'$ of $Z_{\eta}$ such that $Z'$ is geometrically connected. Since $G$ is connected, we have by Lemma \ref{Lem:induced_action} that there is a left $G$-action on $Z'$ that makes the embedding $Z' \hookrightarrow Z_{\eta}$ $G$-equivariant. 
	
	Let us now prove that $Z'$ is a homogeneous space. We know that $X$ is a homogeneous space, and that $Z'$ is smooth (because $Z' \rightarrow X$ is étale and $X$ is smooth) and geometrically connected. Hence, since $Z' \rightarrow X$ is finite and $G$-equivariant, $Z'$ is a homogeneous space by Proposition \ref{Prop:HomSpace}. Moreover, by our assumption, the geometric stabilizers of the $G$-action are connected on $Z$, so, in particular, they are on $Z'$.
	
	Letting $F'$ be the stabilizer of $Z'$ under the $F^{\eta}$-action, the last part is straightforward.
\end{proof}

\begin{lemma}\label{Lem:existence_torsor}
	Let $X$ be a (left) homogeneous space under a connected $K$-group $G$. There exists a finite group scheme $F/\Spec K$, and a right $F$-torsor $\phi:Z \rightarrow X$ such that $Z$ is endowed with a left $G$-action with connected geometric stabilizers and such that $\phi$ and the $F$-action are $G$-equivariant.
\end{lemma}
\begin{proof}
	Let $L/K$ be a Galois extension such that there exists a point $\bar{x} \in X(L)$. Let $H=\Stab_{G_L}(\bar{x})$ and let $H^0 \le H$ be the connected component of $H$ in which lies the identity. We have, by \cite[Proposition 1.39]{Milne}, that $H^0$ is a normal subgroup of $H$. We denote by $H_f$ the (finite) quotient $H/H^0$. Let $Y\defeq G_L/H^0$.
	We have a $G_L$-equivariant morphism:
	\begin{equation}\label{Eq:torsor}
	\psi:Y=G_L/H^0 \rightarrow G_L/H \cong X_L,
	\end{equation}
	where the last isomorphism is induced from the map $G_L \rightarrow X_L$, $g \mapsto g\cdot \bar{x}$.
	The identifications of (\ref{Eq:torsor}) make $Y$ a right $H_f$-torsor over $X_L$ (see \cite[Section 3.2]{Skorobogatov}), and the right $H_f$-action commutes with the right $G_L$ action. 
	Hence the induced morphism 
	\[
	R_{\psi}:R_{L/K}Y \rightarrow R_{L/K}X_L,
	\] 
	makes $R_{L/K}Y$ a right $F \defeq R_{L/K}H_f$-torsor over $R_{L/K}X_L$. The left $F$-action commutes with the left $R_{L/K}G_L$-action on $R_{L/K}Y$ (defined as in Remark \ref{Rmk:Weil_action}). We endow $R_{L/K}Y$ with the left $G$-action given by restricting the $R_{L/K}G_L$-action to a $G$-action through the embedding $\iota_G:G \hookrightarrow R_{L/K}G_L$. 
	
	Let $\iota_X:X \rightarrow R_{L/K}X_L$ be the morphism of Proposition \ref{Prop:WeilRestriction}(iii).
	Let $Z$ be the fibered product $X\times_{R_{L/K}X_L}  R_{L/K}Y$. We notice that, since, by functoriality of $R_{L/K}$ and Proposition \ref{Prop:weilres_group}(i), $R_{\psi}$ and $\iota_X$ are both $G$-equivariant, the $k$-variety $Z$ is equipped with a left $G$-action by Proposition \ref{Prop:FiberedProduct}. Moreover, the projection $Z \rightarrow X$ can be endowed with the structure of a right $R_{L/K}H_f$-torsor over $X$ (since $Z \rightarrow X$ is just a base change of the right $R_{L/K}H_f$-torsor $R_{\phi}:R_{L/K}Y \rightarrow R_{L/K}X_L$)
.
	
	Lastly, we prove that the geometric stabilizers of $Z$ are connected. Let $\bar{z} \in \bar{Z}$ be a geometric point. Since $\bar{Z} \hookrightarrow \overline{R_{L/K}Y_K}$ (where  the morphism is $\bar{G}$-equivariant), we have that, by Corollary \ref{Cor:RestrStabilizer}, there exists a $g \in G(\bar{K})$ such that $\bar{S}\defeq \Stab_{\bar{G}}(\bar{z}) \subset g \bar{H^0} g^{-1}$, where $\bar{H^0}=H^0_{\bar{K}}$. Moreover, since $\dim \bar{S}= \dim \bar{G} - \dim \bar{Z}=\dim \bar{G}- \dim \bar{X}= \dim \bar{H^0}$, and $\bar{H^0}$ is integral and algebraic subgroups are always closed, we actually have that $\bar{S}=g\bar{H^0}g^{-1}$, which is connected.
\end{proof}


\begin{lemma}\label{Lem:KeyLemma}
	Let $X$ be a (left) homogeneous space under a connected $K$-group $G$, with linear stabilizers. Suppose there is no étale Brauer-Manin obstruction for the variety $X$, i.e. that there exists 
	\[
	(P_v)_{v \in M_K} \in X(\A_K)^{\acute{e}t,\Br}
	\]
	Then, there exists a homogeneous space $Z$ under $G$ with geometrically connected stabilizers, an adelic point $(Q_v)\in Z(\A_K)^{\Br Z}$, and a $G$-equivariant morphism $\psi:Z \rightarrow X$ such that $(\psi_v(Q_v))=(P_v)$. Moreover, $Z$ is a (right) torsor over $X$ under a finite group scheme.
\end{lemma}
\begin{proof}
	We know by Lemma \ref{Lem:conn_comps} (whose hypothesis hold by Lemma \ref{Lem:existence_torsor}) that there exists a finite group scheme $F$ and a right $F$-torsor $\psi':Z' \rightarrow X$, where $Z'$ is a homogeneous space with geometrically connected stabilizers and $\psi'$ is $G$-equivariant. Since $(P_v)_{v \in M_K} \in X(\A_K)^{\acute{e}t,\Br}$, we know that there exists an element $\eta \in H_{sx}^1(K,F)$ and an element $(Q_v)_{v \in M_K}\in Z(\A_K)^{\Br Z}$, where $Z \defeq Z'_{\eta}$, such that $(\psi_v(Q_v))_{v \in M_K}=(P_v)_{v \in M_K}$, where $\psi\defeq (\psi')^{\eta}:Z \rightarrow X$. We observe that $Z=Z'_{\eta}$ is still a $G$-homogeneous space (since it is a twist of a $G$-homogeneous space, with respect to an action that commutes with the $G$ one) and $Z \rightarrow X$ is a right torsor under $F^{\eta}$.
\end{proof}

\begin{lemma}\label{Lem:rev}
	Let $X$ and $Y$ be connected $\bar{k}$-varieties, with $Y$ simply connected, and let $y_0 \in Y(\bar{k})$. Let $\phi:\mathcal{Y}\rightarrow X \times_{\bar{k}} Y$ be an étale cover such that there exists a section $\sigma_{y_0}:X\times_{\bar{k}}\{y_0\}\rightarrow \restricts{\mathcal{Y}}{X\times_{\bar{k}}\{y_0\}}$ to the restricted cover $\restricts{\phi}{X\times_{\bar{k}}\{y_0\}}$. There exists then a unique section $\sigma:X \times_{\bar{k}} Y \rightarrow \mathcal{Y}$ to $\phi$ extending $\sigma_{y_0}$.
\end{lemma}
\begin{proof}
	We can assume, without loss of generality, that $\mathcal{Y}$ is connected (otherwise we can restrict $\phi$ to the connected component containing the image of $\sigma_{y_0}$).
	
	Let $x_0 \in X(\bar{k})$ be any point, which we are going to use as a ``basepoint''. We have a canonical embedding $\iota:\pi_1(Y,y_0) \hookrightarrow \pi_1(X\times Y, (x_0,y_0))$. Since $X$ is simply connected and we are in characteristic $0$, we have that $\pi_1(X\times Y, (x_0,y_0)) \cong \pi_1(Y,y_0) \times \pi_1(X,x_0) \cong \pi_1(Y,y_0)$ through natural isomorphisms (this follows from GAGA-like theorems, see \cite[XIII 4.6]{SGA1}, whose hypothesis hold by \cite{strong_resolution}). Hence, the natural embedding $\iota$ is an isomorphism. 
	
	Let now $P \defeq \sigma_{y_0}((x_0,y_0)) \in \mathcal{Y}$. By construction $P$ is $\pi_1(Y,y_0)$-, and, hence, $\pi_1(X\times Y, (x_0,y_0))$-invariant. By the standard theory of étale covers, this means that the connected étale cover $\mathcal{Y} \rightarrow X \times Y$ has degree $1$, and, hence, is an isomorphism. In particular, the cover $\mathcal{Y} \rightarrow X \times Y$ has a unique section.
\end{proof}

The following lemma is a slightly more general case of \cite[Prop. 5.1]{HW18}.

\begin{lemma}\label{Lem:Gscu}
	Let $G$ be a connected simply connected linear $k$-group, $X$ be a $k$-variety endowed with a $G$-action, and $\phi:Z\rightarrow X$ be an étale cover. There exists then a unique $G$-action on $Z$ such that $\phi$ is $G$-equivariant.
\end{lemma}
\begin{proof}
	It is sufficient, by Galois descent, to prove the existence and uniqueness over $\bar{k}$. So we can assume without loss of generality that $k =\bar{k}$. Let $m_X:G\times_k X \rightarrow X$ be the $G$-action on $X$. We consider the following diagram:
	\begin{equation}\label{Eq:comm_square}
	\begin{tikzcd}
	G \times_{\bar{k}} Z  \arrow[d] \arrow[r, "m_Z", dotted]  & Z \arrow[d]\\
	G \times_{\bar{k}} X  \arrow[r,"m_X"] & X
	\end{tikzcd} \quad ,
	\end{equation}

	which we would like to complete with a (unique) group action $m_Z$ on the first row that makes it commute.

	Let us consider the following commutative diagram:
	\begin{equation}\label{Diagr_grp_action1}
	\begin{tikzcd}
	G \times_{\bar{k}} Z  \arrow[d, "{(\varphi,id_G)}"'] \arrow[rd] \arrow[rr, "m_Z" description, dotted, bend right] &                                                                       & Z \arrow[ld, "\varphi"] \arrow[ll, "{\iota:=(id_Z,e)}"', hook']  &                                  \\
	G \times_{\bar{k}} X  \arrow[r, "m_X"']                                                                & X                                                                     &                                                                                               &                                 
	\end{tikzcd},
	\end{equation}
	
	We claim that there is a unique $m_Z$ that makes diagram (\ref{Diagr_grp_action1}) above commute with all but $\iota$, and such that $\iota$ is a section of it. From this, and the fact that $m_X$ is a group action, it is a straightforward verification to see that $m_Z$ is a group action itself.
	
	
	We enlarge the commutative diagram above to the following:
	
	\begin{equation}
	\begin{tikzcd}
	& W  \arrow[ld, "{(id_{G \times_{\bar{k}} Z}, \varphi)}"'] \arrow[rd, "\pi_Z"]  &                                                                                             \\
	G \times_{\bar{k}} Z  \arrow[d, "{(\varphi,id_G)}"'] \arrow[rd] \arrow[rr, "m_Z" description, dotted, bend right] &                                                                       & Z \arrow[ld, "\varphi"] \arrow[ll, "{\iota:=(id_Z,e)}"', hook'] \arrow[lu, "{f=(\iota,id_Z)}"', bend right=60] &                                  \\
	G \times_{\bar{k}} X  \arrow[r, "m_X"']                                                                & X                                                                     &                                                                                               &                                 
	\end{tikzcd},
	\end{equation}
	
	
	
	where $W := (G \times_{\bar{k}} Z ) \times_X Z$ , and hence the square $[Z,X,G \times_{\bar{k}} Z,W]$ is cartesian by definition. The existence and uniqueness of the sought morphism $m_Z$ (such that the lower trapezoid commutes and $\iota$ is a section of it) is equivalent to the existence and uniqueness of a morphism $\sigma$ such that it is a section of $(id_{G \times_{\bar{k}} Z }, \varphi)$ and such that $\pi_Z \circ \sigma \circ \iota=id_Z$. Lemma \ref{Lem:rev} implies that the existence and uniqueness of such a section is equivalent to the existence and uniqueness of a morphism $f:Z \rightarrow W$ such that $(id_{G \times_{\bar{k}} Z }, \varphi) \circ f =\iota$ and such that it is a section of $\pi_Z$. The morphism $(\iota, id_Z)$ is the unique morphism that satisfies these properties. 
	

\end{proof}

\begin{lemma}\label{Lem:Brauer_trivial}
	Let $G$ be a connected simply connected linear $k$-group and $X$ a $k$-variety endowed with a $G$-action. Let $B \in \Br X$ be an element of the Brauer group of $X$, and let $P \in X(k)$. Then, for every element $g \in G(k)$, we have an equality $B(P)=B(P\cdot g) \in \Br k$.
\end{lemma}
\begin{proof}
	We know that the Brauer group of $G$ is constant, i.e. $\Br G=\Br k$ \cite[Prop. 8.2.1]{Brauer_book}. Let $m_X:G\times_k X \rightarrow X$ be the $G$-action, let $m_P:G \rightarrow X$ denote the morphism defined by $g \mapsto m_X(g,P)= `` g \cdot P  "$, and let $B_P = (m_P)^*B \in \Br G=\Br k$. It is now immediate that, for every $g \in G(k)$, $B(m_X(g,P))=B_P(g)=B_P(e)=B$, as wished.
\end{proof}

\begin{proof}[Proof of Theorem \ref{Thm:Homspaces}]
	We start by showing that $ X(\A_K)_{\bullet}^{\acute{e}t,\Br} \subset \overline{G^{scu}(K_{S_{f}}) \cdot X(K)}$. Let $(P_v) \in X(\A_K)_{\bullet}^{\acute{e}t,\Br}$.
	
	We know by Lemma \ref{Lem:KeyLemma} that there exists a left torsor $\phi:Z \rightarrow X$, under some finite group scheme, such that $Z$ is a homogeneous space under $G$ with connected geometric stabilizers, with $\phi$ being $G$-equivariant, and such that there exists $(Q_v) \in Z(\A_K)_{\bullet}^{\Br Z}$ such that $(\phi_v(Q_v))=(P_v)$. A theorem of Borovoi and Demarche, \cite[Theorem 1.4]{borovoi}, tells us that $(Q_v) \in \overline{G^{scu}(K_{S_{f}}) \cdot Z(K)}$. Since $(\phi_v):Z(\A_K)_{\bullet}\rightarrow X(\A_K)_{\bullet}$ is continuous, this implies that $(P_v)=(\phi_v(Q_v)) \in \overline{G^{scu}(K_{S_{f}}) \cdot \phi(Z(K))}\subset \overline{G^{scu}(K_{S_{f}}) \cdot X(K)}$. 
	
	We now prove that $ X(\A_K)_{\bullet}^{\acute{e}t,\Br} \supset \overline{G^{scu}(K_{S_{f}}) \cdot X(K)}$. Since $ X(\A_K)_{\bullet}^{\acute{e}t,\Br}$ is closed, it suffices to prove that $ X(\A_K)_{\bullet}^{\acute{e}t,\Br} \supset {G^{scu}(K_{S_{f}}) \cdot X(K)}$. Let $P \in X(K)$ and $(g_v)_{v \in S_f}\in G^{scu}(K_{S_{f}})$, and let $\mathbf{P_1}=({P_1}_v)_{v \in M_K} \in X(\A_K)_{\bullet}$ be the adelic point defined as ${P_1}_v=P_v$ if $v \notin S_f$ and ${P_1}_v=g_v \cdot P_v$ if $v \in S_f$. Let $\psi:W \rightarrow X$ be a left torsor under a finite group scheme $F$. We know that there exists a twist $\psi^{\sigma}:W^{\sigma}\rightarrow X$, for some $\sigma \in H^1(K,F)$ such that $P =\psi^{\sigma}(P')$, for some $P' \in W^{\sigma}(K)$. By Lemma \ref{Lem:Gscu}, we know that there exists a right $G^{scu}$-action on $W^{\sigma}$ such that $\psi^{\sigma}$ is $G^{scu}$-equivariant. 
	
	Letting $\mathbf{P'_1}=({P'_1}_v)_{v \in M_K} \in W^{\sigma}(\A_K)_{\bullet}$ be the adelic point defined by ${P'_1}_v=P'_v$ if $v \notin S_f$ and ${P'_1}_v=g_v \cdot P'_v$ if $v \in S_f$, it follows from Lemma \ref{Lem:Brauer_trivial} that $\mathbf{P'_1} \in 
	W^{\sigma}(\A_K)_{\bullet}^{\Br W^{\sigma}}$. Since $\psi(\mathbf{P'_1})=\mathbf{P_1}$, this proves that $\mathbf{P_1} \in X(\A_K)_{\bullet}^{\psi}$. Since the argument works for any finite torsor $\psi:W \rightarrow X$, we have that $\mathbf{P_1} \in X(\A_K)_{\bullet}^{\acute{e}t,\Br}$, as wished. 	
	
	This concludes the proof of Theorem \ref{Thm:Homspaces}.
\end{proof}

\section{Compatibility of abelianization and Brauer-Manin pairing}\label{Sec:compatibility}


In this section $X$ denotes a quotient $G/H$, where $G$ is a connected $K$-group, and $H$ is a connected linear closed subgroup.

The goal of this section is to prove the following theorem with an explicit computation. 

\begin{theorem}\label{Thm:compatibility}
	Let $v \in M_K$, $x \in X(K_v)$ be a local point, and $B \in \H^2(K_v,C_X^d)$. Then one has that:
	\begin{equation}\label{Eq:2pairings}
	<x,\alpha(B)>=<\ab_{K_v}^0(x),B>,
	\end{equation}
	where:
	\begin{itemize}
		\item the first pairing is the local Brauer pairing;
		\item the second pairing is the one induced by 
		the local pairing (\ref{Mor:local_pairing}).
	\end{itemize}
\end{theorem}

For convenience of the reader we start by recalling some standard notations that we will use in the course of the computation.

For any algebraic group $\bar{H}/\bar{k}$, and any $\bar{k}$-variety $\bar{Y}$, endowed with a $\bar{H}$-action, we use the following notation (introduced by Borovoi and van Hamel \cite{BorovoivanHamel}):
\[
\operatorname{UPic}_{\bar{H}}(\bar{Y})^{1}\defeq \left\{(D, z) \in \operatorname{Div}(\bar{Y}) \times \bar{k}(H \times Y)^{*}:\left\{\begin{array}{l}{z_{h_{1} h_{2}}(y) = z_{h_{1}}\left(h_{2} \cdot y\right) \cdot z_{h_{2}}(y)} \\ {\operatorname{div}(z)=m^{*} D-\operatorname{pr}_{Y}^{*} D}\end{array}\right.\right\},
\]
where $z_h(y)$ stands for $z(h,y)$.

We have a natural morphism \(\bar{k}(Y)^{*} / \bar{k}^{*} \stackrel{d}{\rightarrow} \operatorname{UPic}_{\bar{H}}(\bar{Y})^{1}\), defined by \(d(f):=\left(\operatorname{div}(f), \frac{m^{*} f}{\operatorname{pr}_{Y}^{*} f}\right)\). Moreover, we define:
\[
\Pic_{\bar{H}}(\bar{Y})\defeq \operatorname{UPic}_{\bar{H}}(\bar{Y})^{1}/d(K^*(Y)).
\]

\smallskip

\subsection{Proof of Theorem \ref{Thm:compatibility}}\label{Ssec:Proof_theorem}


We anticipate a lemma that we are going to need in the proof.

\begin{lemma}\label{Lem:easier_calculation}
	Let $\bar{k}$ be an algebraically closed field (of characteristic $0$, as usual), and $X\defeq G/H$ a homogeneous space, where $H\subset G^{lin}$ is a subgroup of the maximal linear subgroup $G^{lin}$ of a connected $k$-group $G$. Let $G^{ant} \subset G$ be the maximal anti-affine subgroup of $G$, and let $Y \defeq X/G^{ant}$ (this makes sense as $G^{ant}$ is normal in $G$). We denote by $\alpha:X \rightarrow Y$, and by $\pi:X \rightarrow G^{ab}\defeq G/G^{lin} = X/G^{lin}$, the two natural projections. We then have that $\Pic X = \pi^* \Pic G^{ab} + \alpha^* \Pic Y$.  
\end{lemma}

\begin{proof}
	
	We consider the following commutative diagram, whose rows are exact by \cite[Prop. 3.12]{brion}:
	
	\begin{equation}\label{Comm:Brion}
	\begin{tikzcd}[column sep=normal]
	\X(G^{lin}) \arrow[r] \arrow[d] & \Pic(G^{ab}) \times \X(H)  \arrow[r] \arrow[d] & \Pic(G/H) \arrow[r] \arrow[d] & \Pic(G^{lin}) \arrow[d] \\
	\X(G^{lin}) \arrow[r]           & \X(H) \arrow[r]                           & \Pic(G^{lin}/H) \arrow[r]     & \Pic(G^{lin})          
	\end{tikzcd}.
	\end{equation}
	
	A simple diagram chasing of (\ref{Comm:Brion}) gives the following exact sequence:
	\begin{equation}\label{SES:group_fibration}
	\Pic(G^{ab})\rightarrow \Pic(X) \rightarrow \Pic (G^{lin}/H).
	\end{equation}
	
	Lemma \ref{Lem:Ausilio}, which follows, shows that the morphism 
	\[
	\Pic(Y)=\Pic(G/(G^{ant}\cdot H)) = \Pic(G^{lin}/(B\cdot H))\rightarrow \Pic(G^{lin}/H),
	\]
	where $B \defeq \Ker (G^{lin} \rightarrow G/G^{ant})=G^{lin} \cap G^{ant}$,
	is surjective, which, together with the exact sequence \ref{SES:group_fibration}, is sufficient to conclude the proof of this lemma.
	
	
	
\end{proof}

\begin{lemma}\label{Lem:Ausilio}
	Let $G$ be a connected linear $\bar{k}$-group (with $\bar{k}$ algebraically closed), let $B \subset G$ be a central algebraic subgroup, and let $H \subset G$ be an algebraic subgroup. We have that the following morphism is surjective:
	\[
	\Pic \left(\faktor{G}{B\cdot H}\right) \rightarrow \Pic(G/H).
	\]
\end{lemma}

\begin{proof}
	Let $\pi: \tilde{G} \rightarrow G$ be a central isogeny such that $\Pic (\tilde{G})=0$ (this exists by \cite[Thm 3]{popov}). Let $\tilde{H}\defeq \pi^{-1}(H)$, and $\tilde{B}\defeq \pi^{-1}(B)$. We then have the following two natural surjections (by \cite[Thm 4]{popov}):
	\[
	\X(\tilde{H}) \twoheadrightarrow \Pic(G/H), \quad \X(\tilde{B}\cdot \tilde{H}) \twoheadrightarrow \Pic \left(\faktor{G}{B\cdot H}\right).
	\]
	To conclude the proof of the lemma it is therefore enough to show that the following morphism is a surjection:
	\[
	\X(\tilde{B}\cdot \tilde{H}) \rightarrow \X(\tilde{H}).
	\]
	Hence it is enough to show that the morphism
	\begin{equation}\label{Eq:Final_Pic_thing}
		\faktor{\tilde{H}}{[\tilde{H},\tilde{H}]} \rightarrow \faktor{\tilde{B}\cdot\tilde{H}}{[\tilde{B}\cdot\tilde{H},\tilde{B}\cdot\tilde{H}]}
	\end{equation}
	is an injection. Since $\tilde{B}$ is central in $\tilde{G}$, as we now prove, we are done.
	
	We have that $[\tilde{B},\tilde{G}]\subset \Ker (\pi)$, which is finite. By connectedness of $\tilde{G}$, this implies that $[\tilde{B},\tilde{G}] = [\tilde{B},\tilde{e}]=\tilde{e}$. This concludes the proof of the lemma.
	

\end{proof}

%
%
%
%

\begin{lemma}\label{Lem:derived}
	Let $[A_1 \xrightarrow{\iota_1} \dots \xrightarrow{\iota_{N-1}} A_N]$ be a complex, with the $A_i$ belonging to some abelian category $\mathcal{C}$ and, for some integer $1 \leq n \leq N$,  $A'_n\subset A_n$ be such that $A'_n \rightarrow A_n/\iota_{n-1}(A_{n-1})$ is an epimorphism. Then the following is a quasi-isomorphism (the complexes being the horizontal ones):
	\begin{equation}
		\begin{tikzcd}[column sep=small]
		A_1 \arrow[r] \arrow[d] & \ldots & A_{n-2} \arrow[d] \arrow[r] & A'_{n-1} \arrow[d] \arrow[r] & A'_{n} \arrow[d] \arrow[r] & A_{n+1} \arrow[d] \arrow[r] & \ldots \arrow[r] & A_{N} \arrow[d] \\
		A_1 \arrow[r]   & \ldots & A_{n-2}   \arrow[r] & A_{n-1}   \arrow[r] & A_{n}   \arrow[r] & A_{n+1}   \arrow[r] & \ldots \arrow[r] & A_{N}  
		\end{tikzcd},
	\end{equation}
	where $A'_{n-1}\defeq A'_n\times_{A_n} A_{n-1} \hookrightarrow A_{n-1}$.
\end{lemma}
\begin{proof}
	This follows immediately from a diagram chasing.
\end{proof}

\begin{proof}[Proof of Theorem \ref{Thm:compatibility}]
	
	The proof will essentially follow the simple idea of making everything as explicit as possible in terms of Galois cocycles. The two expressions that arise from this computation are, respectively, \ref{Eq:Spec} for the LHS, and \ref{Eq:abelian} for the RHS of Equation (\ref{Eq:2pairings}). These two expressions are unfortunately not equal in an ``obvious'' manner. Hence, after these first two computations, the rest of the proof will be dedicated to show that the two obtained expressions are, in fact, equivalent in $\H^2(K_v,\overline{K_v}^*)$. We set $k=K_v$, and $\Gamma=\Gamma_k$.

	We recall some quasi-isomorphisms (see diagram \ref{Table1} below), borrowed from \cite{Demarche_comparison} (in the figure the complexes are the 3-term horizontal ones, and the vertical morphisms define the quasi-isomorphisms between them, and the complexes end in degree $2$), which will serve to make the isomorphism $\alpha$ mentioned above as explicit as possible. Let $ \Div^0(\overline{Z})\defeq \Ker (\Div(\overline{Z}) \rightarrow \Pic(SA_G) \xleftarrow{\sim} \Pic(G^{ab}) \rightarrow NS(\bar{G}^{ab}))$ (i.e. the kernel of that composition). 
	
	The vertical morphisms between the second, third and fourth row are the natural ones (see \cite[Sec 1.2.1]{Demarche_comparison} for the morphism $\Pic(\overline{H}') \rightarrow \Pic'(\overline{Z}/\overline{X})$), and the ones between the first and second row are recalled below (see the proof of Lemma \ref{Lem:RHS}), and they are actually isomorphisms (as proven in \cite{Demarche_comparison}). The horizontal arrows (forming the complexes) are always the natural ones, except for the morphism $\widehat{T_G} \rightarrow \operatorname{Pic}\left(\bar{G}^{\mathrm{ab}}\right)$, which factors through $\widehat{T_G} \rightarrow \operatorname{Pic}^0\left(\bar{G}^{\mathrm{ab}}\right) \rightarrow \operatorname{Pic}\left(\bar{G}^{\mathrm{ab}}\right)$ and is the one arising from the construction of the dual motive of $SA_G$ (see e.g. \cite[Sec. 1]{harariszamuely} for details on this construction).

	\begin{equation}\label{Table1}
	\begin{tikzcd}[column sep=tiny]
	Q'_X:                                          & \widehat{T_G} \arrow[rr]                                                           &  & \operatorname{Pic}\left(\bar{G}^{\mathrm{ab}}\right) \oplus \widehat{T_{G^{\mathrm{sc}}}} \oplus \widehat{Z_{\bar{H}^{\mathrm{red}}}} \arrow[rr]                                                                                            &  & \widehat{Z_{H^{\mathrm{sc}}}}                                               \\
	Q_{mix1}: \arrow[u, "\sim"]                   & \widehat{T_G} \arrow[rr] \arrow[u, "="]                                            &  & \operatorname{Pic}\left(\bar{G}^{\mathrm{ab}}\right) \oplus \operatorname{Pic}_{\overline{T_{G}}}\left(\bar{G}^{\operatorname{lin}}\right) \oplus \operatorname{Pic}_{\overline{T_{G}}}(\overline{Z^{\prime}}) \arrow[rr] \arrow[u, "\sim"] &  & \operatorname{Pic}(\overline{H^{\prime}}) \arrow[u, "\sim"]                 \\
	Q_{mix}: \arrow[u, "\sim"] \arrow[d, "\sim"'] & \widehat{T_{G}} \oplus \bar{k}(Z)^{*} / \bar{k}^{*} \arrow[rr] \arrow[d] \arrow[u] &  & \operatorname{UPic}_{\overline{T_{G}}}(\bar{Z})^{1} \arrow[rr] \arrow[d] \arrow[u]                                                                                                                                                          &  & \operatorname{Pic}(\overline{H^{\prime}}) \arrow[d, "\sim"'] \arrow[u, "="] \\
	Q_{geom}:                                     & \bar{k}(Z)^{*} / \bar{k}^{*} \arrow[rr]                                            &  & \operatorname{Div}(\bar{Z}) \arrow[rr]                                                                                                                                                                                                      &  & \operatorname{Pic}^{\prime}(\bar{Z} / \bar{X})                             
	\end{tikzcd}.
	\end{equation}
	
	From the quasi-isomorphisms above follows that 
	
	\begin{equation}\label{Confr:abelianization}
	\H^2(k, Q'_X) \cong \H^2(k,Q_{geom})\cong \Br_a(X,G),
	\end{equation}
	
	where the last isomorphism is a direct consequence of (\ref{Iso:CgeomtoBrauer}).

	We also need the following quasi-isomorphisms:
	
	\begin{equation}\label{Table2}
	\begin{tikzcd}[column sep=tiny]
	C_X^d:                                          & \widehat{T_G} \arrow[rr]                                                           &  & \operatorname{Pic}^0\left(\bar{G}^{\mathrm{ab}}\right) \oplus \widehat{T_{G^{\mathrm{sc}}}} \oplus \widehat{Z_{\bar{H}^{\mathrm{red}}}} \arrow[rr]                                                                                            &  & \widehat{Z_{H^{\mathrm{sc}}}}                                               \\
	Q^0_{mix1}: \arrow[u, "\sim"]                   & \widehat{T_G} \arrow[rr] \arrow[u, "="]                                            &  & \operatorname{Pic}^0\left(\bar{G}^{\mathrm{ab}}\right) \oplus \operatorname{Pic}_{\overline{T_{G}}}\left(\bar{G}^{\operatorname{lin}}\right) \oplus \operatorname{Pic}_{\overline{T_{G}}}(\overline{Z^{\prime}}) \arrow[rr] \arrow[u, "\sim"] &  & \operatorname{Pic}(\overline{H^{\prime}}) \arrow[u, "\sim"]                 \\
	Q^0_{mix}: \arrow[u, "\sim"] \arrow[d, "\sim"'] & \widehat{T_{G}} \oplus \bar{k}(Z)^{*} / \bar{k}^{*} \arrow[rr] \arrow[d] \arrow[u] &  & \operatorname{UPic}^0_{\overline{T_{G}}}(\bar{Z})^{1} \arrow[rr] \arrow[d] \arrow[u]                                                                                                                                                          &  & \operatorname{Pic}(\overline{H^{\prime}}) \arrow[d, "\sim"'] \arrow[u, "="] \\
	Q^0_{geom}:                                     & \bar{k}(Z)^{*} / \bar{k}^{*} \arrow[rr]                                            &  & \operatorname{Div}^0(\bar{Z}) \arrow[rr]                                                                                                                                                                                                      &  & \operatorname{Pic}^{\prime}(\bar{Z} / \bar{X})                             
	\end{tikzcd},
	\end{equation}
	
	where $ \Div^0(\overline{Z})\defeq \Ker (\Div(\overline{Z}) \rightarrow \Pic(SA_G) \xleftarrow{\sim} \Pic(G^{ab}) \rightarrow NS(G^{ab}))$ and $\operatorname{UPic}^0_{\overline{T_{G}}}(\bar{Z})^{1}=\{(D,f) \in \operatorname{UPic}_{\overline{T_{G}}}(\bar{Z})^{1}: D \in \Div^0(\overline{Z})\}$.
	
	
	The morphism $\alpha$ cited in (\ref{Eq:alpha}) is the composition $\H^2(k, C_X^d) \rightarrow \H^2(k, Q'_X) \cong \H^2(k,Q_{geom}) \cong \Br_a(X,G)$ (see \cite[Thm 2.1]{Demarche_comparison}).


	An easy computation of non-abelian Galois cohomology, using, for instance, the explicit description of non-abelian cocycles given in Proposition \ref{Prop:explicit_H0} below, gives that
		\begin{equation}\label{Abx}
		\ab^0(x)=[((\partial \bar{h}_{\sigma})^{-1}=\partial t_{\sigma}| z_{\sigma}, t_{\sigma}| \xi)] \in \H^2(k,C_X),
		\end{equation}
			where:
		
		\begin{itemize}
			\item for a non-abelian $1$-cochain $\alpha_{\sigma} \in \Fun(\Gamma, D(\overline{k}))$ (where $D$ denotes some $k$-algebraic group) we use the notation $\partial (\alpha_{\sigma})=\alpha_{\sigma \eta} \cdot  (\act{\sigma}{\alpha_{\eta}})^{-1} \cdot\alpha_{\sigma}^{-1}$, and, for a non-abelian $0$-cochain $\alpha \in D(\overline{k})$ we use the notation $\partial \alpha= \alpha^{-1}\cdot \act{\sigma}{\alpha}\cdot$;
			\item $g \in G(\overline{k})$ is such that its projection to $X$ is the point $x$;
			\item for all $\sigma \in \Gamma$, $h_{\sigma}\defeq g^{-1}\cdot \leftidx{^\sigma}{g}$, with $h_{\sigma} \in H^{red}(\overline{k})$ (we can assume wlog that $h_{\sigma} \in H^{red}(\overline{k})$ because $H^u$ and any of its twists are cohomologically trivial, see also \cite[Lem. 2.7]{Demarche_abelian});
			\item $g=\rho_G(\bar{g}) \cdot \xi$, with $\bar{g} \in G^{sc}(\overline{k})$ and $\xi \in SA_G(\overline{k})$;
			\item for all $\sigma \in \Gamma$, $h_{\sigma}=\rho_H(\bar{h}_{\sigma})\cdot z_{\sigma}$, with $z_{\sigma} \in Z_{H_{red}}(\overline{k})$, and $\bar{h}_{\sigma} \in H^{sc}(\overline{k})$;
			\item for all $\sigma \in \Gamma$, $t_{\sigma}\defeq  \bar{g}^{-1}\cdot \leftidx{^g}{\bar{h}_{\sigma}}^{-1}\cdot \leftidx{^\sigma}{\bar{g}} \in T_{G^{sc}}(\overline{k})$.
		\end{itemize}

		We observe that the following identity holds: 
		\begin{equation}\label{Eq:deltaxi}
			\partial \xi= z_{\sigma}\cdot \rho_G(t_{\sigma})^{-1}\in C^1(\Gamma, T_G).
		\end{equation}

	
	We fix the notation for the element $B \in \H^2(k,{C_{X}}^d)$ as follows. Let $\beta \in Z^2(k,Q^0_{mix})$ be such that its class in $\H^2(k,Q^0_{mix})$ corresponds to $B$ via the quasi-isomorphism between last and second-last rows in (\ref{Table2}). Using the standard notation for Galois cocycles, as defined in Subsection \ref{Ssec:cocycles}, we put:
	\begin{equation}\label{The_Brauer_element}
	\beta=((\chi_{\sigma,\eta},f_{\sigma,\eta})|(D_{\sigma},f_{\sigma})|\mathcal{L}  ), 
	\end{equation}
	
	where  $f_{\sigma,\eta}(1)=1$ (we use the notation $1$ to denote the identity element in $G$, and its projection to $Z$), after identifying $\bar{k}(Z)^{*} / \bar{k}^{*}$ with $\{f \in \bar{k}(Z)^{*}: f(1)=1\}\subset \bar{k}(Z)^*$.

	Because of Lemma \ref{Lem:easier_calculation}, applied to $Z=G/(H^u\cdot Z_{H^{red}})$, and Lemma \ref{Lem:derived} we may assume without loss of generality (after changing $\beta$ by a coborder), that $D_{\sigma} \in \pi^*\Div^0(G^{ab}) + \alpha^*\Div(Y)$.
	
	By definition of $\Div_{\overline{T_G}}(Z)$, we have that the following identities hold:
	
	\begin{align}
		\div((f_{\sigma})_t)=t^*D_{\sigma}-D_{\sigma} \quad & \forall t \in T_G(\overline{k}) \label{Div_unif1},\\
		(f_{\sigma})_{t_1t_2}=(f_{\sigma})_{t_1}\cdot t_1^*(f_{\sigma})_{t_2} \label{Div_unif2} \quad &\forall t_1,t_2 \in T_G(\overline{k}).
	\end{align}	
	
	Moreover, by definition of cocycle, we have the following identities:
	\begin{align}
		\div f_{\sigma,\eta}&=\partial D_{\sigma}, \label{Cocy1}\\
		[\restricts{D_{\sigma}}{x\cdot \bar{H'}}]&=\partial \mathcal{L},	\label{Cocy2}\\
		\chi_{\sigma,\eta}(t)\cdot \frac{t^*f_{\sigma,\eta}}{f_{\sigma,\eta}}&=(\partial f_{\sigma})_t. \label{Cocy3}
	\end{align}
	
	We will assume throughout the rest of the proof that all the specializations of functions at the specific points appearing are $\neq 0$. This may always be done without loss of generality.

	We start the computation of the RHS of (\ref{Eq:2pairings}).
	
	\begin{lemma}\label{Lem:RHS}
		The following identity holds:
		\begin{equation}\label{Eq:abelian}
		<\ab^0(x),B>=\left[ \widetilde{\chi}_{\sigma,\eta}(\xi) \cdot \left( (\leftidx{^{\sigma}}{f_{\eta}})_{z_{\sigma}}(\square)^{-1}\cdot (\leftidx{^{\sigma}}{f_{\eta}})_{t_{\sigma}}(\bigstar)\cdot \frac{\leftidx{^{\sigma}}{\epsilon_{\eta}}(\bigstar)}{\leftidx{^{\sigma}}{\epsilon_{\eta}}(t_{\sigma}\cdot \bigstar)} \right)\cdot \left(\frac{\psi((\partial \bar{h'}_{\sigma})^{-1}x)}{\psi(x)} \right)^{-1}\right],
		\end{equation}
		for any $\square \in {T_G}(\overline{k})$, $\bigstar \in {G^{sc}}(\overline{k})$, and
		for any $\tilde{\chi}_{\sigma,\eta} \in \bar{k}(SA_G)^*_{vert}/\bar{k}^*$ such that:
		\begin{itemize}
			\item $\restricts{\widetilde{\chi}_{\sigma,\eta}}{T_G}=\chi_{\sigma,\eta}$,
			\item $\div \widetilde{\chi}_{\sigma,\eta}= \partial D_{\sigma} \mod \pi^*(\Div(\bar{k}(G^{ab})^*/\bar{k}^*))$.
		\end{itemize}
	\end{lemma}
	
	\begin{proof}
		
		
		
		%

		
		
		The isomorphisms appearing between the first and second lines in the diagram (\ref{Table1}) are the following (we refer to \cite{Demarche_comparison} for the proof that these are (iso)morphisms):
		
		\hspace{-0.07\textwidth}
		\begin{minipage}[c]{0.33\textwidth}
			\begin{equation}\label{Eq1}
			\begin{matrix}
			&\operatorname{Pic}_{\overline{T_{G}}}(\overline{Z^{\prime}})  &\cong  &\widehat{Z_{H_{red}}}\\
			&[(D,f)] &\mapsto &(z \mapsto f_z(1))
			\end{matrix}
			\end{equation}
		\end{minipage}
		\begin{minipage}[c]{0.38\textwidth}
			\begin{equation}\label{Eq2}
			\begin{matrix}
			&\operatorname{Pic}_{\overline{T_{G}}}(\overline{G}^{\operatorname{lin}}) &\cong &\widehat{T_{G^{\mathrm{sc}}}}\\
			&[(D,f)] &\mapsto &\left(t \mapsto f_t(\star)^{-1}\cdot \frac{\epsilon(t \cdot \star)}{\epsilon(\star)}\right)
			\end{matrix}
			\end{equation}
		\end{minipage}
		\begin{minipage}[c]{0.33\textwidth}
			\begin{equation}\label{Eq3}
			\begin{matrix}
			&\operatorname{Pic}(\overline{H^{\prime}}) &\cong &\widehat{Z_{\overline{H}^{\mathrm{c}}}}\\
			&[D] &\mapsto &\left(z \mapsto \frac{z^*\psi}{\psi}\right)
			\end{matrix}
			\end{equation}
		\end{minipage}
		
		where $\epsilon \in \bar{k}(G^{sc})^*$ is such that $\div \epsilon = \rho_G^* D$, $\star$ is any element  in $\bar{G^{sc}}(\overline{k})$, and $\psi \in \bar{k}(\bar{H^{sc}})^*$ is such that $\div \psi = \rho_H^* D$. We denote by $\hat{z}_{\sigma},\hat{t}_{\sigma}$ and $w$ the images of respectively, $(D_{\sigma},f_{\sigma})$ under the isomorphism (\ref{Eq1}), $(D_{\sigma},f_{\sigma})$ under (\ref{Eq2}), and $\mathcal{L}$ under (\ref{Eq3}).
		
		%
		%
		We recall (one of) the construction(s) of the Cassels-Weil pairing (for semi-abelian varieties). We will take the liberty to identify, with a slight abuse of notation, a semi-abelian variety $Y$, defined over an algebraically closed field $\bar{k}$ with the $\Gamma_k$-group $Y(\bar{k})$:
		
		\begin{lemma}\label{Lem:Cassels_Weil}
			Let $S$ be a semiabelian variety over a field $k$, $e$ be the identity element in $S$, $A\defeq S^{ab}$ and $T \defeq \Ker(S \rightarrow A)$. We denote by $Z(\overline{S})$ the degree $0$ zero-cycles on $\overline{S}\defeq S \times_k \bar{k}$, and by $Z^0(\overline{S})$ the degree $0$ zero-cycles $\sum_{i=0}^N n_i(P_i)$ such that $\prod P_i^{n_i}=e \in \overline{S}$. 
			Moreover, let:
			\begin{equation}
			Z_{ab}(\overline{S})\defeq \faktor{Z(\overline{S})}{<(t\cdot P)- (P) -(t) + (e), \ t \in T(\overline{k}), \ P 	\in SA_G(\overline{k})>}
			\end{equation}
			There is a natural morphism $Z(\overline{S})\rightarrow \overline{S}$ defined by $\sum_{i=0}^N n_i(P_i) \mapsto \prod P_i^{n_i} \in \overline{S}$, which factors through $Z_{ab}(\overline{S})$. We define then
			\begin{equation}
			Z^0_{ab}(\overline{S})\defeq \Ker (Z_{ab}(\overline{S})\rightarrow \overline{S})\cong Z^0(A).
			\end{equation}
			The following morphisms of complexes are quasi-isomorphisms:
			\[
			[Z_{ab}^0(\overline{S})\rightarrow Z_{ab}(\overline{S})]\rightarrow [Z^0(\overline{S})\rightarrow Z(\overline{S})] \rightarrow [0 \rightarrow \overline{S}].
			\]
			The following morphism (that exists only in the derived category) is a quasi-isomorphism as well:
			\[
			[\bar{k}(S)_{vert}^*/\bar{k}^*\rightarrow \Div^0 (A)] \rightarrow [\widehat{T}\rightarrow A^*=\Pic^0(A)],
			\]
			\[
			(f,D) \longmapsto (\restricts{f}{T}, [D]).
			\]
			The Cassels-Tate pairing $[\widehat{T_{G}}\rightarrow A^*] \otimes^{L}[0 \rightarrow \overline{S}]\rightarrow \G_m[-1]\defeq [\bar{k}^*][-1]$ is then induced by the following pairing:
			
			\begin{equation}\label{Eq:pairing_casselsTate}
					[\bar{k}(S)_{vert}^*/\bar{k}^*\rightarrow \Div^0 (A)] \otimes [Z_{ab}^0(\overline{S})\rightarrow Z_{ab}(\overline{S})]  \rightarrow \G_m[-1],
			\end{equation}
			where
			\[
			\bar{k}(S)_{vert}^*/\bar{k}^* \otimes Z(\overline{S}) \rightarrow \bar{k}^*
			\]
			is the pairing induced by the evaluation of a function, and:
			\[
			\Div^0 (\overline{A}) \otimes Z^0(\overline{S}) \rightarrow \bar{k}^*
			\]
			is defined as:
			\[
			\Div^0 (\overline{A}) \otimes Z^0(\overline{S}) \rightarrow \Div^0 (\overline{A}) \otimes Z^0(\overline{A}) \rightarrow  \bar{k}^*,
			\]
			where the last arrow is defined as in \cite[Sec. 3.2]{Cassels_Tate}.
		\end{lemma}
			
		\begin{proof}
			When $S$ is an abelian variety, this is well-known ( see \cite[Sec. 3]{Cassels_Tate}). In the general case it follows by devissage.
%
%
%
		\end{proof}
		
		We have the following quasi-isomorphism:
		%
		%
		\begin{equation}\label{Table3}
		\begin{tikzcd}
		C_X^d:&\widehat{T_G} \arrow[r]                                          & \operatorname{Pic}^0\left(\bar{G}^{\mathrm{ab}}\right) \oplus \widehat{T_{G^{\mathrm{sc}}}} \oplus \widehat{Z_{\bar{H}^{\mathrm{red}}}} \arrow[r] & \widehat{Z_{H^{\mathrm{sc}}}}           \\
		{C'_X}^d \arrow[u, "\eta"]: &\overline{k}(SA_G)^*_{vert}/\overline{k}^* \arrow[r] \arrow[u] & \Div^0(\overline{G}^{ab})\oplus \widehat{T_{G^{\mathrm{sc}}}}\oplus \widehat{Z_{H^{red}}} \arrow[r] \arrow[u]                                   & \widehat{Z_{H^{\mathrm{sc}}}} \arrow[u]
		\end{tikzcd},
		\end{equation}
		
		where  the first vertical morphism is restriction to $T_G \subset SA_G$. We remark that technically $\eta$ is only defined in the derived category.
		
		We also have the following quasi-isomorphism:
		
		\begin{equation}\label{Table4}
		\begin{tikzcd}
		C'_X \arrow[d, "v"]: &Z_{H^{\mathrm{sc}}} \arrow[r] \arrow[d] & Z_{\bar{H}^{\mathrm{red}}}\oplus T_{G^{\mathrm{sc}}}\oplus Z^0_{ab}(SA_G) \arrow[r] \arrow[d] & Z_{ab}(SA_G) \arrow[d] \\
		C_X : &Z_{H^{\mathrm{sc}}} \arrow[r]           & Z_{\bar{H}^{\mathrm{red}}} \oplus T_{G^{\mathrm{sc}}} \arrow[r]                                & SA_G                  
		\end{tikzcd}.
		\end{equation}
		
		We notice that $[\eta([\tilde{\chi}_{\sigma,\eta}|D_{\sigma},\hat{z}_{\sigma},\hat{t}_{\sigma}|w])]=B \in \H^2(\Gamma_k, C_X^d)$, and 
		$v([(\partial \bar{h}_{\sigma})^{-1}|z_{\sigma},t_{\sigma},O|(\xi)])=\ab^0(x)$, hence, in view of (\ref{Abx}), (\ref{Eq1}), (\ref{Eq2}), (\ref{Eq3}) and Lemma \ref{Lem:Cassels_Weil}, Lemma \ref{Lem:Cup_product} (applied to $\mathcal{C}=C'_X$ and $\mathcal{C'}={C'_X}^d$) gives: 
		
		%
		
		\begin{equation*}
			\alpha \cup \ab(x)=\alpha' \cup \ab'(x)= \left[ \widetilde{\chi}_{\sigma,\eta}(\xi) \cdot \left( (\leftidx{^{\sigma}}{f_{\eta}})_{z_{\sigma}}(\square)^{-1}\cdot (\leftidx{^{\sigma}}{f_{\eta}})_{t_{\sigma}}(\bigstar)\cdot \frac{\leftidx{^{\sigma}}{\epsilon_{\eta}}(\bigstar)}{\leftidx{^{\sigma}}{\epsilon_{\eta}}(t_{\sigma}\cdot \bigstar)} \right)\cdot \left(\frac{\psi((\partial \bar{h'}_{\sigma})^{-1}x)}{\psi(x)} \right)^{-1}\right].
		\end{equation*}
		
		This concludes the proof of this lemma.
		
		
	\end{proof}

	%
	%
	To use (\ref{Eq:abelian}) for our comparison purposes we make a specific choice of $\tilde{\chi}_{\sigma,\eta}$, that will lead us to identity (\ref{Eq:ab2}) below.
	
	We notice that, since $\overline{SA_{G}}\rightarrow \overline{G^{ab}}$ defines a $\G_m^n$-torsor (for some $n \geq 0$)  on $\overline{G^{ab}}$ that is Zariski locally trivial, there exists, for each $\sigma \in \Gamma_k$, a $g_{\sigma} \in \overline{k}(SA_G)^*$ such that $g_{\sigma}(t)=(f_{\sigma})_t(1), \ \forall t \in T_G(\overline{k})$ and $\div(g_{\sigma})=D_{\sigma}-E_{\sigma}$, where $E_{\sigma} \in \pi^*\Div(\overline{G^{ab}})$. 
	
	We now define $\tilde{\chi}_{\sigma,\eta}\defeq f_{\sigma,\eta}\cdot (\partial g_{\sigma})^{-1}$. We notice that, by (\ref{Cocy3}), $f_{\sigma,\eta}\cdot (\partial g_{\sigma})^{-1}$ restricts to $\chi_{\sigma,\eta}^{-1}$ on $\overline{T_G} \subset \overline{SA_G}$.

	We have that, for any $t \in T(\bar{k})$, $\div(t^*g_{\sigma}/g_{\sigma})=t^*D_{\sigma}-D_{\sigma}=\div((f_{\sigma})_t)|_{SA_G}$
	(since $t^*E_{\sigma}=E_{\sigma}$). From this we deduce that
	
	\[t^*g_{\sigma}/g_{\sigma}=(f_{\sigma})_t|_{SA_G} \cdot (\tilde{\chi}_{\sigma})_t,\]
	
	for some $(\tilde{\chi}_{\sigma})_t(\star) \in \overline{k}(T_G\times SA_G)^*$ such that $(\tilde{\chi}_{\sigma})_t \in \overline{k}[SA_G]^*$ for every $t \in T(\bar{k})$ where the specialization makes sense. One may easily see that $(\tilde{\chi}_{\sigma})_1=1$ and $(\tilde{\chi}_{\sigma})_1(\star)=1$, from which it follows that $(\tilde{\chi}_{\sigma})_t = 1 \in \overline{k}(T_G\times SA_G)^*$.

	
	Noticing that $\partial g_{\sigma}(\xi)= \partial (g_{\sigma}(\xi))\cdot \leftidx{^\sigma}{g_{\eta}}(\leftidx{^\sigma}{\xi})^{-1}\cdot \leftidx{^\sigma}{g_{\eta}}(\xi)$, we can rewrite $\widetilde{\chi}_{\sigma,\eta}(\xi)$, up to a $2$-coborder as follows (we remind the reader on Notation \ref{Not:Delta}):
	
	\begin{align}
		\widetilde{\chi}_{\sigma,\eta}(\xi)^{-1}
		&=f_{\sigma,\eta}(\xi)\cdot \leftidx{^\sigma}{g_{\eta}}(\leftidx{^\sigma}{\xi})\cdot \leftidx{^\sigma}{g_{\eta}}(\xi)^{-1}=f_{\sigma,\eta}(\xi)\cdot (\leftidx{^\sigma}{f_{\eta}})_{(\partial \xi)^{-1}}(\xi)^{-1}
		\nonumber \\
		&=
		f_{\sigma,\eta}(\xi)\cdot (\leftidx{^\sigma}{f_{\eta}})_{\partial \xi}(\xi)^{-1}
		\triangleq f_{\sigma,\eta}(\xi)\cdot (\leftidx{^\sigma}{f_{\eta}})_{\partial \xi}(\xi)^{-1}. \label{Eq:ab2}
	\end{align}
	
	\smallskip

	To (hopefully) smoothen the computation of the LHS of (\ref{Eq:2pairings}) we introduce a second Galois pseudo-action (as defined in Section \ref{Sec:notation}) on $\bar{k}(H^{sc})^*$ and $\Div(H^{sc})$. It is defined as follows:
	
	\begin{equation}\label{Eq:Pseudo_action}
	(\leftidx{^{\sigma^{\dagger}}}{f})(x)\defeq (\leftidx{^{\sigma}} f)(\bar{h}_{\sigma}^{-1}\cdot x), \quad \leftidx{^{\sigma^{\dagger}}}D\defeq (\bar{h}_{\sigma}^{-1})^* (\leftidx{^{\sigma}}D),
	\end{equation}
	
	where $\leftidx{^{\sigma}}\star$ denotes the usual $\Gamma_k$-action on $\bar{k}(H^{sc})^*$ and $\Div(\overline{H^{sc}})$. We refer to this pesudo-action as {\it the} ``twisted'' $\Gamma_k$-pseudo-action (as it will be the only non-standard one appearing). To avoid confusion, we will use the notation $\partial ^{\dagger}$ to denote a coborder morphism taken with respect to the twisted Galois action. We notice that the restriction of the pseudo-action (\ref{Eq:Pseudo_action}) to $\bar{k}(H')^*\subset \bar{k}(H^{sc})^*$is an actual action of $\Gamma_k$. In fact, it is exactly the one obtained by pullbacking via the isomorphism $\overline{H'} \xrightarrow{g \cdot} \overline{Z_x}$, where $Z_x \defeq Z \times_X x \hookrightarrow Z$, the usual $\Gamma_k$-action on $\bar{k}(Z_x)^*$. 
	
	\begin{lemma}\label{Lem:Spec}
		The following identity holds:
		\begin{equation}\label{Eq:Spec}
		x^*\alpha=\left[f_{\sigma,\eta}(g \cdot x)^{-1}\cdot \partial^{\dagger}\left(\epsilon_{\sigma}(\xi^{-1}\bar{g}\xi x)\cdot (f_{\sigma})_{t}(\xi^{-1}\bar{g}\xi x)\cdot 
		\left(\frac{\leftidx{^{\sigma^{\dagger}}}{\psi(x)}{}}{\psi(x)} \right)^{-1} 
		\right)\right] \in \H^2(\Gamma_k,\overline{k}^*).
		\end{equation}
	\end{lemma}
	
	\begin{proof}

		It follows immediately from the fact that the quasi-isomorphism \ref{Triangle} is natural that we have the following commutative diagram:
		
		\begin{equation}\label{Comm_diag_spec}
		\begin{tikzcd}
		\tau_{ \leq 2} \mathbf{R} p_{X_{*}} \mathbf{G}_{m X} \arrow[r] \arrow[d]& C_X^d \arrow[d] \\
		\tau_{ \leq 2} \mathbf{R} p_{X_{*}} \mathbf{G}_{m Z_x} \arrow[r] & \widehat{C_{Z_x}},
		\end{tikzcd}
		\end{equation}

		where vertical morphisms are defined though pullback via the inclusion $Z_x \hookrightarrow Z$.

		We notice that we have the following isomorphism of complexes:
		
		\begin{equation}\label{Eq:Spec1}
		[\bar{k}(Z_x)^*\rightarrow \Div (Z_x) \rightarrow \Pic (\bar{H}')] \xrightarrow{g^*} [\bar{k}(H')^*\rightarrow \Div (H') \rightarrow \Pic (\bar{H}')],
		\end{equation}

		where the action on the LHS is the usual one, and the one on the LHS is the 
		twisted one.
		%
		%
		%
		It follows now from Remark \ref{Rmk:computing_spec}, in view of the commutativity of (\ref{Comm_diag_spec})  and the isomorphism (\ref{Eq:Spec1}) that,
		for any choice $(g_{\sigma}) \in \Fun(\Gamma_k,\bar{k}(H')^*)$ and $\overline{\mathcal{L}}\in \Div(\bar{H}')$ such that $\div(g_{\sigma})=g^*D_{\sigma}-\partial^{\dagger} {\overline{\mathcal{L}}}$ and $[\overline{\mathcal{L}}]=\mathcal{L} \in \Pic(\bar{H}')$, we have that $f'_{\sigma,\eta}=f_{\sigma,\eta}\cdot (\partial g_{\sigma})^{-1} \in Z^2(\Gamma_k,\bar{k}^*)$ is a cocycle representing $x^*\alpha$. 
		
		Hence, Lemma \ref{Lem:Spec} follows from the following:
		
		\[
		\div\left(\epsilon_{\sigma}(\xi^{-1}\bar{g}\xi x)\cdot (f_{\sigma})_{t}(\xi^{-1}\bar{g}\xi x)\cdot \frac{\leftidx{^{\sigma^{\dagger}}}{\psi(x)}}{\psi(x)}  \right)=	g^*D_{\sigma}-\partial^{\dagger} \mathcal{L} \in \Div(\overline{H^{\prime}}) \subset \Div(\overline{H^{sc}})
		\]
	\end{proof}
	

	
	In the calculations that follow we use the following notation: any $=$ sign with references under it stands for an equality that is justified by the operations or references under it. Most of the references will refer to Lemma \ref{Lem:Final_lemma}, which appears just after the calculations, and is basically just a collection of easy-to-prove identities.
	
	We have the following, where $\Delta=\act{\xi^{-1}}{\bar{g}}$:
	\begin{align}
		RHS&=\left(f_{\sigma,\eta}(g \cdot x)^{-1}\right)\cdot \partial^{\dagger}\left(\epsilon_{\sigma}(\xi^{-1}\bar{g}\xi x)\cdot (f_{\sigma})_{t}(\xi^{-1}\bar{g}\xi x)\cdot \left(\frac{\leftidx{^{\sigma^{\dagger}}}{\psi(x)}}{\psi(x)} \right)^{-1} \right)\\
		\underset{x\mapsto e}{=}&
		\left(\frac{(\partial \epsilon_{\sigma}) (\Delta)}{f_{\sigma,\eta}(\xi \cdot \Delta)}
		\cdot (\partial f_{\sigma})_{t}(\Delta)\right)
		\cdot	
		\left(\frac{(\leftidx{^{\sigma}}{f_{\eta}})_{t}(\Delta)}{(\leftidx{^\sigma}{f_{\eta}})_{{\leftidx{^\sigma}{t}{}}}\left(\leftidx{^\sigma}{\Delta}\cdot \bar{h}_{\sigma}^{-1}\right)}	
		\cdot 		
		\frac{ \leftidx{^{\sigma}}{\epsilon_{\eta}}(\Delta)}{\leftidx{^{\sigma}}{\epsilon_{\eta}} \left(\leftidx{^\sigma}{\Delta}\cdot \bar{h}_{\sigma}^{-1}\right)} \right) \cdot \partial^{\dagger} \left(\frac{\leftidx{^{\sigma^{\dagger}}}{\psi(x)}}{\psi(x)}\right)^{-1}
		\\
		\underset{\ref{Lem:Final_lemma}(4)}{=} & \left(\frac{f_{\sigma,\eta}( \Delta)}{f_{\sigma,\eta}(\xi \cdot \Delta)}
		\cdot (\partial f_{\sigma})_{t}(\Delta)\right)
		\cdot	
		\left( \frac{(\leftidx{^{\sigma}}{f_{\eta}})_{t}(\Delta)}{(\leftidx{^\sigma}{f_{\eta}})_{{\leftidx{^\sigma}{t}{}}}\left(\leftidx{^\sigma}{\Delta}\cdot \bar{h}_{\sigma}^{-1}\right)}	
		\cdot 		
		\frac{ \leftidx{^{\sigma}}{\epsilon_{\eta}}(\Delta)}{\leftidx{^{\sigma}}{\epsilon_{\eta}} \left(\leftidx{^\sigma}{\Delta}\cdot \bar{h}_{\sigma}^{-1}\right)}\right) \cdot \partial^{\dagger} \left(\frac{\leftidx{^{\sigma^{\dagger}}}{\psi(x)}}{\psi(x)}\right)^{-1} \\
		\underset{\ref{Lem:Final_lemma}(6)}{=} & \left(\frac{f_{\sigma,\eta}( \Delta)}{f_{\sigma,\eta}(\xi \cdot \Delta)}
		\cdot (\partial f_{\sigma})_{t}(\Delta)\right)
		\cdot	
		\left((\leftidx{^{\sigma}}{f_{\eta}})_{\partial a}(1) \cdot (\leftidx{^{\sigma}}{f_{\eta}})_{\partial \xi^{-1}}(\Delta)
		\cdot 		
		\frac{ \leftidx{^{\sigma}}{\epsilon_{\eta}}(\Delta)}{\leftidx{^{\sigma}}{\epsilon_{\eta}} \left(\leftidx{^\sigma}{\Delta}\cdot \bar{h}_{\sigma}^{-1}\right)}\right) \cdot \partial^{\dagger} \left(\frac{\leftidx{^{\sigma^{\dagger}}}{\psi(x)}}{\psi(x)}\right)^{-1}.
	\end{align}	
	
	On the other hand we have that:
	
	\begin{align}
		LHS = &\widetilde{\chi}_{\sigma,\eta}(\xi) \cdot \left( (\leftidx{^{\sigma}}{f_{\eta}})_{z_{\sigma}}(\square)^{-1} \cdot (\leftidx{^{\sigma}}{f_{\eta}})_{t_{\sigma}}(\bigstar)\cdot \frac{ \leftidx{^{\sigma}}{\epsilon_{\eta}}(\bigstar)}{ \leftidx{^{\sigma}}{\epsilon_{\eta}}(t_{\sigma}\cdot \bigstar)} \right)\cdot \left(\frac{\psi((\partial \bar{h'}_{\sigma})^{-1}x)}{\psi(x)} \right) ^{-1} \\
		\underset{(\ref{Eq:ab2})}{\triangleq} &f_{\sigma,\eta}(\xi)^{-1}\cdot (\leftidx{^\sigma}{f_{\eta}})_{\partial \xi}(\xi)\cdot \left( (\leftidx{^{\sigma}}{f_{\eta}})_{z_{\sigma}}(\square)^{-1}\cdot (\leftidx{^{\sigma}}{f_{\eta}})_{t_{\sigma}}(\bigstar)\cdot \frac{ \leftidx{^{\sigma}}{\epsilon_{\eta}}(\bigstar)}{ \leftidx{^{\sigma}}{\epsilon_{\eta}}(t_{\sigma}\cdot \bigstar)} \right)\cdot \left(\frac{\psi((\partial \bar{h'}_{\sigma})^{-1}x)}{\psi(x)} \right) ^{-1} \\
		\underset{{(\ref{Eq:deltaxi})},{(\ref{Div_unif2})}}{\underset{\star \mapsto e, \square \mapsto e}{=}} &\left(f_{\sigma,\eta}(\xi)^{-1} \cdot (\partial f_{\sigma})_{t}(1) \right) \cdot (\partial f_{\sigma})_{t}(1)^{-1} \cdot (\leftidx{^\sigma}{f_{\eta}})_{\partial \xi}(\xi) \cdot \left( (\leftidx{^{\sigma}}{f_{\eta}})_{\partial \xi^{-1}}(1)\cdot \frac{ \leftidx{^{\sigma}}{\epsilon_{\eta}}(1)}{ \leftidx{^{\sigma}}{\epsilon_{\eta}}(t_{\sigma})} \right)\cdot \left(\frac{\psi((\partial \bar{h'}_{\sigma})^{-1}x)}{\psi(x)} \right) ^{-1}.
	\end{align}



	In view of points 5., 9., 1., and 8. (appearing in the order they are being used) of Lemma \ref{Lem:Final_lemma}, we obtain $LHS \triangleq RHS$, as wished.
	
\end{proof}

\begin{lemma}\label{Lem:Final_lemma}
	The following hold:
	\begin{enumerate}
		\item The function \[ (\leftidx{^{\sigma}}{f_{\eta}})_{\partial \xi^{-1}} (\bigstar) \cdot \frac{\leftidx{^{\sigma}}{\epsilon_{\eta}}(\bigstar)}{\leftidx{^{\sigma}}{\epsilon_{\eta}}(t_{\sigma} \cdot (\leftidx{^{z_{\sigma}^{-1}}}{\bigstar} ))},\] is constant in $\bigstar \in {G^{sc}}(\overline{k})$ (we remind the reader of the convention of our notation $\leftidx{^g}{h}$, when $[H \rightarrow G]$ is a crossed module, as is, in our case $[G^{sc}\rightarrow G]$).
		\item If $\Delta \defeq \leftidx{^{\xi^{-1}}}{\bar{g}} \in {G^{sc}}(\overline{k})$, we have that:
		\[
		\leftidx{^\sigma}{\Delta}\cdot \bar{h}_{\sigma}^{-1}=t_{\sigma}\cdot \leftidx{^{z_{\sigma}^{-1}}}{\Delta}.
		\]
		\item For any $(f,D) \in \operatorname{UPic}_{\overline{T_{G}}}(\bar{Z})^{1}$, with $D \in \pi^*\Div(G^{ab})+\alpha^*\Div(Y)$, we have that $a^*f_t=f_t$ for any $a \in G^{ant}$. Moreover, for any $a \in T_G \cap G^{ant}, t \in T_G(\overline{k})$ such that $f_t(1)\neq 0$, we have that $f_a(t)=f_a(1)$.
		\item $f_{\sigma,\eta}(\bigstar) = (\partial \epsilon_{\sigma})(\bigstar)$, for any $\bigstar \in \overline{G^{sc}}$.
		\item We have the following:
		\[
		\frac{f_{\sigma,\eta}(\xi \cdot \bigstar)}{f_{\sigma,\eta}(\bigstar)} \cdot (\partial f_{\sigma})_t^{-1}(\bigstar) = \frac{f_{\sigma,\eta}(\xi)}{f_{\sigma,\eta}(1)} \cdot (\partial f_{\sigma})_t^{-1}(1)=f_{\sigma,\eta}(\xi) \cdot (\partial f_{\sigma})_t^{-1}(1),
		\]
		for any $\bigstar \in G^{sc}$ such that all the quantities appearing above are $\neq 0$.
		\item We have the following:
		\[
		\frac{(\leftidx{^{\sigma}}{f_{\eta}})_{t}(\Delta)}{(\leftidx{^\sigma}{f_{\eta}})_{{\leftidx{^\sigma}{t}{}}}\left(\leftidx{^\sigma}{\Delta}\cdot \bar{h}_{\sigma}^{-1}\right)}= (\leftidx{^{\sigma}}{f_{\eta}})_{\partial t^{-1}}(\Delta)=(\leftidx{^{\sigma}}{f_{\eta}})_{\partial \xi^{-1}}(\Delta)\cdot (\leftidx{^{\sigma}}{f_{\eta}})_{\partial a^{-1}}(\Delta)^{-1}=(\leftidx{^{\sigma}}{f_{\eta}})_{\partial a}(1) \cdot (\leftidx{^{\sigma}}{f_{\eta}})_{\partial \xi^{-1}}(\Delta) ,
		\]
		where the pseudo-action of $\Gamma_k$ on $\Delta$ is the twisted one, otherwise it is the standard $\Gamma_k$-action.
		\item For any $a \in G^{ant}\cap G^{aff}$, and any $(f,D)\in \UPic_{T_G}(Z)$ such that $D \in \pi^* \Div(G^{ab})+\alpha^*G^{aff}$, we have $f_a(\Delta)=f_a(1)$ for every $\Delta \in G^{sc}$.
		\item We have the following:
		\[
		\partial^\dagger \left(\frac{\leftidx{^{\sigma^{\dagger}}}{\psi(x)}}{\psi(x)}\right) = \frac{\psi((\partial \bar{h'}_{\sigma})^{-1}x)}{\psi(x)} ,
		\]
		where the pseudo-action of $\Gamma_k$ on the RHS is the one defined in (\ref{Eq:Pseudo_action}).
		\item We have the following equality:
		\[
		(\partial f_{\sigma})_{t}^{-1}(1) \cdot (\leftidx{^\sigma}{f_{\eta}})_{\partial \xi}(\xi) \triangleq (\leftidx{^{\sigma}}{f_{\eta}})_{\partial a}(1).
		\]
	\end{enumerate}
\end{lemma}
\begin{proof}
	\begin{enumerate}
		\item The divisor of this function is null, so the first point follows from the fact that $\bar{k}[G^{sc}]^*=\bar{k}^*$.
		\item We have
		\[
		\leftidx{^\sigma}{\Delta}\cdot \bar{h}_{\sigma}^{-1}=\leftidx{^{\leftidx{^{\sigma}}{\xi}^{-1}}}{(\leftidx{^{\sigma}}{\bar{g}}{})}{}\cdot \bar{h}_{\sigma}^{-1}=
		\leftidx{^{\leftidx{^{\sigma}}{\xi}^{-1}}}{\bar{g}}\cdot \leftidx{^{\leftidx{^{\sigma}}{\xi}^{-1}}}{\bar{g}^{-1}} \cdot \leftidx{^{\leftidx{^{\sigma}}{\xi}^{-1}}}{(\leftidx{^{\sigma}}{\bar{g}}{})}\cdot \bar{h}_{\sigma}^{-1}.
		\]
		But
		\[
		t_{\sigma}=\bar{g}^{-1} \cdot \leftidx{^{\leftidx{^{\sigma}}{g}{}}}{(\bar{h}_{\sigma}^{-1})} \cdot \leftidx{^{\sigma}}{\bar{g}}{}
		=\bar{g}^{-1}\cdot \leftidx{^{\sigma}}{\bar{g}}{} 		\cdot \leftidx{^{\leftidx{^{\sigma}}{\xi}{}}}{(\bar{h}_{\sigma}^{-1})}
		=\leftidx{^{\leftidx{^{\sigma}}{\xi}^{-1}}}{\bar{g}^{-1}} \cdot \leftidx{^{\leftidx{^{\sigma}}{\xi}^{-1}}}{(\leftidx{^{\sigma}}{\bar{g}}{})}\cdot \bar{h}_{\sigma}^{-1},
		\]
		where the third equality follows from the fact that $t_{\sigma}$ and $\leftidx{^{\sigma}}{\xi}$ commute. Hence:
		\[
		\leftidx{^\sigma}{\Delta}\cdot \bar{h}_{\sigma}^{-1}=\leftidx{^{\leftidx{^{\sigma}}{\xi}^{-1}}}{\bar{g}}\cdot t_{\sigma}=\leftidx{^{(\partial \xi )^-1}}{\Delta} \cdot t_{\sigma}=t_{\sigma}\cdot \leftidx{^{z_{\sigma}^{-1}}}{\Delta},
		\]
		as wished.
		\item We have that $\div(a^*f_t/f_t)=(a^*-\id)(t^*-\id)D=0$. We may consider then the function $A \times T \rightarrow \bar{k}[Z]^*/\bar{k}^*$, that sends $(a,t)$ to $[a^*f_t/f_t]$. If we endow $\bar{k}[Z]^*/\bar{k}^*$ with the discrete topology, and $A \times T$ with the Zariski one, than this function is clearly continuous. Since  $A \times T$ is connected, and $f_1 =1$, we conclude that $a^*f_t/f_t$ is a constant $c_{a,t}\neq 0$ for every $a,t$. Clearly the morphism $(a,t) \mapsto c_{a,t}$ is algebraic, and $c_{1,t}=1$ for every $ t \in \overline{T}$. If we fix $t$, $c_{a,t}$ is a function in $\bar{k}[G^{ant}]=\bar{k}$. We conclude that $c_{a,t}=1$ for every $a$ and $t$.
		
		For the second part, we have: $f_t(1)\cdot f_a(1)=f_t(a)\cdot f_a(1)=f_{at}(1)=f_a(t)\cdot f_t(1)$. 
		\item We have that $\div (f_{\sigma,\eta}\cdot (\partial \epsilon_{\sigma})^{-1})=0$. Since $\bar{k}[G^{sc}]^*=\bar{k}^*$ and $(f_{\sigma,\eta}\cdot (\partial \epsilon_{\sigma})^{-1})(1)=1$, we conclude the sought equality.
		\item It is enough to notice that the divisor of the LHS (as a function in $\bigstar \in G^{sc}$), which is $\restricts{(\xi^*(\partial D_{\sigma}) - t^*(\partial D_{\sigma}))}{G^{sc}}$, is trivial since $(\xi^*(\partial D_{\sigma}) - t^*(\partial D_{\sigma}))$ is the pullback of a divisor from $G^{ab}$.
		\item The first and middle equalities are a consequence of Point 3. of this lemma, the last is a consequence of the point below.
		\item The divisor of $f_a$ is $a^*D-D$. Since $a \in G^{ant}\cap G^{aff}$, we have that $\div (f_a)= O$. Hence $\restricts{f_a}{G^{sc}}$ is constant, as wished.
		\item This follows immediately by expanding the LHS. 
		\item \[
		(\partial f_{\sigma})_{t}(1) \cdot (\leftidx{^\sigma}{f_{\eta}})_{\partial \xi}(\xi)^{-1} \triangleq (\leftidx{^{\sigma}}{f_{\eta}})_{\partial a}(1)^{-1}.
		\]
		We have 
		\[
		(\partial f_{\sigma})_{t}(1) =  \partial ((f_{\sigma})_{t})(1) \cdot (\leftidx{^\sigma}{f_{\eta}})_{\leftidx{^\sigma}{t}{}}(1) \cdot (\leftidx{^\sigma}{f_{\eta}})_t (1) ^{-1}\triangleq (\leftidx{^\sigma}{f_{\eta}})_{\leftidx{^\sigma}{t}{}}(1) \cdot (\leftidx{^\sigma}{f_{\eta}})_t (1) ^{-1} = (\leftidx{^\sigma}{f_{\eta}})_{\partial t}(t),
		\]
		and
		\[
		(\leftidx{^\sigma}{f_{\eta}})_{\partial \xi}(\xi)=(\leftidx{^\sigma}{f_{\eta}})_{\partial \xi}(t)=(\leftidx{^\sigma}{f_{\eta}})_{\partial t}(t) \cdot (\leftidx{^\sigma}{f_{\eta}})_{\partial a}(t)=(\leftidx{^\sigma}{f_{\eta}})_{\partial t}(t) \cdot (\leftidx{^\sigma}{f_{\eta}})_{\partial a}(1),
		\]
		where the first and middle and equalities follow from Point 3. .
	\end{enumerate}
\end{proof}

\section{Removing places}\label{Sec:Removing_places}

In what follows we are going to give a version of Theorem \ref{Thm:Homspaces} for strong approximation outside some set of (finite) places $S$, see Theorem \ref{Thm:S-et-Obstr}. So, in what follows, we fix a number field $K$, and a set of finite places $S \subset M_K^{fin}$.

For a $K$-variety $X$ we define the $S$-modified Brauer-Manin group of $X$ to be the following:
\[
\Br^SX \defeq \Ker \left( \Br X \rightarrow \prod_{v \in S}\Br X_{K_v}  \right).
\]

 We define the Brauer-Manin set \textit{outside $S$} of $X$ as follows:
\begin{equation}
	X(\A_K^S)^{\Br^SX}\defeq 
	\begin{cases}
	\{x \in X(\A_K^S): \ <x,B>=0 \ \text{ for all } B \in \Br^SX \} \ \text{if } X(K_S) \neq \emptyset, \\
	\emptyset \ \text{otherwise}.
	\end{cases}
\end{equation}
We clearly have an inclusion $\overline{X(K)}^S\subset X(\A_K^S)^{\Br^SX}$, where $\overline{\star}^S$ denotes the closure in the $S$-adeles.

We define now the étale Brauer-Manin obstruction to strong approximation \textit{outside $S$} on $X$ as follows:
\begin{equation}
	X(\A_K^S)^{\acute{e}t,\Br^S} = \bigcap_{\substack{f:Y \xrightarrow{F} X \\  F \text{ finite} \\ \text{group scheme}}}\bigcup_{[\sigma]\in H^1(K,F)} f^{\sigma}(Y^{\sigma}(\A_K^S)^{\Br^S Y^{\sigma}}).
\end{equation}

Since, by \cite[Proposition 6.4]{demarche}, $X(\A_K^S)^{\acute{e}t,\Br^S}$ is closed, we have that $\overline{X(K)}^S\subset X(\A_K^S)^{\acute{e}t,\Br^S}$.

\begin{theorem}\label{Thm:SObstr}
	Let  $G$ be a connected $K$-group, $X$ be a $G$-homogeneous space with geometrically connected linear stabilizers, and $S$ be a finite set of places of $K$, and $S_0 \subset S \cap M_K^{fin}\defeq S_f$. We assume that the Tate-Shafarevich group $\Sha(K,G^{ab})$ is finite and that $G^{sc}(K)$ is dense in $G^{sc}(\A^S_K)$ . 
	We then have that $\overline{G^{scu}(K_{S_f\setminus S_0}) \cdot X(K) }^{S_0}=X(\A_K^{S_0})_{\bullet}^{\Br^{S_0}X}$, where $\Br^{S_0}X \subset \Br X$ denotes the kernel of $\Br X \rightarrow \prod_{v \in S_0} \Br X_{K_v}$.
\end{theorem}

\begin{remark}\label{Rmk:Quot_by_constant}
	In Theorem \ref{Thm:SObstr} above one can also substitute $\Br^{S_0}X \subset \Br X$ with its quotient by constant elements instead; i.e. the quotient $\Br^{S_0}X/\Br^S K\subset \Br X/\Br K$. Moreover, if we denote by $\Br^{S_0}_{loc}(X) \defeq \Ker \left( \Br X \rightarrow \prod_{v \in S_0}\Br X_{K_v}/\Br K_v  \right)$, we have that $X(\A_K^{S_0})^{\Br^{S_0}X}=X(\A_K^{S_0})^{\Br^{S_0}_{loc}(X)}$ (and the same holds for $X(\A_K)_{\bullet}$). This follows from the surjectivity $Br K \twoheadrightarrow \oplus_{v \in S_0} \Br K_v$, which is an immediate consequence of the Albert-Brauer-Hasse-Noether Theorem.
\end{remark}

\begin{remark}\label{Rmk:S0bigenough}
	In the case that $S_0=S \cap M_K^{fin}$, Theorem \ref{Thm:SObstr} yields the equality $\overline{X(K)}^{S_0}=X(\A_K^{S_0})_{\bullet}^{\Br^{S_0}X}$, i.e. a result of strong approximation on $X$, without the ``$G^{scu}$ error".
\end{remark}

\begin{remark}\label{...tobeadded}
	In the case that $S_0=\emptyset$, Theorem \ref{Thm:SObstr} reproves the main theorem of \cite{borovoi}. However, it does use \cite[Cor. 6.3]{Demarche_abelian}, and Theorem \ref{Thm:compatibility}, itself somehow a complement to \cite{Demarche_comparison}. As mentioned before, Demarche remarks in \cite[Rmq 6.4]{Demarche_abelian} that the compatibility of Theorem \ref{Thm:compatibility} is enough to re-prove Theorem \ref{Thm:Dem_ab} in this case.
\end{remark}

We notice that Theorem \ref{Thm:SObstr} does not follow immediately from Theorem \ref{Thm:Homspaces} by projecting on $X(\A_K^S)$, as shown in Remark \ref{Rmk:not_trivial_strong} and Proposition \ref{GBsp}.

The following will follow from Theorem \ref{Thm:SObstr} using an argument similar to that of the proof Theorem \ref{Thm:Homspaces}. We notice that, when $S_0=\emptyset$, Theorem \ref{Thm:S-et-Obstr} is a reformulation of Theorem \ref{Thm:Homspaces}.

\begin{theorem}\label{Thm:S-et-Obstr}
	Let  $G$ be a connected $K$-group, $X$ be a $G$-homogeneous space with linear stabilizers, and $S$ be a finite set of places of $K$, and $S_0 \subset S \cap M_K^{fin}$. We assume that the Tate-Shafarevich group $\Sha(K,G^{ab})$ is finite and that $G^{sc}(K)$ is dense in $G^{sc}(\A^S_K)$. 
	We then have that $\overline{G^{scu}(K_{S_f\setminus S_0}) \cdot X(K) }^{S_0}=X(\A_K^{S_0})_{\bullet}^{\acute{e}t, \Br^{S_0}}$. 
\end{theorem}

\subsection{Lemmas on complexes}

\begin{definition}\label{Def:good}
	A \textit{good} complex is a complex of commutative algebraic $K$-groups $[M_1 \xrightarrow{f_1} M_2 \xrightarrow{f_2} M_3]$ such that $M_1$ and $M_2$ are groups of multiplicative type, $M_3$ is a semi-abelian variety, $\Ker f_1$ is a finite group, and $M_i$ is in degree $i-3$.
\end{definition}

For a good complex $C$ of commutative algebraic $K$-groups, and a field $F \supset K$, we denote by $\H^i(F,C)$ the $F$-hypercohomology of the complex $C$. When $F$ is a local field, we endow the groups $\H^i(F,C)$ with their natural topologies as in \cite[Sec 5.1]{Demarche_abelian}.



We will also use the notation $\P^i(K,C) \defeq \prod'_{v \in M_F}\H^i(F_v,C)_{\bullet}$, where the restricted product is taken over $\H^i(O_v,C) \rightarrow \H^i(F_v,C)$ 
(after an implied choice of an integral model for $C$ has been made)
, and $\H^i(F_v,C)_{\bullet}$ denotes the usual hypercohomology for $v \in M_K^{fin}$ and hypercohomology modified à la Tate (as defined in \cite[p. 103]{harariszamuely}) for $v \in M_K^{\infty}$.

\begin{lemma}\label{Lem:Compactness}
	Let $C$ be a good complex. The topological group $\P^0(K,C)/\H^0(K,C)$ is quasi-compact. As a direct consequence, $\P^0(K,C)/\overline{\H^0(K,C)}$ is compact, where $\overline{\H^0(K,C)}$ denotes the closure of $\H^0(K,C)$ in $\P^0(K,C)$. 
\end{lemma}

\begin{proof}
	We follow a \textit{devissage} used by Demarche in \cite{Demarche_abelian}. We first prove the result when the complex $C$ is middle exact, i.e. if $C=[M_1 \xrightarrow{f_1} M_2 \xrightarrow{f_2} M_3]$ with $\Ker f_2=\im  f_1$.  
	
	We have a commutative diagram:
	\begin{equation}\label{Diagram:devissage1}
		\begin{tikzcd}
			\Ker f_1 \arrow[d,hook] \arrow[r] & 0 \arrow[r]\arrow[d]& 0 \arrow[d]\\
			M_1 \arrow[d] \arrow[r] &M_2 \arrow[d] \arrow[r] &M_3 \arrow[d]\\
			0 \arrow[r]& 0 \arrow[r]& \Coker(f_2).
		\end{tikzcd}
	\end{equation}
	Denoting $\Ker f_1$ by $F$ and $\Coker f_2$ by $M$, the commutative diagram (\ref{Diagram:devissage1}) induces the following distinguished triangle:
	\begin{equation}\label{Triangle:1}
		F[2] \rightarrow C \rightarrow M \rightarrow F[3],
	\end{equation}
	
	where, for an abelian group $A$, we also use the letter $A$, with a slight abuse of notation, to denote the complex $[\cdots \rightarrow 0 \rightarrow A \rightarrow 0 \rightarrow \cdots]$, where $A$ lies in degree $0$. We notice that, because of the assumption that $C$ is good, $M$ is a semi-abelian variety (being the quotient of a semi-abelian variey by a subgroup of multiplicative type) and both $F[2]$ and $M=M[0]$ are good. 

	The triangle (\ref{Triangle:1}) induces the following commutative diagram with exact rows:
	\begin{equation}\label{Commdiagr1}
		\begin{tikzcd}
		H^2(K,F) \arrow[r]\arrow[d] &\H^0(K,C) \arrow[r]\arrow[d]& H^0(K,M)\arrow[r]\arrow[d]&H^3(K,F)\arrow[d]\\
		P^2(K,F) \arrow[r] &\P^0(K,C) \arrow[r,"\alpha"]& P^0(K,M)\arrow[r]&P^3(K,F).
		\end{tikzcd}
	\end{equation}

	All the morphisms of (\ref{Commdiagr1}) are continuous by construction, and $\alpha$ is open (this easily follows by a computation of the long exact sequence associated with the distinguished triangle $\mathcal{F}[2] \rightarrow \mathcal{C} \rightarrow \mathcal{M} \rightarrow \mathcal{F}[3]$, where $\mathcal{F}, \mathcal{C}, \mathcal{M}$ are integral models of $F, C$ and $M$ over a ring of $S_0$-integers $O_{K,S_0}$, for $S_0$ sufficiently large).
	
	The Poitou-Tate theorem \cite[Thm. 17.13]{Harari_book} implies that the last vertical arrow of (\ref{Commdiagr1}) is an isomorphism of finite groups. Since $H^0(K,M) \rightarrow P^0(K,M)$ is injective, a diagram chasing of (\ref{Commdiagr1}) yields the following exact sequence:
	\begin{equation}\label{Seq1}
		{P^2(K,F)}/{H^2(K,F)} \rightarrow \P^0(K,C)/\H^0(K,C) \xrightarrow{\bar{\alpha}} P^0(K,M)/H^0(K,M) \rightarrow F',
	\end{equation}
	where $F'$ is some finite discrete group (this follows by the finiteness of $P^3(K,F)$). Moreover, all the morphisms of (\ref{Seq1}) are continuous, and $\bar{\alpha}$ is open. As proven in \cite[Lemme 4]{Harari}, the quotient $P^0(K,M)/H^0(K,M)$ is quasi-compact (actually Harari proves the compactness of $P^0(K,M)/\overline{H^0(K,M)}$, but, because of Lemma \ref{Lem:Comp_nnHausdorff}, this is equivalent to what we want). We also have that ${P^2(K,F)}/{H^2(K,F)}\cong  H^0(K,{F}^d)^D$ (as topological groups) by the Poitou-Tate exact sequence. Since $H^0(K,{F}^d)^D$ is a finite discrete set, we deduce that ${P^2(K,F)}/{H^2(K,F)}$ is one as well.  Therefore, applying Lemma \ref{Lem:Ex_seq}, we deduce that $\P^0(K,C)/\H^0(K,C)$ is quasi-compact, as wished.
	
	We now turn to the case of a general good $C$. Let $p:M_2/\im(f_1)\rightarrow P$ be an embedding of the quotient $M_2/\im(f_1)$ into a quasitrivial torus $P$, and let $C'$ be the complex $[M_1\rightarrow M_2 \rightarrow M_3 \oplus P]$, which is middle exact. We have the following distinguished triangle:
	\begin{equation}\label{Triangle2}
		P \rightarrow C' \rightarrow C \rightarrow P[1],
	\end{equation}
	which induces the following commutative diagram with exact rows and continuous morphisms:
	\begin{equation}\label{Commdiagr2}
	\begin{tikzcd}
	H^0(K,P) \arrow[r]\arrow[d] &\H^0(K,C') \arrow[r]\arrow[d]& \H^0(K,C)\arrow[r]\arrow[d]&H^1(K,P)\arrow[d]\\
	P^0(K,P) \arrow[r] &\P^0(K,C') \arrow[r]& \P^0(K,C)\arrow[r]&P^1(K,P).
	\end{tikzcd}
	\end{equation}
	Hilbert's Theorem 90 and Shapiro's lemma imply that the last column is zero, hence we get the surjectivity of:
	\[
	\P^0(K,C') \rightarrow \P^0(K,C),
	\]
	which induces a (continuous) surjective morphism $\P^0(K,C')/\H^0(K,C') \rightarrow \P^0(K,C)/\H^0(K,C)$. It follows now from the quasi-compactness of $\P^0(K,C')/\H^0(K,C')$ that $\P^0(K,C)/\H^0(K,C)$ is quasi-compact as well, thus concluding the proof.
\end{proof}

\begin{corollary}\label{Lem:SameImage}
	Let $C$ be a good complex. The groups $\P^0(K,C)$ and $(\P^0(K,C))^{\wedge}$ have same image in $\H^2(K,C^d)^D$, under the morphism $\theta:(\P^0(K,C))^{\wedge} \rightarrow \H^2(K,C^d)^D$ defined by local duality 
	(we recall that $\theta$ is the map induced from local duality as in (\ref{Mor:theta})).
\end{corollary}

\begin{proof}
	The compactness of $\P^0(K,C)/\overline{\H^0(K,C)}$ implies that its image in  $\P^0(K,C)^{\wedge}/{\H^0(K,C)}^{\wedge}$ is closed. Since it is also, by the definition of profinite completion, dense, we see that its image is the whole quotient $\P^0(K,C)^{\wedge}/{\H^0(K,C)}^{\wedge}$.
	
	The corollary now follows from the following commutative diagram:
	\begin{center}
		\begin{tikzcd}
		\H^0(K,C) \arrow[hook,r] \arrow[d] &\P^0(K,C)\arrow[r,"\theta"] \arrow[d] &\H^2(K,C^d)^D \arrow[d,equal]\\
		(\H^0(K,C))^{\wedge} \arrow[r] &(\P^0(K,C))^{\wedge}\arrow[r,"\theta"] &\H^2(K,C^d)^D
		\end{tikzcd},
	\end{center}
	
	of which both rows are complexes and the second one is exact.
\end{proof}

\begin{lemma}\label{Lem:duality}
	Let $C$ be a good complex, defined over a local field $K_v$. Then, the pairing:
	\begin{equation}\label{Eq:pairing}
		\H^0(K_v,C)^{\wedge}_{\bullet} \times \H^2(K_v,C^d)_{\bullet} \rightarrow H^2(K_v,\overline{K_v}^*) \rightarrow \Q/\Z,
	\end{equation} 
	is perfect.
\end{lemma}
\begin{proof}
	We focus on the proof for non-archimedean $v$, the proof when $v$ is archimedean follows the same pattern with regular cohomology replaced by Tate cohomology. The proof follows the same devissage as the one in Lemma \ref{Lem:Compactness}, described in diagram (\ref{Diagram:devissage1}) and (\ref{Triangle2}), of which we borrow the notation. We start with the case that $C$ is middle exact. We have the following distinguished triangles:
	\[
	F[2] \rightarrow C \rightarrow M \rightarrow F[3],
	\]
	and
	\[
	F^d[-3] \rightarrow M^d \rightarrow C^d \rightarrow F^d[-2].
	\]
	We deduce that the rows of the following commutative diagram are exact:
	\begin{equation}\label{EQ:COMMDIAGR}
		\begin{tikzcd}
		{\H^{-1}(K_v,M)} \arrow[r] \arrow[d] & {H^2(K_v,F)} \arrow[r] \arrow[d] & {\H^0(K_v,C)} \arrow[r] \arrow[d] & {\H^0(K_v,M)} \arrow[r] \arrow[d] & {H^3(K_v,F)} \arrow[d] \\
		{\H^3(K_v,M^d)^D} \arrow[r]          & {H^0(K_v,F^d)^D} \arrow[r]       & {\H^2(K_v,C^d)^D} \arrow[r]       & {\H^2(K_v,M^d)^D} \arrow[r]       & {H^{-1}(K_v,F^d)^D}  .
		\end{tikzcd}
	\end{equation}
	
	Since $H^2(K_v,F)$ and $H^3(K_v,F)$($\cong 0$) are finite groups, $\H^0(K_v,C)$ and $\H^0(K_v,M)$ are endowed with their profinite topologies, and $\H^{-1}(K_v,M)=\H^{-1}_{\wedge}(K_v,M)\cong 0$, we deduce that the first row of the following diagram is exact:
	
	\begin{equation}\label{EQ:COMMDIAGR2}
	\begin{tikzcd}
	{\H^{-1}_{\wedge}(K_v,M)} \arrow[r] \arrow[d] & {H^2(K_v,F)} \arrow[r] \arrow[d] & {\H^0(K_v,C)^{\wedge}} \arrow[r] \arrow[d] & {\H^0(K_v,M)^{\wedge}} \arrow[r] \arrow[d] & {H^3(K_v,F)} \arrow[d] \\
	{\H^3(K_v,M^d)^D} \arrow[r]          & {H^0(K_v,F^d)^D} \arrow[r]       & {\H^2(K_v,C^d)^D} \arrow[r]       & {\H^2(K_v,M^d)^D} \arrow[r]       & {H^{-1}(K_v,F^d)^D}   
	\end{tikzcd}.
	\end{equation}
	Since the first and fourth columns are isomorphisms by \cite[Thm 0.1]{harariszamuely}, and the second and fifth columns are isomorphisms by local duality, we deduce that the middle column is an isomorphism as well, concluding the proof of the middle exact case.	
	For the general case, we use the following distinguished triangle (again, we borrow the notation used in the proof of the previous lemma):
	\begin{equation}\label{Triangle2.2}
	P \rightarrow C' \rightarrow C \rightarrow P[1],
	\end{equation}
	from which we deduce the exactness of the rows of the following commutative diagram:
	\begin{equation}\label{Commdiagr2.2}
	\begin{tikzcd}
	H^0(K_v,P) \arrow[r]\arrow[d] &\H^0(K_v,C') \arrow[r]\arrow[d]& \H^0(K_v,C)\arrow[r]\arrow[d]&H^1(K_v,P)\arrow[d]\\
	H^2(K_v,P^d)^D \arrow[r] &H^2(K_v,(C')^d)^D \arrow[r]& H^2(K_v,C^d)^D \arrow[r]&H^1(K_v,P^d)^D.
	\end{tikzcd}
	\end{equation}
	Since $P$ is quasi-trivial, we deduce that the last column is $0$. Moreover, since all groups appearing on the upper row are endowed with their profinite topologies, and profinite completion is a right exact functor, we deduce the exactness of the rows of the following commutative diagram:
	\begin{equation}\label{Commdiagr2.3}
	\begin{tikzcd}
	H^0(K_v,P)^{\wedge} \arrow[r]\arrow[d] &\H^0(K_v,C')^{\wedge}  \arrow[r]\arrow[d]& \H^0(K_v,C)^{\wedge} \arrow[r]\arrow[d]&0\arrow[d]\\
	H^2(K_v,P^d)^D \arrow[r] &H^2(K_v,(C')^d)^D \arrow[r]& H^2(K_v,C^d)^D \arrow[r]&0.
	\end{tikzcd}
	\end{equation}
	Since the second column is an isomorphism by the previous case, and the first is an isomorphism by local duality for tori, we deduce that the third column is an isomorphism as well.
%
\end{proof}

\subsection{Main theorem with connected stabilizers and removed places}

\begin{proof}[Proof of Theorem \ref{Thm:SObstr}]
	
	We claim that we may assume, without loss of generality that $X(K) \neq \emptyset$. In fact, we have the following inclusion:
	\[
	X(\A_K)^{\Br^{S_0}X} \subset X(\A_K)^{\B_{\infty}(X)}.
	\]
	In particular, if $X(\A_K)^{\Br^{S_0}X} \neq \emptyset$, then $X(\A_K)^{\B_{\infty}(X)} \neq \emptyset$, and, by Theorem \ref{Thm:HPforX}, we deduce that $X(K) \neq \emptyset$. On the other hand, if $X(\A_K)^{\Br^{S_0}X} = \emptyset$, then, since $X(K) \subset X(\A_K)^{\Br^{S_0}X}$, there is nothing to prove. So this concludes the proof of the claim. From now on, we can and will assume that $X=G/H$, with $H$ linear and connected. We may and will use all the abelianization paraphernalia of sections \ref{SSec:abelianized} and \ref{SSec:abelianized2} (in particular Theorem \ref{Thm:compatibility}), of which we borrow the notation as well.
	
	We notice the following:
	\begin{equation}\label{Eq:quotient}
		\overline{X(K)}^{S_0}=\pi^S(\overline{X(K)_{S_0}}),
	\end{equation}
	where we are using the notations introduced in Section \ref{Sec:notation}.
	
	We have the following commutative diagram of Hausdorff topological spaces, where the rows are exact (the first is exact in a set-wise sense, in fact the middle term is even a direct product of the other two here) and every morphism
	is continuous: 
	\begin{equation}\label{Diagr:S-proof}
\begin{tikzcd}
X(K_{S_0}) \arrow[rr] \arrow[d, "\ab_{S_0}"]              &  & X(\A_K)_{\bullet} \arrow[rr, "\pi^{S_0}"] \arrow[d, "\ab"] \arrow[dd, "\theta", pos=1/3, bend left=49] &  & X(\A^{S_0}_K)_{\bullet} \arrow[d, "\ab^{S_0}"] \arrow[dd, "\theta^{S_0}", bend left=60] \\
{\H^0(K_{S_0},C_X)} \arrow[d, "\theta'_{S_0}"] \arrow[rr, hook, "\iota"]            &  & {\P^0(K,C_X)} \arrow[rr] \arrow[d, "\theta'"]                                    &  & {\P^0_{S_0}(K,C_X)} \arrow[d]                                                 \\
{\H^2(K_{S_0},{C_X^d})^D} \arrow[rr, "\pi_{S_0}^D"] &  & {\H^2(K,{C_X^d})^D} \arrow[rr]                                               &  & {\H^2_{S_0}(K,{C_X^d})^D}                                                
\end{tikzcd},
	\end{equation}
	where $\H^0_{S_0}(K,{C_X^d}) \defeq \Ker (\H^2(K,{C_X^d}) \rightarrow \H^2(K_{S_0},{C_X^d}))$.
	We remind the reader that there is a natural morphism ${\H^2_{S_0}(K,{C_X^d})} \xrightarrow{\alpha} \Br^{S_0}(X)$, and that this is compatible with Brauer-Manin obstruction in the sense of Theorem \ref{Thm:compatibility} . We hence have the following sequence of inclusions:
	\[
	X(\A_K)_{\bullet}^{\alpha(\H^2_{S_0}(K,{C_X^d}))} \supset X(\A_K)_{\bullet}^{\Br^{S_0}(X)}\supset  \overline{G^{scu}(K_{S_f\setminus S_0})\cdot X(K)_{S_0}}.
	\]
	
	Therefore, because of (\ref{Eq:quotient}), to prove  Theorem \ref{Thm:SObstr} it is enough to prove that $\Ker (\theta^{S_0}\circ \pi^{S_0}) = X(\A_K)_{\bullet}^{\alpha(\H_{S_0}^2(K,{C_X^d}))} \subset \overline{G^{scu}(K_{S_f\setminus S_0}) \cdot X(K)_{S_0}}$ (we use, with slight abuse of notation, the symbol $\Ker$ to denote the fiber of $0$). We have that 
	\begin{enumerate}
		\item ${\overline{G^{scu}(K_{S_f\setminus S_0})\cdot X(K)_{S_0}}} = \overline{(\Ker \theta)_{S_0}}$ by Theorem \ref{Thm:Dem_ab},
		\item ${(\Ker \theta)_{S_0}}= \ab^{-1}({(\Ker \theta')_{S_0}})$, as it easily follows from the commutativity of (\ref{Diagr:S-proof}), the fact that $\ab_{S_0}$ is surjective by \cite[Prop. 2.18]{Demarche_abelian} and Lemma \ref{Lem:Set_Lemma},
		\item $\overline{(\Ker \theta)_{S_0}} \supset \ab^{-1}\left( \overline{(\Ker \theta')_{S_0}}\right)$ by the point above, the openness  (proved in Lemma \ref{Lem:ab_continuity_openness}) of $\ab:X(\A_K)_{\bullet} \rightarrow \P^0(K,C_X)$ and Lemma 	\ref{Lem:Inverse_image_closure},
		\item $\Ker (\theta^{S_0}\circ \pi^{S_0}) = {\theta}^{-1}(\im(\pi_{S_0}^D))= \ab^{-1}({\theta'}^{-1}(\im(\pi_{S_0}^D)))$ by the commutativity of (\ref{Diagr:S-proof}) and the exactness of its third row.
	\end{enumerate}

	
	Hence, by the points above, it is sufficient that we prove that ${\theta'}^{-1}(\im(\pi_{S_0}^D))=\overline{(\Ker \theta')_{S_0}}$.

%
	
	We have the following factorization of the morphism $\theta'$:
	\begin{equation}\label{Eq:fact}
		\P^0(K,C_X) \rightarrow \faktor{\P^0(K,C_X)}{\overline{\H^0(K_{S_0},C_X)}}  \rightarrow \faktor{\P^0(K,C_X)}{\Ker \theta'} \xhookrightarrow{\theta^{\prime \prime}} \H^2(K,{C_X^d})^D.
	\end{equation}
	
	We have that:
	\begin{align}\label{Eq:almost_last_S-pf}
	\begin{split}
		\theta^{\prime}\left( \overline{\iota(\H^0(K_{S_0},C_X))\cdot (\Ker \theta')} \right)&= \theta^{\prime \prime}\left( \overline{\iota(\H^0(K_{S_0},C_X))\cdot (\Ker \theta')}/(\Ker \theta') \right)\\
	&=\overline{\theta^{\prime \prime}\left( \left( \iota(\H^0(K_{S_0},C_X))\cdot (\Ker \theta')\right) /(\Ker \theta') \right)  } \\ &=\overline{\pi_{S_0}^D\left(\theta'_{S_0}(\H^0(K_{S_0},C_X))\right)}=\pi_{S_0}^D\left(\overline{\theta'_{S_0}(\H^0(K_{S_0},C_X))}\right) \\
	&=\pi_{S_0}^D\left(\H^2(K_{S_0}, {C_X^d})^D\right)=\im(\pi_{S_0}^D),
	\end{split}
	\end{align}
	
	where the second and fourth identity follow from Lemma \ref{Lem:some_top_lemma} (whose hypothesis hold by Lemma \ref{Lem:Compactness} and Corollary \ref{Lem:SameImage} for the second identity and by the fact that the dual of a torsion group is profinite, hence compact, for the fourth identity) , the third by the commutativity of the lower-left square of (\ref{Diagr:S-proof}), and the fifth one follows from the fact that $\theta'_{S_0}$ has dense image in $\H^2(K_{S_0},{C_X^d})^D$ (by \ref{Mor:local_pairing}). Now, it easily follows from (\ref{Eq:almost_last_S-pf}) that ${\theta'}^{-1}(\im(\pi_{S_0}^D))=\overline{(\Ker \theta')_{S_0}}$.



\end{proof}

\begin{proof}[Proof of Theorem \ref{Thm:S-et-Obstr}]
	The inclusion $\overline{G^{scu}(K_{S_f\setminus S_0})\cdot X(K)}^{S_0}\subset X(\A_K^{S_0})_{\bullet}^{\acute{e}t_{S_0},\Br^{S_0}}$ follows from
	the fact that $G^{scu}(K_{S_f})\cdot X(K)\subset X(\A_K^{S_0})_{\bullet}^{\acute{e}t_{S_0},\Br^{S_0}}$ (which follows from Lemmas \ref{Lem:Gscu} and \ref{Lem:Brauer_trivial} as in the proof of Theorem \ref{Thm:Homspaces}) and the fact that the latter is closed.
	
	The inclusion $\overline{G^{scu}(K_{S_f\setminus S_0})\cdot X(K)}^{S_0}\supset X(\A_K)_{\bullet}^{\acute{e}t_{S_0},\Br^{S_0}}$ can be proven as follows. Let $\alpha \in X(\A_K)_{\bullet}^{\acute{e}t_{S_0},\Br^{S_0}}$, using Lemmas \ref{Lem:conn_comps} and \ref{Lem:existence_torsor} as in the proof of Theorem \ref{Thm:Homspaces}, we know that there is a (right) torsor $Z \xrightarrow{\phi} X$ under a finite group scheme $F$, such that $Z$ is a (left) homogeneous space under $G$ with geometrically connected stabilizers. Since $\alpha \in X(\A_K)_{\bullet}^{\acute{e}t_{S_0},\Br^{S_0}}$, we may assume, up to twisting $Z$ by some cocycle $\in H^1_{sx}(K,F)$, that there is a $\beta \in Z(\A_K)^{\Br^{S_0}Z}$ such that $\phi(\beta)=\alpha$. Since we know, by Theorem \ref{Thm:SObstr}, that $\beta \in \overline{G^{scu}(K_{S_f\setminus S_0})\cdot Z(K)}^{S_0}$, we deduce that $\alpha \in \phi\left(\overline{G^{scu}(K_{S_f\setminus S_0})\cdot Z(K)}^{S_0}\right)\subset \overline{G^{scu}(K_{S_f\setminus S_0}\cdot X(K))}^{S_0} $.
\end{proof}

%

%


With the same method of proof, one may obtain the following, which is, in some sense, a limit of Theorem \ref{Thm:S-et-Obstr} as $S_0$ grows to the whole $M_K^{fin}$ (putting $S=S_0 \cap M_K^{fin}$):
\begin{proposition}\label{Prop:HassePrinciple}
		Let  $G$ be a connected $K$-group, $X$ be a  (left)$G$-homogeneous space with linear stabilizers. We assume that the Tate-Shafarevich group $\Sha(K,G^{ab})$ is finite. We then have that $X(K)\neq \emptyset $ if and only if $X(\A_K)_{\bullet}^{\acute{e}t, \B_{\infty}(X)}\neq \emptyset$. 
\end{proposition}
\begin{proof}
	The implication $X(K)\neq \emptyset \Rightarrow  X(\A_K)_{\bullet}^{\acute{e}t, \B_{\infty}(X)}\neq \emptyset$ is clear, since $X(K) \subset X(\A_K)_{\bullet}^{\acute{e}t, \B_{\infty}(X)}$.
	On the other hand, assume that $X(\A_K)_{\bullet}^{\acute{e}t, \B_{\infty}} \neq \emptyset$, then by Lemma \ref{Lem:conn_comps} there exists a finite group scheme $F$, a right $F$-torsor $\phi:Y \rightarrow X$ such that $Y$ is (left) homogeneous space with geometrically connected stabilizers. Moreover, since there exists a $(P_v) \in X(\A_K)_{\bullet}^{\acute{e}t, \B_{\infty}}$, we may assume (up to twisting $Y$) that there exists an adelic point $(Q_v) \in Y(\A_K)^{\B_{\infty}}$. Hence, by Theorem \ref{Thm:HPforX} there exists a $Q \in Y(K)$, hence $\phi(Q)\in X(K)\neq \emptyset$. 
\end{proof}

\section{Appendix A: Topological and set-theoretic lemmas}\label{AppendixA}

We will use the following notation\footnote{We warn that this notation is similar to one defined in Section \ref{Sec:notation}, which was referring to the particular case of adele-like sets. We believe that there should be no risk of confusion, since the the one that follows is only used in this appendix.}:

\begin{notation}\label{Not:LowerS}
	Let $X$ and $A$ be non-empty sets, if $Y \subset X \times A$, we denote by $Y_A$ the set $\pi_A^{-1}(\pi_A(Y))$, where $\pi_A:X \times A \rightarrow A$ is the projection on the second factor.
\end{notation}

\begin{lemma}\label{Lem:Set_Lemma}
	Let $X,Y,A,A'$ be non-empty sets and assume we have functions $f:X \rightarrow Y$, $p:A' \rightarrow A$, and a subset $Z \subset Y \times A$. If $p$ is surjective, we have that:
	\[
	((f \times p)^{-1}(Z))_{A'}=(f \times p)^{-1}(Z_A),
	\]
	where we are using the notation \ref{Not:LowerS}.
\end{lemma}
\begin{proof}
	The proof is straightforward.
%
\end{proof}

%
%
%

\begin{lemma}\label{Lem:Inverse_image_closure}
	Let $f:X \rightarrow Y$ be an open morphism of topological spaces. We have that, for any subset $Z \subset Y$, $\overline{f^{-1}(Z)}\supset f^{-1}(\overline{Z})$.
\end{lemma}
\begin{proof}
	For any $U\subset X$ disjoint from $f^{-1}(Z)$, the image $f(U) \subset Z$ is open and disjoint from $Z$, hence from $\overline{Z}$. Unraveling the definitions, the lemma follows.
\end{proof}

\begin{lemma}\label{Lem:some_top_lemma}
	Let $\alpha:X \rightarrow Y$ be a continuous map of topological spaces, with $X$ compact and $Y$ Hausdorff. Then, for any subset $S \subset X$ we have that $\overline{\alpha(S)}=\alpha(\overline{S})$.
\end{lemma}
\begin{proof}
	This is common knowledge.
\end{proof}

\begin{lemma}\label{Lem:Comp_nnHausdorff}
	Let $B$ be a topological abelian group (not necessarily Hausdorff), let $0 \in B$ be the unit element, and let $D=\overline{\{0\}}$ be its closure. Then, $B$ is quasi-compact if and only if the quotient $B/D$, which is Hausdorff, is compact.
\end{lemma}
\begin{proof}
	If $B$ is quasi-compact, then $B/D$, being a quotient of it, is clearly quasi-compact as well.
	
	If $B/D$ is compact, we are going to prove compactness of $B$ by showing that, if $\mathcal{C}=\{C_i\}_{i \in I}$ is a collection of closed subsets of $B$ such that $\cap_{i \in I} C_i =\emptyset$, then there exists a finite subset of indexes $I_0 \subset I$ such that $\cap_{i \in I_0} C_i=\emptyset$. In fact we notice that, whenever $P \in C_i$, then $\pi^{-1}(\pi(P))=P+D=\overline{P} \subset C_i$, where $\pi:B \rightarrow B/D$ denotes the projection. Hence, for each $i \in I$, we have that $C_i=\pi^{-1}(\pi(C_i))$. Therefore, $\cap_{i \in I} \pi(C_i) =\emptyset$, and there exists a finite $I_0 \subset I$ such that $\cap_{i \in I_0} \pi(C_i) =\emptyset$. It follows that $\cap_{i \in I_0} C_i=\cap_{i \in I_0} \pi^{-1}(\pi(C_i)) = \pi^{-1}(\cap_{i \in I_0} \pi(C_i))=\emptyset$, as wished.
\end{proof}

\begin{lemma}\label{Lem:Sum_top}
	Let $B$ be a topological group, and let $D$ and $K$ be two subsets of $B$, where $D$ is closed and $K$ is quasi-compact. Then the sum $D+K$ is closed in $B$.
\end{lemma}
\begin{proof}
	Since $D$ is closed, we know that the topological quotient $B/D$ is Hausdorff. Let $\pi:B \rightarrow B/D$ be the projection. We have that $D+K=\pi^{-1}(\pi(K))$. Since $K$ is quasi-compact, $\pi(K)\subset B/D$ is as well. Hence, since $B/D$ is Hausdorff, $\pi(K)$ is closed. Therefore $\pi^{-1}(\pi(K))$, i.e. $D+K$ is closed as well.
\end{proof}


\begin{lemma}\label{Lem:Ex_seq}
	Let $A \xrightarrow{\phi} B \xrightarrow{\psi} C \xrightarrow{\alpha} F$ be an exact sequence of topological groups, where all the morphisms are continuous. Assume that $F$ is finite and discrete, $\psi$ is open, and $A$ and $C$ are quasi-compact. Then $B$ is quasi-compact as well.
\end{lemma}
\begin{proof}
	First of all we may assume, up to changing $C$ with $\Ker \alpha$, without loss of generality, that $F=\{1\}$.
	
	Let us now prove that the morphism $\psi$ is closed. Let $D \subset B$ be a closed subset, and let $D' \defeq A'+D$ be the sum of the image $A'$ under $\phi$ of $A$ and $D$. Since $A$ is quasi-compact, so is $A'$. Hence, by Lemma \ref{Lem:Sum_top}, $D' \defeq A'+D$ is closed. Since $\psi(D)=\psi(D')$, and $\psi(B \setminus D') = C \setminus \psi(A'+D)=C \setminus \psi(D)$, and the former is open, one has that $\psi(D)$ is closed, as wished.
	
	We now prove the compactness of $B$. Let $I$ be a set of indices, and $B= \cup_{i\in I} U_i$ be a covering of $B$. Let, for each $c \in C$, $\{U^c_1, \dots, U^c_{n(c)} \}$ be a finite subcovering of $\mathcal{U}$ such that $\psi^{-1}(c) \subset \cup_{i=1}^{n(c)} U^c_i$ (the subcovering may always be assumed to be finite since $\psi^{-1}(c)$, being a translate of $A'$ is quasi-compact). Let now, 
	\[
	V_c \defeq \left\{c' \in C \mid \psi^{-1}(c') \subset \bigcup_{i=1}^{n(c)} U^c_i \right\}=C \setminus \psi\left(B\setminus\cup_{i=1}^{n(c)}U_i^c\right),
	\]
	which is open (since $\psi$ is closed).
	
	Since $C$ is compact and $\cup_{c \in C}V_c=C$, there exist a finite number $n \in \N$ and $c_1,\dots,c_n \in C$ such that $\cup_{j=1}^nV_{c_j}=C$. It is straightforward to verify that then $\cup_{j=1}^n\cup_{i=1}^{n(c_j)}U_i^{c_j}=B$.
\end{proof}

\section{Appendix B: Hasse principle}\label{AppendixB}

The following result is basically already present in \cite[Appendix A]{BoCTSko}. However, since it is not explicitly stated there, we represent it here for completeness, with a proof that is just a simplified version of the proof of \cite[Theorem A.1]{BoCTSko}.

\begin{theorem}\label{Thm:HPforX}
	Let $K$ be a number field and $X$ be a left homogeneous space under a connected algebraic group $G/K$, satisfying $\Sha(K,G^{ab})$ finite. We assume that the $G$-action on $X$ has connected geometric stabilizers. Let $\B_{\infty}(X) \subset \Br_{1,ur}(X)$ be defined as:
	\begin{equation}\label{Eq:Br_loc_na_const}
		\B_{\infty}(X) \defeq \Ker \left(\Br_{1}(X) \rightarrow \prod_{v \in M_K^{fin}}\Br (X_{K_v})/\Br K_v\right).
	\end{equation}
	We then have that $X(K)\neq \emptyset$ if and only if $X(\A_K)^{\B_{\infty}(X)}\neq \emptyset$.
\end{theorem}

\begin{remark}\label{Rmk:HPforX}
	\begin{enumerate}
		\item For any smooth geometrically connected variety $X/K$ (\ref{Eq:Br_loc_na_const}) describes the Brauer set that is locally constant on non-archimedean places. We notice that, if $B \in \Br(X)$ is locally constant for all $v \notin S$ (with $S$ finite), then, since there exists a smooth model $\mathcal{X}$ for $X$ over some $\Spec O_{K,S'}$ (with $S' \supset S$ finite), such that $B \in \Br(\mathcal{X})$, and, by enlarging $S'$, we may assume, by Lang-Weil estimates, that $\mathcal{X}(O_v)\neq \emptyset$ for all $v \notin S \cup M_K^{\infty}$, we have that $B$ is necessarily $0$ for all $v \notin S \cup M_K^{\infty}$ (since it is constant, and its value on the integral points is automatically $0$).
		\item We notice that $\B_{\infty}(X)$ differs from the classical $\B(X)$ just by the behaviour at $M_K^{\infty}$. In particular, if $K$ is totally imaginary, then $\B_{\infty}(X)=\B(X)$. So, in this case, Theorem \ref{Thm:HPforX} reduces to \cite[Thm 3.4]{BoCTSko}, so it is already stated explicitly in the paper \cite{BoCTSko}.
		\item We notice that, by \cite[Thm. 2.1.1]{harari1994} and the first point of this remark, for any smooth geometrically connected variety $Y/K$, all elements of $\B_{\infty}(Y)$ are unramified elements of $\Br(Y)$.
	\end{enumerate}
\end{remark}

\begin{proof}
	As said before, we follow step-by-step the reductions of \cite[Theorem A.1]{BoCTSko}. 
	
	If $X(K) \neq \emptyset$, then $\emptyset \neq X(K)\subset X(\A_K)^{\B_{\infty}(X)}$. So we focus now on proving the other direction. Namely, we assume that $X(\A_K)^{\B_{\infty}(X)}\neq \emptyset$.
	
	We do a first reduction to show that it is enough to prove the result for $G$ such that $G^{lin}$ is reductive.
	
	Let $Y \defeq G^u\backslash X$, $G'\defeq G/G^u$, so that $Y$ is a $G'$ homogeneous space (this is because $G^u$ is normal in $G$). Let $\phi_G:G \rightarrow G'$ be the standard projection. We have a canonical morphism $\phi:X \rightarrow Y$, that is $\phi_G$-equivariant. We assume that we already know the result for $(Y,G')$. We notice that $\phi^*\B_{\infty}(Y)\subset \B_{\infty}(X)\subset \Br_{1}X$. In particular, if $(x_v) \in X(\A_K)^{\B_{\infty}(X)}$, then $(\phi(x_v)) \in Y(\A_K)^{\B_{\infty}(Y)}$. So, we deduce, from the reduction assumption, that $Y(K)\neq \emptyset$. Let $y_0 \in Y(K)$. We then have that $X_{y_0} \rightarrow y_0 \cong \Spec K$ is a homogeneous space of $G^u$. In particular, by \cite[Lem 3.2]{Borovoi1995TheBO}, we deduce that $X_{y_0}(K) \neq \emptyset$, concluding the proof of this reduction step.
	
	For the second reduction step, we know by \cite[Proposition 3.1]{BoCTSko} that there exists a group $\tilde{G}$ such that $X$ may be regarded as a homogeneous space under $\tilde{G}$, with linear connected stabilizers and such that $\tilde{G}^{ss}$ is semisimple simply connected. Moreover, by \cite[Lemma A.3]{BoCTSko} we still have that $\Sha(K,\tilde{G}^{ab})$ is finite. So we can and will assume from now on that $G$ is such that $G^u=\{1\}$, $\tilde{G}^{ss}$ is semisimple simply connected, and the geometric stabilizer $\bar{H}\subset G^{lin}$ is connected.
	
	The homogeneous space $X$ defines a $K$-form $M$ of $\bar{H}^{mult}=\bar{H}/\bar{H}^{ssu}$ (see \cite[Sec. 4.1]{Borovoi1995TheBO}), the largest quotient of $\bar{H}$ of multiplicative type, and a natural homomorphism $\chi_X:M \rightarrow G^{sab}\defeq G/G^{ss}$.

	We treat the case where $\chi_X$ is injective first (i.e. $\bar{H}^{ssu}=\bar{G} \cap \bar{G}^{ss}$). In this case, let $Y' \defeq X/G^{ss}$, and $\psi:X \rightarrow Y'$ be the standard projection. We have that $Y'$ is a homogeneous space with linear stabilizers under $G^{sab}=G/G^{ss}$, a semi-abelian variety, and is therefore, a torsor under a semi-abelian variety $G''$ itself. Moreover, since $Y'$ has linear stabilizers by the $G^{sab}$-action, we have that $G^{ab}\cong (G'')^{ab}$, and, consequently $\Sha((G'')^{ab})\cong \Sha(G^{ab})$, so that $\Sha(G'')$ is finite. 
	
	Let $(x_v) \in X(\A_K)^{\B_{\infty}(X)}$, and let $\mathcal{U}\subset X(\A_K)^{\B_{\infty}(X)}$ be the open subset defined as $\prod_{v \in M_K^{fin}}' X(K_v) \times \prod_{v \in M_K^{\infty}} C_{x_v}$, where $\prod'$ denotes the usual restricted product defining the adele sets (see Section \ref{Sec:notation}), and $C_{x_v} \subset X(K_v)$ denotes, for an archimedean $v$, the connected component in which lies ${x_v}$. Since the Brauer-Manin pairing is constant on the connected components of the archimedean places, that $\mathcal{U} \subset X(\A_K)^{\B_{\infty}(X)}$. We have by \cite[Lem. A.2]{BoCTSko}, that, for each $v \in M_K^{\infty}$, $\psi(C_{x_v})\subset Y(K_v)$ is a connected component of $Y(K_v)$. So, if we define $\mathcal{V} \subset Y(\A_K)$ to be $\prod_{v \in M_K^{fin}}' Y(K_v) \times \prod_{v \in M_K^{\infty}} \psi(C_{x_v})$, we have that $\emptyset \neq \mathcal{V} \subset Y(\A_K)^{\B_{\infty}(Y)}$ (since $\psi^* \B_{\infty}(Y) \subset \B_{\infty}(X)$) and, by \cite{Harari_torsors}, there exists a $y_1 \in \mathcal{V}\cap Y(K)$. The fiber $X_{y_1}$ is a homogeneous space under the semisimple simply connected group $G^{ss}$, with geometric stabilizers isomorphic to $\bar{H}^{ssu}$. Moreover, by construction, $X_{y_1}$ has real points in all real places. Hence, by \cite[Cor. 7.4]{borovoi_abelian}, there exists a rational point $x' \in X_{y_1}(K)$, concluding this case.
	
	We turn now to the general case.

	We construct as in the proof of \cite[Theorem A.5]{BoCTSko} a quasi-trivial torus $P$, a $P$-torsor $\phi:Z \rightarrow X$, such that $Z$ is a $G \times P$ homogeneous space, and $\phi$ is equivariant by $\pi:G \times P \rightarrow G$. Moreover, as in \cite[Theorem A.5]{BoCTSko}, we may and will assume that $(Z,G \times P)$ satisfies all of the reductions above, that the geometric stabilizers are still isomorphic to $\bar{H}$, and that the homomorphism $M \cong M_Z \rightarrow (G \times P)^{sab}\cong G^{sab} \times P$ is injective (here $M_Z$ denotes the $K$-form of $\bar{H}^{mult}$ defined by $Z$, which happens to be, in this case, isomorphic to $M$). Since, by Hilbert Theorem 90 and Shapiro's Lemma, the $P_{K_v}$-torsors $Z_{x_v} \rightarrow x_v$ are trivial for each $v \in M_K$, there exists an adelic point $(z_v) \in Z(A_K)$ such that $(\phi(z_v))=(x_v)$. Moreover, by Lemma \ref{Lem:pullback}, $\phi^*:\B_{\infty}(X) \rightarrow \B_{\infty}(Z)$ is an isomorphism, hence, $(\phi(z_v)) \in Z(\A_K)^{\B_{\infty}(Z)}$. Since $Z$ satisfies the assumption of the previous case, we already know that there exists a point $z_0 \in Z(K)$. In particular, $\phi(z_0) \in X(K)$, from which we conclude.
\end{proof}

The following lemma is a slightly modified version of \cite[Lemma A.4]{BoCTSko}:

\begin{lemma}\label{Lem:pullback}
	Let $\phi:Z \rightarrow X$ be a torsor under a quasi-trivial torus $P$, where $Z$ and $X$ are smooth geometrically connected varieties over a number field $K$. Then there is an induced homomorphism $\phi^*:\B_{\infty}(X)\rightarrow \B_{\infty}(Z)$ and it is an isomorphism.
\end{lemma}
\begin{proof}
	Let $\phi^c:Z^c \rightarrow X^c$ be smooth compactifications of $\phi$, $Z$ and $X$. We have the following commutative diagram (see Remark \ref{Rmk:HPforX}(ii) for the rows), where the columns are defined by the pullback $(\phi^c)^*$:
	\begin{equation}\label{Eq:diagram}
		\begin{tikzcd}
		0 \arrow[r] \arrow[d] & \B_{\infty}(X) \arrow[r] \arrow[d] & \Br_1(X^c) \arrow[r] \arrow[d, "\sim"] &  \prod_{v \in M_K^{fin}}\Br_1 (X^c_{K_v})/\Br K_v \arrow[d, "\sim"] \\
		0 \arrow[r]           & \B_{\infty}(Z) \arrow[r]           & \Br_1(Z^c) \arrow[r]                   &  \prod_{v \in M_K^{fin}}\Br_1 (Z^c_{K_v})/\Br K_v    ,            
		\end{tikzcd}
	\end{equation}
	where the last two columns are isomorphisms by \cite[Lemma A.4]{BoCTSko}. Hence the morphism in the first column is an isomorphism as well, concluding the proof of the lemma.
\end{proof}

\section{Appendix C: 2-torsors}\label{AppendixC}

In this appendix, let $\iota:H \hookrightarrow G$ be an embedding of connected reductive groups, defined over a field $k$ of characteristic $0$, with $H$ linear. We use the notation $C_H \defeq [H^{sc} \rightarrow H^{red}]$, and $C_G \defeq [G^{sc} \rightarrow G^{red}]$.

We recall the following definition (present e.g. in \cite{Demarche_abelian}):
\begin{definition}\label{Def:H0abDemarche}
	The set $H^0(k,[C_H \rightarrow C_G])$ is the set of couples $(\mathcal{D}, r: \mathcal{D} \bigwedge^{\mathcal{H}}\mathcal{G} \xrightarrow{\sim} \mathcal{G}) \in \TORS(\mathcal{H}) \times \Mor(\TORS(\mathcal{G}))$ up to the following equivalence. Two elements $(\mathcal{D}_1, r_1: \mathcal{D}_1 \bigwedge^{\mathcal{H}}\mathcal{G} \xrightarrow{\sim} \mathcal{G})$ and $(\mathcal{D}_2, r_2: \mathcal{D}_2 \bigwedge^{\mathcal{H}}\mathcal{G} \xrightarrow{\sim} \mathcal{G})$ are equivalent  if there exists a morphism $\Mor(\TORS(\mathcal{H})) \ni s: \mathcal{D}_1 \xrightarrow{\sim} \mathcal{D}_2$ and a $2$-morphism $\alpha \in \Mor^2(\TORS(\mathcal{G}))$:
	\begin{equation*}
		\xymatrix{
			& \mathcal{G}&  \\
			\mathcal{D}_2 \bigwedge^{\mathcal{H}}\mathcal{G} \ar[ru]^{r_2} & \ar@{=>}[u]_{\alpha} & \mathcal{D}_1 \bigwedge^{\mathcal{H}}\mathcal{G}. \ar[ll]^{s \bigwedge^{\mathcal{H}}\mathcal{G}} \ar[lu]_{r_1}\\
		}
	\end{equation*}
\end{definition}

The formulas appearing in Proposition \ref{Prop:explicit_H0} below are just the ones coming from making Definition \ref{Def:H0abDemarche} explicit using cocycle formulas, such as the ones that one may find in \cite[Sec. 6]{Breen} (see especially subsections $6.2$ and $6.3$ of {\itshape loc. cit.}). 

Before coming to Proposition \ref{Prop:explicit_H0}, giving an explicit description of the set $H^0(k,[C_H \rightarrow C_G])$ in terms of cocyles, we recall the following notation (which is in essence borrowed from \cite{Borovoi_abelianization}).

We define $T \defeq \Spec k$ and $S \defeq \Spec \bar{k}$.
We are going to denote by $S^n$ the scheme $S \times_T \ldots \times_T S$, where the product is taken $n$ times. 

We recall that we have isomorphisms $\phi_n:S^n_T \xrightarrow{\sim} \bigsqcup_{(\gamma_1,\ldots,\gamma_{n-1}) \in \Gamma^{n-1}} S$ defined as follows:
\[
\restricts{\phi_n^{-1}}{S_{(\gamma_1,\ldots,\gamma_{n-1})}}: S \rightarrow S^n, \ \ (\id, \gamma_1, \gamma_1\cdot \gamma_2,\ldots, \gamma_1  \cdots  \gamma_{n-1});
\]
where $S_{(\gamma_1,\ldots,\gamma_{n-1})}$ denotes the copy of $S$ in $\bigsqcup_{(\gamma_1,\ldots,\gamma_{n-1}) \in \Gamma^{n-1}} S$ indexed by $(\gamma_1,\ldots,\gamma_{n-1}) \in \Gamma^{n-1}$, and $\gamma_i:S \xrightarrow{\gamma_i} S$ denotes the $\Spec k$-morphism $\Spec (\Spec \act{\gamma_i}{\star}:\Spec \overline{k} \rightarrow \Spec \overline{k})$. We notice that $\phi_n$ is a morphism of $S$-schemes, if we endow $\bigsqcup_{(\gamma_1,\ldots,\gamma_{n-1}) \in \Gamma^{n-1}} S$ with its natural $S$-scheme structure, and $S^n$ with the $S$-scheme structure coming from the projection on the first coordinate.

For an a $T$-scheme $Y$, sometimes we denote, with a slight abuse of notation, the elements of $Y(S^n)$ by $y_{\gamma_1,\ldots, \gamma_{n-1}}$, to actually denote the composition:
\[
S^n_T \xrightarrow{\phi_n} \bigsqcup_{(\gamma_1,\ldots,\gamma_{n-1}) \in \Gamma^{n-1}} S \xrightarrow{y_{\gamma_1,\ldots, \gamma_{n-1}}} Y.
\]

For a crossed module $\mathcal{C} \defeq [C \xrightarrow{\rho} P]$:

We define $C^0_{\mathcal{C}}(S) \defeq P(S^2) \times C(S^3)$, $C^1_{\mathcal{C}}(S) \defeq P(S) \times C(S)$, and $C^2_{\mathcal{C}}(S) \defeq C(S)$. According to the notation above, we denote elements of $C^0_{\mathcal{C}}(S)$ (resp. $C^1_{\mathcal{C}}(S), C^2_{\mathcal{C}}(S)$) by $(p_{\sigma}, c_{\sigma,\eta})$ (resp. $(\bar{p}, \bar{c}_{\sigma}), \tilde{p}$). We recall that $C^1_{\mathcal{C}}(S)$ and $C^2_{\mathcal{C}}(S)$ have group operations (denoted by $\circ_1$ and $\circ_2$) defined as follows:
\begin{equation}\label{Eq:grp_operation_1}
	(\bar{p}^1, \bar{c}^1_{\sigma})\circ_1(\bar{p}^2, \bar{c}^2_{\sigma})\defeq (\bar{p}^1\cdot \bar{p}^2,  \bar{c}^1_{\sigma}\cdot \act{\bar{p_1}}{\bar{c}^2_{\sigma}}),
\end{equation}
\begin{equation}\label{Eq:grp_operation_2}
	\tilde{p}_1\circ_1 \tilde{p_2} \defeq \tilde{p}_1\cdot \tilde{p_2}.
\end{equation}
We define $Z^0_{\mathcal{C}}(S)$ as follows:
\begin{equation}
Z^0_{\mathcal{C}}(S) \defeq \left\{(p_{\sigma}, c_{\sigma,\eta}) \in P(S^2) \times C(S^3), \ \ \begin{cases}
p_{\sigma \eta}=\rho(c_{\sigma,\eta})p_{\sigma}\act{\sigma}{p}_{\eta}, \\
c_{\sigma\eta,\nu}c_{\sigma,\eta}=c_{\sigma,\eta\nu}\act{p_{\sigma}\sigma}{c_{\sigma\eta,\nu}}.
\end{cases} \right\}
\end{equation}

We have a left action of $C^1_{\mathcal{C}}(S)$ on $Z^0_{\mathcal{C}}(S)$, and one of $C^2_{\mathcal{C}}(S)$ on $C^1_{\mathcal{C}}(S) \times Z^0_{\mathcal{C}}(S)$, defined as follows:
\begin{equation}\label{Eq:action_operation_0}
\star_0:C^1_{\mathcal{C}}(S) \times Z^0_{\mathcal{C}}(S) \rightarrow C^0_{\mathcal{C}}(S), \ \ (\bar{p}, \bar{c}_{\sigma}) \star_0 (p_{\sigma}, c_{\sigma,\eta}) \defeq (p^2_{\sigma}=\bar{c}_{\sigma}^{-1} \cdot \bar{p}\cdot p_{\sigma} \cdot \act{\sigma}{\bar{p}}^{-1},  \ \
 \bar{c}_{\sigma\eta}^{-1} \cdot \act{\bar{p}}{c_{\sigma,\eta}}\cdot \bar{c}_{\sigma} \cdot \act{p^2_{\sigma}\sigma}{\bar{c}_{\eta}}),
\end{equation}
\begin{equation}\label{Eq:action_operation_1}
\star_1:
C^2_{\mathcal{C}}(S) \times 
\left( C^1_{\mathcal{C}}(S) \times Z^0_{\mathcal{C}}(S) \right) 
\rightarrow C^1_{\mathcal{C}}(S) \times Z^0_{\mathcal{C}}(S), \ \  \tilde{p}\star_1\left((\bar{p}, \bar{c}_{\sigma}), (p_{\sigma}, c_{\sigma,\eta})\right)\defeq 
 \left((
\rho(\tilde{p}) \cdot \bar{p}, \ \
\tilde{p}\cdot \bar{c}_{\sigma} \cdot \act{p_{\sigma}\sigma}{\tilde{p}}^{-1}),(p_{\sigma}, c_{\sigma,\eta}) \right).
\end{equation}

\begin{proposition}\label{Prop:explicit_H0}
	Keeping the notation, $T=\Spec k$, $S=\Spec \overline{k}$, we have a natural isomorphism between $H^0(k,[{C}_H \xrightarrow{\iota_*} {C}_G])$ and the following set: 
	\begin{align*}
		(\psi_{\sigma}, u_{\sigma,\eta}, g, a) \in Z^0_{{C_H}}(S)\times C^1_{{C_G}}(S)=H(S^2) \times H^{sc}(S^3) \times G(S) \times G^{sc}(S^2), \\ \text{s.t.} \
		\begin{cases}
		\psi_{\sigma \eta}= \rho_H(u_{\sigma,\eta})\psi_{\sigma}\act{\sigma}{\psi}_{\eta}, \\
		u_{\sigma\eta,\nu}u_{\sigma,\eta}=u_{\sigma,\eta\nu}\act{\psi_{\sigma}\sigma}{u_{\sigma\eta,\nu}},\\
		g\cdot \psi_{\sigma}=\rho_G(a_{\sigma}) \cdot \act{\sigma}{g}, \\
		a_{\sigma\eta}=\act{g}{u_{\sigma,\eta}}\cdot a_{\sigma} \cdot \act{\psi_{\sigma}\sigma}{a_{\eta}}.
		\end{cases},\ \ i.e. \text{ with } \alpha= (\psi_{\sigma}, u_{\sigma,\eta})  \text{ and } \beta=(g, a_{\sigma}), \ \beta \star_0 \iota(\alpha) = e;
	\end{align*}
	quotiented by the following equivalence relation:
	\begin{align*}
    (a_1,b_1) \sim (a_2,b_2) \in Z^0_{{C_H}}(S)\times C^1_{{C_G}}(S), \ \ \text{if there exists } \\ (c,d) \in  C^1_{{C_H}}(S)\times C^2_{{C_G}}(S), \ \ s.t. \begin{cases}
    a_1=c \star_0 a_2, \\
    d\star_1 (b_2, \iota(a_2))= (\iota(c) \circ_1 b_1, \iota(a_2)).
    \end{cases} 
	\end{align*}
	Moreover, the image under the natural morphism $H^0(k,[H \xrightarrow{\iota} G]) \rightarrow H^0(k,[{C}_H \xrightarrow{\iota} {C}_G])$ of the element $(h_{\sigma},g) \in Z^0_{[1 \rightarrow H]}(S) \times C^1_{[1 \rightarrow G]}(S) = H(S^2)\times G(S)$ (where $g \cdot h_{\sigma}= \act{\sigma}{g}$) is represented by $(h_{\sigma},e,g,e)$ (where the $e$'s denote constant cocyles valued in the identity element).
\end{proposition}
\begin{proof}[Sketch of Proof]
	A more detailed proof may appear in a future version of this work or in other work. However, the proof is in essence an easy calculation from \cite[Sec. 6]{Breen} (keeping in mind that, since we are over the étale site of the spectrum of a field, the hypercoverings appearing in \cite[Sec. 6]{Breen} may always be dominated by \v{C}ech coverings by \cite[Example 9.11]{Artin-Mazur}).
\end{proof}

\bibliographystyle{alpha2}      
\bibliography{homspaces}

\end{document}